\documentclass[10pt,reqno]{amsart}

\usepackage{graphicx}
\usepackage{times}
\usepackage[colorlinks=true,linkcolor=blue,citecolor=blue]{hyperref}%
\usepackage{dsfont}

\usepackage{xcolor}
\usepackage{stmaryrd}
\usepackage{pdfrender,xcolor}
\usepackage{tikz}
\usepackage{pgfplots}
\usetikzlibrary{decorations.pathmorphing}

\newtheorem{theorem}{Theorem}[section]
\newtheorem{lemma}[theorem]{Lemma}
\newtheorem{corollary}[theorem]{Corollary}
\newtheorem{prop}[theorem]{Proposition}
\theoremstyle{definition}
\newtheorem{definition}[theorem]{Definition}
\theoremstyle{remark}
\newtheorem{remark}[theorem]{Remark}
\numberwithin{equation}{section}

\usepackage[top=1.2in, bottom=1.2in, left=1.2in, right=1.2in]{geometry}

\usepackage[scr]{rsfso}
\usepackage[english]{babel}
\usepackage[utf8x]{inputenc}
\usepackage{fancyhdr}
\usepackage{amsmath}
\usepackage{amsthm}
\usepackage{amssymb}
\usepackage{eucal}
\usepackage{comment}
\usepackage{enumitem}
\setlist{leftmargin=*}
\usepackage[integrals]{wasysym}

\newcommand\nc{\newcommand}
\nc{\on}{\operatorname}
\nc{\E}{\mathbb{E}}
\nc{\R}{\mathbb R}
\nc{\C}{\mathbb C}
\nc{\Q}{\mathbb Q}
\nc{\Z}{\mathbb Z}
\nc{\N}{\mathbb N}
\nc{\F}{\mathbb F}
\nc{\T}{\mathbb{T}}
\nc{\wt}{\widetilde}
\nc{\ol}{\overline}
\nc{\short}[3]{0 \longrightarrow #1 \longrightarrow #2 \longrightarrow #3 \longrightarrow 0}
\nc{\pd}[2]{\frac{\partial #1}{\partial #2}}
\nc{\rnc}{\renewcommand}
\nc{\e}{\varepsilon}
\nc{\DMO}{\DeclareMathOperator}
\nc{\grad}{\nabla}
\nc{\fsp}{}
\nc{\fspp}{}

\nc{\X}{X}
\rnc{\t}{t}
\nc{\x}{x}
\nc{\y}{y}
\nc{\s}{s}
\nc{\z}{z}
\nc{\w}{w}
\rnc{\r}{r}
\rnc{\a}{a}
\rnc{\b}{b}
\rnc{\k}{k}
\nc{\open}{\mathrm{open}}
\rnc{\leq}{\leqslant}
\rnc{\geq}{\geqslant}
\rnc{\d}{\mathrm{d}}
\rnc{\O}{\mathrm{O}}
\newenvironment{nouppercase}{%
  \renewcommand{\uppercasenonmath}[1]{}}{}
\title{\fsp\Large {KPZ equation from ASEP plus general speed-change drift}}

\author{Kevin Yang}

\usepackage{setspace}
\begin{document}
\setstretch{1.0}
\fsp
\raggedbottom
\begin{nouppercase}
\maketitle
\end{nouppercase}
\begin{center}
\today
\end{center}

\begin{abstract}
\fspp We derive the KPZ equation as a continuum limit of height functions in asymmetric simple exclusion processes with drift that depends on the local particle configuration. To our knowledge, it is a first such result for a class of particle systems without duality or explicit invariant measures. The tools developed in this paper consist of estimates on the corresponding Kolmogorov equations, giving a more robust proof of the Boltzmann-Gibbs principle. These tools are not exclusive to KPZ.
\end{abstract}

{\hypersetup{linkcolor=blue}
\setcounter{tocdepth}{1}
\tableofcontents}

\allowdisplaybreaks
\section{Introduction}\label{section:intro}
The \emph{Kardar-Parisi-Zhang} (KPZ) equation, introduced in \cite{KPZ}, is conjectured \cite{C11,Q} to be a universal scaling limit for stochastic growth fluctuations (e.g. burning fronts, bacterial growth, crack formation, etc.). It is written as follows, in which $\mathbb{T}:=\R/\Z$ and $\overline{\mathfrak{d}}\in\R$ is deterministic and $\xi$ is a space-time white noise:
\begin{align}
\partial_{\t}\mathbf{h}_{\t,\x}&=\tfrac12\partial_{\x}^{2}\mathbf{h}_{\t,\x}-\tfrac12|\partial_{\x}\mathbf{h}_{\t,\x}|^{2}+\overline{\mathfrak{d}}\partial_{\x}\mathbf{h}_{\t,\x}+\xi_{\t,\x}, \quad (\t,\x)\in[0,\infty)\times\mathbb{T}.\label{eq:kpz}
\end{align}
There are several equivalent solutions (e.g. regularity structures \cite{Hai13,Hai14}, paracontrolled distributions \cite{GIP,GP2}, energy solutions \cite{GJ15,GP}, renormalization group \cite{D,K}); we take the classical Cole-Hopf solution \cite{BG} in this paper, i.e. we set $\mathbf{h}:=-\log\mathbf{Z}$, where $\mathbf{Z}$ solves the \emph{stochastic heat equation} (SHE):
\begin{align}
\partial_{\t}\mathbf{Z}_{\t,\x}&=\tfrac12\partial_{\x}^{2}\mathbf{Z}_{\t,\x}+\overline{\mathfrak{d}}\partial_{\x}\mathbf{Z}_{\t,\x}-\mathbf{Z}_{\t,\x}\xi_{\t,\x}.\label{eq:she}
\end{align}
The original paper \cite{KPZ} argues for universality of \eqref{eq:kpz} via heuristics for a stochastic Hamilton-Jacobi equation model of random interfaces. These heuristics have been made rigorous in the important work \cite{HQ} and in several extensions \cite{HS,HX,KWX,KZ}. But stochastic Hamilton-Jacobi equations are (in part) simplifications of interacting particle systems, and a ``Big Picture Question" \cite{KPZAIM} on \eqref{eq:kpz} asks for a rigorous derivation of \eqref{eq:kpz} from interacting particle systems. This question is (essentially entirely) open except when the invariant measures are explicit and simple (which is a very rare situation). See Section \ref{subsection:background} for further discussion. So, deriving \eqref{eq:kpz} from interacting particle systems whose invariant measures are unknown is open yet \emph{fundamental} and \emph{basic}. We refer to Section \ref{subsubsection:hompersp} for more explanation on particle systems as a generalization of stochastic Hamilton-Jacobi equations, and the derivation of \eqref{eq:kpz} from particle systems needing techniques going beyond the reach of regularity structures.

In this paper, we derive KPZ from a large class of particle systems whose invariant measures are completely unknown. We focus on a class of \emph{driven lattice gases}, given by the asymmetric simple exclusion process (ASEP) in \cite{BG} but with additional drift in the random walks that depends only on the local particle configuration. These are the subject of a ``Big Picture Question" from \cite{KPZAIM} on the derivation of KPZ from particle systems; see \cite{KL} and Part II of \cite{S} for more motivation and history. For a clear exposition, we will assume that the drift induces only nearest-neighbor jumps; see Figure \ref{figure:model}. However, our methods should apply to models that are \emph{not} absolutely continuous with respect to ASEP. (For example, we can take drifts that induce cluster dynamics, where groups of particles move together, or finite-range jumps; see Figures \ref{figure:ex1}-\ref{figure:ex2}. Both have a lot of history in statistical mechanics \cite{MY} and KPZ universality \cite{DT,GJ17,QS,Y,KPZAIM}.)
\begin{figure}[h]
\begin{tikzpicture}[scale=1.5, 
                    squigarrow/.style={->,decorate,decoration={snake,amplitude=.5mm,segment length=2mm,post length=1mm}}]
    \draw[->] (0,0) -- (8,0);

    \foreach \x in {0,1,2,...,7} {
        \draw (\x,0.1) -- (\x,-0.1);
    }

    \foreach \x in {1,2,5} {
        \fill[blue] (\x,0.2) circle (0.05);
    }
    
    \draw[->, blue] (1.1,0.2) -- (1.9,0.2);
    
    \draw[red, thick] (1.3,0.35) -- (1.7,0.05);
    \draw[red, thick] (1.3,0.05) -- (1.7,0.35);
    
    \draw[->, blue] (5.1,0.2) -- (5.9,0.2);
    
    \draw[->, blue] (4.9,0.2) -- (4.1,0.2);
    
    \draw[squigarrow, teal, bend right=45] (5,0.2) to (4,0.2);
    \draw[squigarrow, magenta, bend left=45] (5,0.2) to (6,0.2);
    
\end{tikzpicture}
\caption{The ASEP as in \cite{BG} with symmetric speed {\small$\frac12N^{2}$}, asymmetric speed {\small$\frac12N^{3/2}$} (biased to the left), and an additional asymmetry (i.e. ``drift") with speed of order $N^{\alpha}$ depicted by $\rightsquigarrow$. The speed of the left- and right-squiggly arrows are different, hence the different coloring and hence a drift. The exact drift speed depends only on the local particle configuration.}\label{figure:model}
\end{figure}
\begin{figure}[h]
\begin{tikzpicture}[scale=1.5, 
                    squigarrow/.style={->,decorate,decoration={snake,amplitude=.5mm,segment length=2mm,post length=1mm}}]
    \draw[->] (0,0) -- (8,0);

    \foreach \x in {0,1,2,...,7} {
        \draw (\x,0.1) -- (\x,-0.1);
    }

    \foreach \x in {1,2,5} {
        \fill[blue] (\x,0.2) circle (0.05);
    }
    
    \draw[->, blue] (1.1,0.2) -- (1.9,0.2);
    
    \draw[red, thick] (1.3,0.35) -- (1.7,0.05);
    \draw[red, thick] (1.3,0.05) -- (1.7,0.35);
    
    \draw[->, blue] (5.1,0.2) -- (5.9,0.2);
    
    \draw[->, blue] (4.9,0.2) -- (4.1,0.2);
    
    \draw[squigarrow, teal, bend right=45] (5,0.2) to (3,0.2);
    \draw[squigarrow, magenta, bend left=45] (5,0.2) to (7,0.2);
    
\end{tikzpicture}
\caption{The model as in Figure \ref{figure:model}, but the drift induces jumps of length $k>1$.}\label{figure:ex1}
\end{figure}
\begin{figure}[h]
\begin{tikzpicture}[scale=1.5, 
                    squigarrow/.style={->,decorate,decoration={snake,amplitude=.5mm,segment length=2mm,post length=1mm}}]
    \draw[->] (0,0) -- (8,0);

    \foreach \x in {0,1,2,...,7} {
        \draw (\x,0.1) -- (\x,-0.1);
    }

    \foreach \x in {1,2,5,6} {
        \fill[blue] (\x,0.2) circle (0.05);
    }
    
    \draw[->, blue] (1.1,0.2) -- (1.9,0.2);
    
    \draw[red, thick] (1.3,0.35) -- (1.7,0.05);
    \draw[red, thick] (1.3,0.05) -- (1.7,0.35);
    
    \draw[squigarrow, teal, bend right=45] (5,0.2) to (3,0.2);
    \draw[squigarrow, teal, bend right=45] (6,0.2) to (4,0.2);
    
    \draw[thick, red, rounded corners] (4.9,0.3) rectangle (6.1,0.1);
    
\end{tikzpicture}
\caption{An example of ``cluster dynamics". For simplicity, we did not draw ASEP-moves for particles in the red box (even though they are present). The additional drift moves particles in the red box together to the left at a speed depending on the local particle configuration.}\label{figure:ex2}
\end{figure}

Let us now describe precisely the class of particle systems we will analyze. (A formulation via infinitesimal generator is given in Section \ref{section:model}.) It turns out to be convenient to introduce these as the following Markovian spin dynamics on the state space $\{\pm1\}^{\mathbb{T}_{N}}$, where $\mathbb{T}_{N}=\Z/N\Z\simeq\{0,\ldots,N-1\}$ is a discrete torus of size $N\gg1$. 
\begin{enumerate}
\item First, throughout this paper, we will denote elements in $\{\pm1\}^{\mathbb{T}_{N}}$ by {\small$\eta=(\eta_{\x})_{\x\in\mathbb{T}_{N}}$}, where {\small$\eta_{\x}\in\{\pm1\}$} for all $\x$. Next, let $\mathfrak{d}:\{\pm1\}^{\mathbb{T}_{N}}\to\R$ be any fixed uniformly bounded \emph{local} ``driving" function, i.e. there is a neighborhood of $0\in\mathbb{T}_{N}$ of size independent of $N$ such that $\mathfrak{d}[\eta]$ depends only on $\eta_{\x}\in\{\pm1\}$ for $\x$ in said neighborhood for any $\eta\in\{\pm1\}^{\mathbb{T}_{N}}$. Also, for $\z\in\Z$, let {\small$\tau_{\x}\eta=(\eta_{\x+\z})_{\x}$} be the configuration $\eta$ shifted by $\z$.
\item Fix any $\x\in\mathbb{T}_{N}$. Suppose the configuration at $\{\x,\x+1\}$ is given by $\{+1,-1\}$, i.e. $\eta_{\x}=1$ and $\eta_{\x+1}=-1$. We swap spins $\{+1,-1\}\mapsto\{-1,+1\}$ via Poisson clock of speed {\small$\frac12N^{2}-\frac12N^{3/2}-\frac12N^{\alpha}\mathfrak{d}[\tau_{\x}\eta]$}. If $\eta_{\x}=-1$ and $\eta_{\x+1}=1$, then the speed of swapping $\{-1,+1\}\mapsto\{+1,-1\}$ is {\small$\frac12N^{2}+\frac12N^{3/2}+\frac12N^{\alpha}\mathfrak{d}[\tau_{\x}\eta]$}.
\item Let {\small$\t\mapsto\eta_{\t}$} denote the configuration evolving according to these dynamics. We define the associated \emph{height function} $\mathbf{h}^{N}:[0,\infty)\times\mathbb{T}_{N}\to\R$ by the rules {\small$\mathbf{h}^{N}_{0,0}=0$} and {\small$\mathbf{h}^{N}_{\t,\x}-\mathbf{h}^{N}_{\t,\x-1}=N^{-1/2}\eta_{\t,\x}$}. A more explicit formula is given in Section \ref{section:model}. We assume that there are $N/2$-many particles on $\mathbb{T}_{N}$ (and thus that $N$ is even) to make sure that $\mathbf{h}^{N}$ is a periodic function on $\mathbb{T}_{N}$.
\end{enumerate}
The prediction \cite{HQ,HS,HX,KWX,KZ} is that $\mathbf{h}^{N}\to\mathrm{KPZ}$ as $N\to\infty$, and that the trichotomy below holds.
\begin{enumerate}
\item If $\alpha\in[\frac12,1)$, then $\mathfrak{d}[\cdot]$ contributes to the average growth-speed of $\mathbf{h}^{N}$ (i.e. the renormalization term \eqref{eq:renorm}).
\item If $\alpha\in[1,\frac32)$, then $\mathfrak{d}[\cdot]$ also affects the coefficient $\overline{\mathfrak{d}}$ in \eqref{eq:kpz}.
\item If $\alpha=\frac32$, then $\mathfrak{d}[\cdot]$ also affects the quadratic coefficient in \eqref{eq:kpz}. 
\end{enumerate}
{\emph{In this paper, we show that for $\alpha=1$, the prediction above is true}}. (Since transport operators are asymmetric objects in the PDE world, the scaling exponent $\alpha=1$ is perhaps the first natural choice to pursue.) A precise statement of our result can be found in Theorem \ref{theorem:main}, and to our knowledge, it is the \emph{first} result for \emph{general} $\mathfrak{d}[\cdot]$. 
\subsection{Background}\label{subsection:background}
If $\mathfrak{d}[\cdot]\equiv0$, then Theorem \ref{theorem:main} is the work \cite{BG} of Bertini-Giacomin. In this case, $\exp\{-\mathbf{h}^{N}\}$ satisfies an on-the-nose discretization of \eqref{eq:she}, which dramatically simplifies the proof. This property is true for a small handful of particle systems as well, whose KPZ equation limits were established in \cite{CGST,CST,CT,CTIn}. 

For driving functions {\small$\mathfrak{d}_{\mathrm{gc}}$} satisfying {\small$\mathfrak{d}_{\mathrm{gc}}[\eta](\eta_{1}-\eta_{0})=\mathfrak{w}[\tau_{1}\eta]-\mathfrak{w}[\eta]$} for a local function $\mathfrak{w}$, i.e. the ``gradient condition", Theorem \ref{theorem:main} was shown in \cite{GJ15,GJS15,Y23}. This assumption gives that i.i.d. product Bernoulli measures on $\{\pm1\}^{\mathbb{T}_{N}}$ are invariant measures for {\small$\eta_{\t}$}; this is key to \cite{GJ15,GJS15,Y23}, and nothing was known without it. But the gradient condition is very restrictive, and there does not seem to be any \emph{natural} reason to want to assume it. (We also re-emphasize that the lack of explicit invariant measures is exactly our main technical challenge.)
\subsubsection{Homogenization and the stochastic-analytic perspective}\label{subsubsection:hompersp}
Although product Bernoulli measures are not invariant for general driving functions $\mathfrak{d}[\cdot]$, an inspection of Theorem \ref{theorem:main} shows that the coefficients in the limit KPZ equation are still computed via product Bernoulli measures. This is surprising, as it is usually the invariant measure that determines contributions of local functions in the large-$N$ limit in a particle system \cite{GPV}. \emph{One way to make sense of this heuristically starts with the following lattice stochastic differential equation}:
\begin{align}
\d\mathbf{h}^{N}_{\t,\x}&=\tfrac12N^{2}\Delta\mathbf{h}^{N}_{\t,\x}\d\t+\tfrac12N^{\frac32}\grad_{-}\mathbf{h}^{N}_{\t,\x}\grad_{+}\mathbf{h}^{N}_{\t,\x}\d\t+N^{\alpha-\frac12}\mathfrak{y}[\tau_{\x}\eta_{\t}]\d\t+\d\mathbf{b}^{N}_{\t,\x}. \label{eq:hfhj}
\end{align}
Above, $\Delta$ is discrete Laplacian on $\mathbb{T}_{N}$ and $\nabla_{\pm}$ is discrete gradient in the forwards direction (for $+$) or backwards direction (for $-$). The term $\mathbf{b}^{N}$ is a martingale in $\t$ for every $\x$ (i.e. a noise term), and $\mathfrak{y}$ is an explicit function of the particle system depending only on $\mathfrak{d}$. Since {\small$\tau_{\x}\eta_{\t}$} is a discrete gradient of $\mathbf{h}^{N}$, \eqref{eq:hfhj} is a lattice stochastic Hamilton-Jacobi equation, and {\small$\mathfrak{y}[\tau_{\x}\eta_{\t}]$} is a ``first-order" term. Because lower-order terms generally do not affect the local behavior in a parabolic equation, it is perhaps plausible that {\small$\mathfrak{y}[\tau_{\x}\eta_{\t}]$} does not affect the large-$N$ behavior of local functions. This is basically the goal of regularity structures in \cite{Hai13,Hai14,HQ,HS,HX,KWX,KZ}, which can be seen as constructing a ``good coupling" between the equation and its linear part. 

However, applying regularity structures and similar couplings to remove the $\mathfrak{d}[\cdot]$-drift in the context of particle systems is an outstanding open problem \cite{KPZAIM}. One difficulty is that $\mathbf{b}^{N}$ is a nonlinear function of the gradient of $\mathbf{h}^{N}$; it knows if a jump is allowed because of exclusion. \emph{Thus, instead of using regularity structures, we develop a new method that works for nonlinear noise}.

Lastly, let us note that the nonlinear noise is \emph{far from being a technical issue}. Indeed, it turns out that $\mathbf{b}^{N}\approx\xi$ in a weak sense after averaging in space and time. This gives a map from particle systems to Hamilton-Jacobi equations defined by replacing {\small$\mathbf{b}^{N}\rightsquigarrow\xi$}, and \cite{HQ,HS,HX,KWX,KZ} tell us how to go from Hamilton-Jacobi to KPZ. It is then natural to ask if the following diagram commutes:
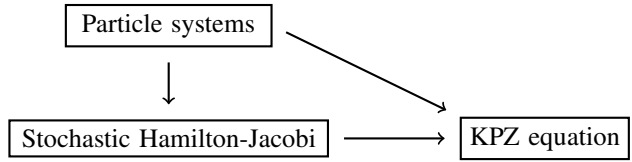
\vspace{5pt}
\begin{center}
\begin{tikzpicture}
\node[draw, rectangle, thick, text width=2.5cm, align=center] (box1) at (0,0) {Particle systems};

\node[draw, rectangle, thick, text width=4cm, align=center] (box2) at (0,-1.5) {Stochastic Hamilton-Jacobi};

\node[draw, rectangle, thick, text width=2cm, align=center] (box3) at (5,-1.5) {KPZ equation};

\draw[->, thick, shorten >=2mm, shorten <=2mm] (box1.south) -- (box2.north); 
\draw[->, thick, shorten >=2mm, shorten <=2mm] (box2.east) -- (box3.west);  
\draw[->, thick, shorten >=2mm, shorten <=2mm, bend left=15] (box1.east) -- (box3.north west); 
\end{tikzpicture}
\end{center}
\vspace{5pt}
It turns out that this \emph{cannot} commute! Indeed, the coefficients in the limiting KPZ equation \eqref{eq:kpz}, when derived from Hamilton-Jacobi, are Gaussian expectations \cite{HQ,HX}, whereas the coefficients in Theorem \ref{theorem:main} are Bernoulli expectations. In particular, the (model-dependent) invariant measures are \emph{crucial}.
\subsubsection{The methods in this paper}\label{subsubsection:papermethod}
To prove homogenization of {\small$\mathfrak{y}[\tau_{\x}\eta_{\t}]$} in \eqref{eq:hfhj} (and therefore compute the constants in the limiting KPZ equation), we develop a new method that consists of three major steps.
\begin{enumerate}
\item To replace the contribution of $\mathfrak{y}[\cdot]$ by a first derivative, we crucially observe that it is enough to average over \emph{local} space-time scales. Thus, we reduce to particle dynamics localized in space to a ``small torus".
\item Show that the product Bernoulli measures are ``stable" or ``almost invariant" for the aforementioned \emph{localized} dynamics. This ``almost-invariance" is proved via estimates on the Kolmogorov forward equation associated to the localized dynamics; to accomplish this, we use Davies' method for $\mathrm{L}^{p}$-bounds on parabolic equations, whose application to KPZ and non-reversible particle systems seems new.
\item Use these $\mathrm{L}^{p}$-estimates to prove a \emph{Kipnis-Varadhan inequality} and \emph{Boltzmann-Gibbs principle} for the localized system. These are powerful bounds on space-time averages of local functions \cite{KL}. The classical proofs use invariant measures crucially, while the proofs in this paper seem to be more robust in this sense. 
\end{enumerate}
This is perhaps our main technical selling point, especially given a surge of interest in perturbations of exclusion processes \cite{F,FMST1,FMST2,JL,JM}. Finally, because of \eqref{eq:hfhj}, our work extends \cite{HQ,HS,HX,KWX,KZ} to nonlinear noises. [Although not KPZ-related, we mention \cite{JL,JM}, which derive hydrodynamic scaling limits and their fluctuation theory from perturbations of the \emph{symmetric} simple exclusion process. Both works require either smoothness of the drift (i.e. replacing $\mathfrak{d}[\tau_{\x}\eta]$ by a deterministic smooth function of $\x$) or smallness of the perturbation (e.g. of order $\log N$, thus very far from $N^{\alpha}$). Moreover, \cite{JL,JM} work in analytically weak topologies (while analyzing \eqref{eq:kpz}-\eqref{eq:she} require estimates in fairly strong $\mathscr{C}^{\upsilon}$-topologies). In particular, this work closes a significant technical gap in what available methods were able to achieve.]
\subsection{The rest of the predicted trichotomy}\label{subsection:scaling}
Our result Theorem \ref{theorem:main} establishes, for the first time, a part of the trichotomy mentioned earlier. However, we believe that it is possible to extend the methods in this paper to cover all $\alpha<5/4$ in the trichotomy; we give the key details in Remark \ref{remark:kvim}, and we refrain from a complete proof in order to give the clearest possible presentation of our work.

At $\alpha=5/4$, however, there appears to be a \emph{phase transition}. In particular, it is not even clear if the renormalization constant $\mathscr{R}_{N}$ from \eqref{eq:renorm} can be computed using only product Bernoulli expectations. (We note that the renormalization constant $\mathscr{R}_{N}$, which gives the global growth speed of $\mathbf{h}^{N}$, gets a contribution from the large-$N$ limit of $\mathfrak{y}[\cdot]$ in \eqref{eq:hfhj}; see \eqref{eq:renorm}.) Indeed:
\begin{itemize}
\item For any $\alpha$, we must compute the large-$N$ asymptotics of some local function of order $N^{\alpha-1/2}$; see \eqref{eq:hfhj}.
\item The driving function $N^{\alpha}\mathfrak{d}[\cdot]$ induces a perturbation in the generator (and invariant measures) of relative size $\mathrm{O}(N^{\alpha-2})$. Thus, when we compute the limiting contribution of a local function of order $N^{\alpha-1/2}$, we expect in general a total change of $\mathrm{O}(N^{\alpha-1/2}N^{\alpha-2})=\mathrm{O}(N^{2\alpha-5/2})$. This persists as soon as $\alpha\geq5/4$.
\end{itemize}
The above suggests that homogenization of a general local function of order $N^{\alpha-1/2}$ cannot be computed only with product Bernoulli measures if $\alpha\geq5/4$. (For example, the {\small$\mathrm{ASEP}(\mathrm{q},\mathrm{j})$} and {\small$\mathrm{ASIP}(\mathrm{q},\mathrm{k})$} particle systems in \cite{CGRS} have a similar tunable drift strength of {\small$\approx N^{2}\mathrm{q}$}, where {\small$\mathrm{q}=\mathrm{O}(N^{\alpha-2})$}; if one puts these particle systems on the lattice {\small$\{1,\ldots,N\}$} without periodic boundary conditions, then the invariant measures are exactly inhomogeneous product Bernoulli measures, and the Bernoulli parameter depends on space and differs between adjacent points by order {\small$\mathrm{q}$}. Because we only need the local dynamics as discussed in Section \ref{subsubsection:papermethod}, the lack of periodic boundary conditions is not so much the issue. Thus, this confirms the second bullet point above, albeit for a different model but one that also belongs to the KPZ universality class, at least in the case of {\small$\mathrm{ASEP}(\mathrm{q},\mathrm{j})$} \cite{CST}.) At the technical level, the main issue for $\alpha\geq5/4$ is that steps (1)-(3) in Section \ref{subsubsection:papermethod} work for all $\alpha\leq3/2$. (So, our methods are perhaps still useful for deriving \eqref{eq:kpz} for $\alpha\leq3/2$; see Remark \ref{remark:kvim}.) However, they are ``incompatible", i.e.:
\begin{itemize}
\item Step (1) in Section \ref{subsubsection:papermethod} lets us reduce to local tori of length scales that are bounded below by some $\ell_{1}$.
\item Steps (2)-(3) in Section \ref{subsubsection:papermethod} succeed for local tori of length scales that are bounded above by some $\ell_{2}$.
\item When $\alpha\geq5/4$, we can only choose $\ell_{1},\ell_{2}$ so that $\ell_{2}\ll\ell_{1}$, i.e. there is no local tori that works for all steps.
\end{itemize}
\subsection{Acknowledgements}
The author is supported in part by the NSF under Grant No. DMS-2203075.
%
%
%
\section{The model and main results}\label{section:model}
We now present {\small$\eta^{}_{\t}$} as a Markov process on $\{\pm1\}^{\mathbb{T}_{N}}$; this will both give a precise construction of the particle system and be important in our analysis of it later on. The construction is equivalent to the random walk description from the introduction (for $\alpha=1$ therein, which is the setting we specialize to for the rest of this paper). The generator of this Markov process is given by {\small$\mathscr{L}^{}_{N}:=\mathscr{L}_{N,\mathrm{S}}+\mathscr{L}^{}_{N,\mathrm{A}}$}, where
\begin{align}
\mathscr{L}_{N,\mathrm{S}}&:=\tfrac12N^{2}\sum_{\x,\x+1\in\mathbb{T}_{N}}\mathscr{L}_{\x}\nonumber\\
\mathscr{L}^{}_{N,\mathrm{A}}&:=\tfrac12N^{\frac32}\sum_{\x,\x+1\in\mathbb{T}_{N}}\left(\mathbf{1}_{\substack{\eta_{\x}=-1\\\eta_{\x+1}=1}}-\mathbf{1}_{\substack{\eta_{\x}=1\\\eta_{\x+1}=-1}}\right)\left(1+N^{-\frac12}\mathfrak{d}[\tau_{\x}\eta]\right)\mathscr{L}_{\x},\label{eq:generator}
\end{align}
where $\mathscr{L}_{\x}$ is the generator for speed $1$ symmetric simple exclusion on the bond $\{\x,\x+1\}$ and $\mathfrak{d}:\{\pm1\}^{\mathbb{T}_{N}}\to\R$ is a fixed local function. Precisely, for any function $\mathfrak{f}:\{\pm1\}^{\mathbb{T}_{N}}\to\R$, we have $\mathscr{L}_{\x}\mathfrak{f}[\eta]=\mathfrak{f}[\eta^{\x,\x+1}]-\mathfrak{f}[\eta]$, where {\small$\eta^{\x,\x+1}_{\z}=\eta_{\z}$} if $\z\neq\x,\x+1$, and where {\small$\eta^{\x,\x+1}_{\x}=\eta_{\x+1}$} and {\small$\eta^{\x,\x+1}_{\x+1}=\eta_{\x}$}. (In words, {\small$\eta^{\x,\x+1}$} is the configuration $\eta$ by swapping $\eta_{\x},\eta_{\x+1}$.)

Let us give a more explicit formula for the height function $\mathbf{h}^{N}$. For any $\t\geq0$, let {\small$\mathbf{h}^{N}_{\t,0}$} be {\small$2N^{-1/2}$} times the net flux of particles across $0$ by time $\t$, i.e. the number of particles that have gone from $0\mapsto N-1$ minus the number of particles that have gone from $N-1\mapsto0$. Next, given $(\t,\x)\in[0,\infty)\times\mathbb{T}_{N}$, in which we identify $\mathbb{T}_{N}:=\{1,\ldots,N\}$ (to be explained immediately below \eqref{eq:hf}), we define
\begin{align}
\mathbf{h}^{N}_{\t,\x}=\mathbf{h}^{N}_{\t,0}+N^{-\frac12}\sum_{\y=1}^{\x}\eta_{\t,\y}.\label{eq:hf}
\end{align}
(We assume throughout that there are $N/2$ particles on $\mathbb{T}_{N}$, so that {\small$\mathbf{h}^{N}_{\t,0}=\mathbf{h}^{N}_{\t,N}$}, i.e. it does not matter if we identify $\mathbb{T}_{N}\simeq\{1,\ldots,N\}$ or $\mathbb{T}_{N}\simeq\{0,\ldots,N-1\}$ in the above.) We also define a microscopic analog of \eqref{eq:she} (often called the \emph{Gartner transform} \cite{G}):
\begin{align}
\mathbf{Z}^{N}_{\t,\x}:=\exp\{-\mathbf{h}^{N}_{\t,\x}+\mathscr{R}_{N}\t\}, \quad(\t,\x)\in[0,\infty)\times\mathbb{T}_{N}.\label{eq:gartner}
\end{align}
The term {\small$\mathscr{R}_{N}:=\frac12N-\frac{1}{24}+\mathscr{R}_{2}$} is the \emph{renormalization constant}, where {\small$\mathscr{R}_{2}=N^{1/2}\mathscr{R}_{2,1}+\mathscr{R}_{2,2}+\mathscr{R}_{2,3}$} with each term defined as follows. Let $\E^{\sigma}$ be expectation with respect to product measure on {\small$\{\pm1\}^{\mathbb{T}_{N}}$} such that {\small$\E^{\sigma}\eta_{\x}=\sigma$} for all $\x$. We first set {\small$\mathscr{R}_{2,1}:=\frac12\E^{0}\{\mathfrak{d}[\eta](1-\eta_{0}\eta_{1})\}$}; the quantity {\small$\mathfrak{d}[\eta](1-\eta_{0}\eta_{1})$} is often called the \emph{flux} associated to $\mathfrak{d}$. Next, we define {\small$\mathscr{R}_{2,2}:=-\overline{\mathfrak{d}}/2$}, where
\begin{align}
\overline{\mathfrak{d}}:=\tfrac12\partial_{\sigma}\E^{\sigma}\{\mathfrak{d}[\eta](1-\eta_{0}\eta_{1})\}|_{\sigma=0}.\label{eq:dcoeff}
\end{align}
The expectation on the RHS of \eqref{eq:dcoeff} is a function of $\sigma$; we differentiate in $\sigma$ and then set $\sigma=0$. Next, we define a more technical renormalization (akin to the $\frac{1}{24}$ in \cite{BG}) below, in which we use the notation $\wt{\mathfrak{s}}$ that is defined in \eqref{eq:wtsterm}:
\begin{align*}
\mathscr{R}_{2,3}:=\E^{0}\wt{\mathfrak{s}}-\tfrac12\E^{0}\left\{\mathfrak{d}[\eta][\eta_{1}-\eta_{0}]\right\}.
\end{align*}
Ultimately, the renormalization constant is equal to
\begin{align}
&\mathscr{R}_{N}=\tfrac12N-\tfrac{1}{24}+N^{\frac12}\mathscr{R}_{2,1}+\mathscr{R}_{2,2}+\mathscr{R}_{2,3}\label{eq:renorm}\\
&=\tfrac12N-\tfrac{1}{24}+\tfrac12N^{\frac12}\E^{0}\{\mathfrak{d}[\eta](1-\eta_{0}\eta_{1})\}-\tfrac12\overline{\mathfrak{d}}-\tfrac12\E^{0}\left\{\mathfrak{d}[\tau_{-2\mathfrak{l}_{\mathfrak{d}}}\eta](1-\eta_{-2\mathfrak{l}_{\mathfrak{d}}}\eta_{-2\mathfrak{l}_{\mathfrak{d}}+1})\right\}-\tfrac12\E^{0}\left\{\mathfrak{d}[\eta][\eta_{1}-\eta_{0}]\right\}.\nonumber
\end{align}
We introduce a last easily removable assumption (as explained below); it makes computations more efficient:
\begin{align}
\partial_{\sigma}^{2}\E^{\sigma}\{\mathfrak{d}[\eta](1-\eta_{0}\eta_{1})\}|_{\sigma=0}=0.\label{eq:assumequad}
\end{align}
To remove the assumption \eqref{eq:assumequad}, we can subtract from $\mathfrak{d}$ an appropriate constant, as {\small$\partial_{\sigma}^{2}\E^{\sigma}\eta_{0}\eta_{1}=\partial_{\sigma}^{2}\sigma^{2}=2$}. In particular, the constant part of the asymmetry in the particle system becomes $N^{3/2}+\zeta N=N^{3/2}(1+\zeta N^{-1/2})$ for a deterministic $\zeta\in\R$. Thus, we are left with the same model but with a constant speed asymmetry perturbed by $\mathrm{o}(1)$. But similar to \cite{BG} (and essentially every other derivation of \eqref{eq:kpz}) our proof of Theorem \ref{theorem:main} extends to any constant speed asymmetry (and is continuous with respect to this asymmetry speed). It is only for clarity of presentation that we do not keep track of this additional $\zeta N^{-1/2}$-perturbation in the constant-speed part of the asymmetry. (In any case, \eqref{eq:assumequad} and the additional $\zeta N^{-1/2}$ perturbation do not change whether we have explicit invariant measures, so this is besides the main point of this paper.)
\begin{theorem}\label{theorem:main}
\fsp Suppose the function {\small$\mathrm{X}\mapsto\mathbf{h}^{N}_{0,N\mathrm{X}}$}, if we extend it from $N^{-1}\mathbb{T}_{N}$ to $\mathbb{T}$ by linear interpolation, is deterministic and converges to some limit {\small$\mathbf{h}_{0,\cdot}\in\mathscr{C}^{\upsilon}(\mathbb{T})$} in the {\small$\mathscr{C}^{\upsilon}(\mathbb{T})$}-topology for all {\small$\upsilon<1/2$}. Then the function {\small$(\t,\mathrm{X})\mapsto\mathbf{Z}^{N}_{\t,N\X}-\mathscr{R}_{N}\t$}, if we extend it from $N^{-1}\mathbb{T}_{N}$ to $\mathbb{T}$ by linear interpolation, converges to the solution of \eqref{eq:she} with initial data {\small$\exp\{-\mathbf{h}_{0,\cdot}\}$}. The convergence is in law in the space {\small$\mathscr{D}([0,1],\mathscr{C}(\mathbb{T}))$}.
\end{theorem}
The terminal time of $1$ is unimportant; any finite, deterministic time is okay. The deterministic assumption on the initial data is not important either; we can always condition on the initial data. 
\subsection{Outline of the paper}\label{subsection:outline}
For a detailed outline and an overview of the main technical steps, see Section \ref{section:outline}.
\begin{itemize}
\item In Section \ref{section:notation}, we give a list of notation for anything that is used anywhere beyond where it is introduced.
\item In Sections \ref{section:mainproof}, \ref{section:bgp}, and \ref{section:propproofs}, we reduce the proof of Theorem \ref{theorem:main} to a key estimate (the \emph{Boltzmann-Gibbs principle} in Theorem \ref{theorem:bgp}) and a weaker but similar estimate (see Proposition \ref{prop:hl}).
\item In Section \ref{section:tools}, we give the tools to prove the Boltzmann-Gibbs principle. It is the technical heart of this paper.
\item In Section \ref{section:bgphlproof}, we prove the remaining key inputs (Theorem \ref{theorem:bgp} and Proposition \ref{prop:hl}). 
\end{itemize}
\section{Notation}\label{section:notation}
\begin{itemize}
\item We use standard big-O notation. We write $a=\mathrm{O}(b)$ if $|a|\leq C|b|$ for a uniformly bounded $C>0$. We write $a\lesssim b$ to mean $a=\mathrm{O}(b)$ and $a\gtrsim b$ to mean $b\lesssim a$. We write $a=\mathrm{o}(b)$ to mean $|a|\leq c_{N}|b|$, where $c_{N}\to0$.
\item For any function $\varphi:\mathbb{T}_{N}\to\R$ and any $\mathfrak{l}\in\Z$, we define the discrete spatial gradient of length-scale $\mathfrak{l}$ to be $\grad^{\mathbf{X}}_{\mathfrak{l}}\varphi_{\x}:=\varphi_{\x+\mathfrak{l}}-\varphi_{\x}$. We also define the discrete Laplacian to be $\Delta:=\grad^{\mathbf{X}}_{1}\grad^{\mathbf{X}}_{-1}$, and next, we define the operator $\mathscr{T}_{N}:=\frac12N^{2}\Delta{{}-}\overline{\mathfrak{d}}N\grad_{-1}^{\mathbf{X}}$, where $\overline{\mathfrak{d}}$ is from \eqref{eq:dcoeff}.
\item For function $\psi:[0,1]\to\R$ and $\s\in\R$, we let the time-gradient of time-scale $\s$ be {\small$\grad^{\mathbf{T}}_{\s}\psi_{\t}:=\psi_{[1\wedge(\t+\s)]\vee0}-\psi_{\t}$}. The use of  $[1\wedge(\t+\s)]\vee0$ is meant to keep the time-variable inside the interval $[0,1]$.
\item Here, we record important length-scales and exponents used throughout the paper. First, $\e_{\mathrm{ap}},\e_{\mathrm{reg}}>0$ are small and fixed, and they satisfy $\e_{\mathrm{reg}}>C\e_{\mathrm{ap}}$ for some large constant $C=\mathrm{O}(1)$. We also define the length-scale {\small$\mathfrak{l}_{\mathrm{reg}}:=N^{1/3+\e_{\mathrm{reg}}}$}.
\end{itemize}
%
%
%
\section{Outline of the main arguments}\label{section:outline}
We now give a heuristic overview of the arguments needed to prove Theorem \ref{theorem:main}. The high-level approach is similar to that of \cite{BG,DT,Y23} and many other works. In particular, the first step is computing a stochastic evolution equation for $\mathbf{Z}^{N}$; this is done in Lemma \ref{lemma:mshe} below. This produces a microscopic version of the limit SPDE \eqref{eq:she}. However, in this microscopic SHE, there are additional error terms whose treatment is very difficult and requires very new ideas and tools. Treating said error terms is the focus for almost all of this paper. (While the high-level approach is similar to that of \cite{Y23}, many important and difficult technical refinements of the analysis done in that paper, e.g. of averaging scales, are needed here, ultimately as a result of the lack of explicit invariant measures.)

The main error term is given by {\small$(\t,\x)\mapsto N^{1/2}\overline{\mathfrak{q}}[\tau_{\x}\eta_{\t}]\mathbf{Z}^{N}_{\t,\x}$}, where $\overline{\mathfrak{q}}$ is given in \eqref{eq:overlineqterm}. (The key properties of $\overline{\mathfrak{q}}$ will be explained shortly.) We explain how this error term is treated below.

The particle system wants to replace $\overline{\mathfrak{q}}$ by an ``ergodically-averaged" version since it is evolving very quickly at local scales. In particular, for every length-scale $\mathfrak{l}$, there is a ``homogenization" $\mathsf{E}[\mathfrak{l}]$ of $\overline{\mathfrak{q}}$; it is defined precisely in \eqref{eq:canexp}. \emph{The key property of $\overline{\mathfrak{q}}$ and $\mathsf{E}[\mathfrak{l}]$ is $|\mathsf{E}[\mathfrak{l}]|\lesssim\mathfrak{l}^{-3/2}$} in some sense, at least up to small powers of $N$. (Proving an estimate of this type is nontrivial, but the goal of this section is to give heuristics.) This property depends on \eqref{eq:assumequad}. Thus, we want to replace $\overline{\mathfrak{q}}\rightsquigarrow\mathsf{E}[\mathfrak{l}]$ with $\mathfrak{l}\gg N^{1/3}$ to beat the $N^{1/2}$-factor in {\small$(\t,\x)\mapsto N^{1/2}\overline{\mathfrak{q}}[\tau_{\x}\eta_{\t}]\mathbf{Z}^{N}_{\t,\x}$}.

Following \cite{CYau,GJ15,GJS15,SX}, we will instead consider a sequence of length-scales $\ell_{\k+1}=\ell_{\k}N^{\e}$ for small $\e>0$, and we control the terms {\small$(\t,\x)\mapsto N^{1/2}\mathsf{R}[\k\mapsto\k+1]\mathbf{Z}^{N}_{\t,\x}$}, where $\mathsf{R}[\k\mapsto\k+1]=\mathsf{E}[\ell_{\k+1}]-\mathsf{E}[\ell_{\k}]$. Because we integrate $\mathsf{R}[\k\mapsto\k+1]$ in space-time against the heat kernel and $\mathbf{Z}^{N}$, both of which have ``nice" space-time regularity, we can replace $\mathsf{R}[\k\mapsto\k+1]$ by its space-time average on a small length-scale $\ell_{\mathrm{av}}$ and time-scale  $\tau_{\mathrm{av}}$. (Technically, the regularity of $\mathbf{Z}^{N}$ is not good enough to justify replacing $\mathsf{R}[\k\mapsto\k+1]$ by a space-average, but space-gradients of $\mathbf{Z}^{N}$ are explicit functions of the particle system, so we are just left with lower-order errors.)

We must now bound the space-time average of $N^{1/2}\mathsf{R}[\k\mapsto\k+1]$ with length-scale $\ell_{\mathrm{av}}$ and time-scale $\tau_{\mathrm{av}}$. \emph{One of our main technical contributions is bounding said space-time average by {\small$\lesssim N^{1/2}(N^{2}\tau_{\mathrm{av}})^{-1/2}\ell_{\mathrm{av}}^{-1/2}$}} in a sufficiently strong sense \emph{for all $\k$}. \emph{This is where a lack of explicit invariant measures poses the biggest technical challenge, and this is what the method described Section \ref{subsubsection:papermethod} is ultimately developed to establish.} (Indeed, the usual proof of this bound, as in \cite{GJ15,GJS15,Y23}, uses invariant measures in a crucial way via the \emph{Kipnis-Varadhan inequality}.) In any case, we can take $\tau_{\mathrm{av}}\gg N^{-4/3}$ and $\ell_{\mathrm{av}}\gg N^{1/3}$, which are well below macroscopic scales, so that our bound becomes {\small$N^{1/2}(N^{2}\tau_{\mathrm{av}})^{-1/2}\ell_{\mathrm{av}}^{-1/2}\ll N^{-1/2}N^{2/3}N^{-1/6}=1$} and is therefore small.

This is the vague idea in proving Theorem \ref{theorem:main}. There are other important details, however, to take care of. For example, we need to make sure that our estimates for the error {\small$(\t,\x)\mapsto N^{1/2}\overline{\mathfrak{q}}[\tau_{\x}\eta_{\t}]\mathbf{Z}^{N}_{\t,\x}$} are in a sufficiently strong topology to be able to rigorously forget it when computing the large-$N$ limit of $\mathbf{Z}^{N}$. We also need strong enough space-time regularity for $\mathbf{Z}^{N}$. Sections \ref{section:mainproof}, \ref{section:bgp}, and \ref{section:propproofs} address these points. Next, our square-root cancellation bound for $(\ell_{\mathrm{av}},\tau_{\mathrm{av}})$-scale averages also requires knowing that the distribution of the particle system, when localized to such scales, is sufficiently close to some convex combination of product Bernoulli measures. (Indeed, it is only then would stability of said measures, as in Section \ref{subsubsection:papermethod}, under the driven process be useful.) This is treated in Section \ref{section:tools}, as is the key square-root cancellation bound for space-time averages that we mentioned above (and that we described in Section \ref{subsubsection:papermethod}). Section \ref{section:bgphlproof} deals with making precise everything else in this section.
%
%
%
\section{Proof of Theorem \ref{theorem:main}}\label{section:mainproof}
\subsection{Stochastic evolution equation for $\mathbf{Z}^{N}$}
We first introduce terms which appear in the evolution equation for $\mathbf{Z}^{N}$ in Lemma \ref{lemma:mshe} below. Recall that $\E^{\sigma}$ is expectation with respect to product measure $\mathbb{P}^{\sigma}$ on $\{\pm1\}^{\mathbb{T}_{N}}$ such that {\small$\E^{\sigma}\eta_{\x}=\sigma$} for all $\x$. We define 
\begin{align}
\mathfrak{q}[\eta]&:=\tfrac12\mathfrak{d}[\eta](1-\eta_{0}\eta_{1})\label{eq:qterm}\\
\wt{\mathfrak{q}}[\eta]&:=\mathfrak{q}[\tau_{-2\mathfrak{l}_{\mathfrak{d}}}\eta],\label{eq:wtqterm}\\
\overline{\mathfrak{q}}[\eta]&:=\wt{\mathfrak{q}}[\eta]-\E^{0}\{\wt{\mathfrak{q}}\}-\overline{\mathfrak{d}}\eta_{0}=\wt{\mathfrak{q}}[\eta]-\E^{0}\{{\mathfrak{q}}\}-\overline{\mathfrak{d}}\eta_{0},\label{eq:overlineqterm}\\
\wt{\mathfrak{s}}[\eta]&:=-\wt{\mathfrak{q}}[\eta]\sum_{\y=0}^{2\mathfrak{l}_{\mathfrak{d}}-1}\eta_{-\y},\label{eq:wtsterm}\\
\mathfrak{s}[\eta]&:=\wt{\mathfrak{s}}[\eta]-\E^{0}\wt{\mathfrak{s}},\label{eq:sterm}\\
\mathfrak{g}[\eta]&:=\tfrac12\mathfrak{d}[\eta](\eta_{1}-\eta_{0})-\E^{0}\{\tfrac12\mathfrak{d}[\eta](\eta_{1}-\eta_{0})\}.\label{eq:gterm}
\end{align}
Only $\overline{\mathfrak{q}}$ and $\mathfrak{s}$ and $\mathfrak{g}$ appear in the equation for $\mathbf{Z}^{N}$, but we list the ingredients needed to construct these terms above for convenience. Let us now explain each term above. As we mentioned before \eqref{eq:dcoeff}, $\mathfrak{q}$ is the microscopic flux associated to $\mathfrak{d}$, and $\mathfrak{l}_{\mathfrak{d}}$ is a ``support length" of $\mathfrak{d}$, i.e. the smallest positive integer for which $\mathfrak{d}[\eta]$ depends only on $\eta_{\x}$ for $|\x|\leq\mathfrak{l}_{\mathfrak{d}}$. The role of {\small$\wt{\mathfrak{q}}$} is technical; it shifts $\mathfrak{q}$ to depend only on $\eta_{\x}$ for $\x\leq0$, and {\small$\wt{\mathfrak{s}}$} turns out to be the cost in doing this shift. The term $\overline{\mathfrak{q}}$ is important; it subtracts from {\small$\wt{\mathfrak{q}}$} its $\E^{0}$-expectation, which is compensated for in \eqref{eq:renorm}, and a linear statistic that compensates the limiting $\overline{\mathfrak{d}}\partial_{\x}$ operator in \eqref{eq:she}. \eqref{eq:sterm} is a correction of {\small$\wt{\mathfrak{s}}$}, and \eqref{eq:gterm} is another error term.
\begin{lemma}\label{lemma:mshe}
\fsp With notation defined afterwards, we have 
\begin{align*}
\mathbf{Z}^{N}_{\t,\x}&={\textstyle\int_{0}^{\t}}\mathscr{T}_{N}\mathbf{Z}^{N}_{\s,\x}\d\s+{\textstyle\int_{0}^{\t}}\mathbf{Z}^{N}_{\s,\x}\d\xi^{N}_{\s,\x}-N^{\frac12}{\textstyle\int_{0}^{\t}}\overline{\mathfrak{q}}[\tau_{\x}\eta_{\s}]\mathbf{Z}^{N}_{\s,\x}\d\s-{\textstyle\int_{0}^{\t}}\mathfrak{s}[\tau_{\x}\eta_{\s}]\mathbf{Z}^{N}_{\s,\x}\d\s\\
&+N^{-\frac12}{\textstyle\int_{0}^{\t}}\mathfrak{b}_{1}[\tau_{\x}\eta_{\s}]\mathbf{Z}^{N}_{\s,\x}\d\s+{\textstyle\int_{0}^{\t}}\mathfrak{g}[\tau_{\x}\eta_{\s}]\mathbf{Z}^{N}_{\s,\x}\d\s+N^{\frac12}{\textstyle\int_{0}^{\t}}\grad_{-2\mathfrak{l}_{\mathfrak{d}}}^{\mathbf{X}}\left\{\mathfrak{b}_{2}[\tau_{\x}\eta_{\s}]\mathbf{Z}^{N}_{\s,\x}\right\}\d\s,
\end{align*}
which we write in the following differential shorthand:
\begin{align}
\d\mathbf{Z}^{N}_{\t,\x}&=\mathscr{T}_{N}\mathbf{Z}^{N}_{\t,\x}\d\t+\mathbf{Z}^{N}_{\t,\x}\d\xi^{N}_{\t,\x}-N^{\frac12}\overline{\mathfrak{q}}[\tau_{\x}\eta_{\t}]\mathbf{Z}^{N}_{\t,\x}\d\t-\mathfrak{s}[\tau_{\x}\eta_{\t}]\mathbf{Z}^{N}_{\t,\x}\d\t+\mathfrak{g}[\tau_{\x}\eta_{\t}]\mathbf{Z}^{N}_{\t,\x}\d\t\label{eq:msheIa}\\
&+N^{-\frac12}\mathfrak{b}_{1}[\tau_{\x}\eta_{\t}]\mathbf{Z}^{N}_{\t,\x}\d\t+N^{\frac12}\grad_{-2\mathfrak{l}_{\mathfrak{d}}}^{\mathbf{X}}\left\{\mathfrak{b}_{2}[\tau_{\x}\eta_{\t}]\mathbf{Z}^{N}_{\t,\x}\right\}\d\t.\label{eq:msheIb}
\end{align}
%
\begin{itemize}
\item The functions $\mathfrak{b}_{1},\mathfrak{b}_{2}:\{\pm1\}^{\mathbb{T}_{N}}\to\R$ are uniformly bounded in $N,\eta$.
\item For any integer $\mathfrak{l}\geq0$ and $\varphi:\mathbb{T}_{N}\to\R$, let the space gradient of length $\mathfrak{l}$ be {\small$\grad_{\mathfrak{l}}^{\mathbf{X}}\varphi_{\x}:=\varphi_{\x+\mathfrak{l}}-\varphi_{\x}$}. 
\item We define $\Delta:=\grad_{1}^{\mathbf{X}}\grad_{-1}^{\mathbf{X}}$ to be the discrete Laplacian, and $\mathscr{T}_{N}:=\frac12N^{2}\Delta{{}-}\overline{\mathfrak{d}}N\grad_{-1}^{\mathbf{X}}$.
\item The process {\small$\t\mapsto\xi^{N}_{\t,\x}$} is a compensated Poisson process for all $\x\in\mathbb{T}_{N}$. Precisely, we have 
\begin{align*}
\d\xi^{N}_{\t,\x}&=\left\{\mathrm{e}^{2N^{-1/2}}-1\right\}\mathbf{1}_{\eta_{\t,\x}=1}\mathbf{1}_{\eta_{\t,\x+1}=-1}\left\{\d\mathscr{Q}^{\mathrm{S},\to}_{\t,\x}-\tfrac12N^{2}\d\t\right\}\\
&+\left\{\mathrm{e}^{-2N^{-1/2}}-1\right\}\mathbf{1}_{\eta_{\t,\x}=-1}\mathbf{1}_{\eta_{\t,\x+1}=1}\left\{\d\mathscr{Q}^{\mathrm{S},\leftarrow}_{\t,\x}-\tfrac12N^{2}\d\t\right\}\\
&-\left\{\mathrm{e}^{2N^{-1/2}}-1\right\}\mathbf{1}_{\eta_{\t,\x}=1}\mathbf{1}_{\eta_{\t,\x+1}=-1}\left\{\d\mathscr{Q}^{\mathrm{A},\to}_{\t,\x}-\left(\tfrac12N^{\frac32}+\tfrac12N\mathfrak{d}[\tau_{\x}\eta_{\t}]\right)\d\t\right\}\\
&+\left\{\mathrm{e}^{-2N^{-1/2}}-1\right\}\mathbf{1}_{\eta_{\t,\x}=-1}\mathbf{1}_{\eta_{\t,\x+1}=1}\left\{\d\mathscr{Q}^{\mathrm{A},\leftarrow}_{\t,\x}-\left(\tfrac12N^{\frac32}+\tfrac12N\mathfrak{d}[\tau_{\x}\eta_{\t}]\right)\d\t\right\}.
\end{align*}
In the above formula, {\small$\t\mapsto\mathscr{Q}^{\mathrm{S},\to}_{\t,\x}$} and {\small$\t\mapsto\mathscr{Q}^{\mathrm{S},\leftarrow}_{\t,\x}$} are Poisson process of speed {\small$\frac12N^{2}$}. The processes {\small$\t\mapsto\mathscr{Q}^{\mathrm{A},\to}_{\t,\x}$} and {\small$\t\mapsto\mathscr{Q}^{\mathrm{A},\leftarrow}_{\t,\x}$} are Poisson processes with speed {\small$\frac12N^{3/2}+\frac12N\mathfrak{d}[\tau_{\x}\eta_{\t}]$}. All Poisson processes are independent.
\end{itemize}
\end{lemma}
\begin{proof}
We will prove the differential shorthand \eqref{eq:msheIa}-\eqref{eq:msheIb}, with the understanding that everything holds after time-integration. We will borrow much of the proof from \cite{Y23} (see the proof of Proposition 2.4 therein). By (2.2) and (2.3) in \cite{Y23}:
\begin{align}
\d\mathbf{Z}^{N}_{\t,\x}&=\tfrac12N^{2}\Delta\mathbf{Z}^{N}_{\t,\x}+\mathbf{Z}^{N}_{\t,\x}\d\xi^{N}_{\t,\x}+N^{2}\Phi^{\mathrm{A}}[\tau_{\x}\eta_{\t}]\mathbf{Z}^{N}_{\t,\x}\d\t+\{N^{\frac12}\mathscr{R}_{2,1}+\mathscr{R}_{2,2}+\mathscr{R}_{2,3}\}\mathbf{Z}^{N}_{\t,\x}\d\t.\label{eq:msheI1}
\end{align}
(This follows by Ito-Dynkin applied to $\mathbf{Z}^{N}$ and standard manipulations.) In \eqref{eq:msheI1}, the term $\Phi^{\mathrm{A}}$ is given by the following, which is just (2.4)-(2.5) in \cite{Y23}, in which $\mathrm{T}^{\pm,N}=[\exp\{-2N^{-1/2}\}-1]\pm[\exp\{2N^{-1/2}\}-1]$:
\begin{align*}
\Phi^{\mathrm{A}}[\tau_{\x}\eta_{\t}]&=\tfrac18N^{-1}\mathrm{T}^{-,N}\mathfrak{d}[\tau_{\x}\eta_{\t}](1-\eta_{\t,\x}\eta_{\t,\x+1})+\tfrac18N^{-1}\mathrm{T}^{+,N}\mathfrak{d}[\tau_{\x}\eta_{\t}](\eta_{\t,\x+1}-\eta_{\t,\x}).
\end{align*}
We note that $\mathrm{T}^{-,N}=-4N^{-1/2}+\mathrm{O}(N^{-3/2})$ and $\mathrm{T}^{+,N}=4N^{-1}+\mathrm{O}(N^{-2})$. Now, by \eqref{eq:qterm}-\eqref{eq:overlineqterm}, we have 
\begin{align*}
\tfrac18N\mathrm{T}^{-,N}\mathfrak{d}[\tau_{\x}\eta_{\t}](1-\eta_{\t,\x}\eta_{\t,\x+1})&=\tfrac14N\mathrm{T}^{-,N}\left(\overline{\mathfrak{q}}[\tau_{\x}\eta_{\t}]+\E^{0}\wt{\mathfrak{q}}+\overline{\mathfrak{d}}\eta_{\t,\x}-\{\mathfrak{q}[\tau_{\x-2\mathfrak{l}_{\mathfrak{d}}}\eta_{\t}]-\mathfrak{q}[\tau_{\x}\eta_{\t}]\}\right).
\end{align*}
At this point, no assumptions made about $\mathfrak{d}[\cdot]$ in \cite{Y23} have been used yet in the proof of Proposition 2.4 in \cite{Y23}; indeed, they are only used to compute the last term in $\Phi^{\mathrm{A}}$ above. In particular, if we follow the bullet points in the proof of Proposition 2.4 in \cite{Y23}, we conclude the following. First, by the asymptotic for $\mathrm{T}^{-,N}$, we have $\frac14N\mathrm{T}^{-,N}\overline{\mathfrak{q}}=-N^{1/2}\overline{\mathfrak{q}}+\mathrm{O}(N^{-1/2})$. By the same token and by definition of $\mathscr{R}_{2,1}$, we have $\frac14N\mathrm{T}^{-,N}\E^{0}\wt{\mathfrak{q}}+N^{1/2}\mathscr{R}_{2,1}=\mathrm{O}(N^{-1/2})$. Now, we use (2.8) in \cite{Y23}; this computes gradients of $\mathbf{Z}^{N}$ by Taylor expanding to get 
\begin{align}
\tfrac14N\mathrm{T}^{-,N}\overline{\mathfrak{d}}\eta_{\t,\x}\mathbf{Z}^{N}_{\t,\x}+\mathscr{R}_{2,2}\mathbf{Z}^{N}_{\t,\x}=-N\overline{\mathfrak{d}}\grad_{-1}^{\mathbf{X}}\mathbf{Z}^{N}_{\t,\x}.\label{eq:msheI2}
\end{align}
Next, by the same idea of Taylor expanding to compute gradients of $\mathbf{Z}^{N}$, (2.9) in \cite{Y23} gives 
\begin{align*}
&-\tfrac14N\mathrm{T}^{-,N}\{\mathfrak{q}[\tau_{\x-2\mathfrak{l}_{\mathfrak{d}}}\eta_{\t}]-\mathfrak{q}[\tau_{\x}\eta_{\t}]\}\mathbf{Z}^{N}_{\t,\x}+\mathscr{R}_{2,3}\mathbf{Z}^{N}_{\t,\x}\\
&=-\mathfrak{s}[\tau_{\x}\eta_{\t}]\mathbf{Z}^{N}_{\t,\x}+N^{\frac12}\grad^{\mathbf{X}}_{-2\mathfrak{l}_{\mathfrak{d}}}\left\{\mathfrak{b}[\tau_{\x}\eta_{\t}]\mathbf{Z}^{N}_{\t,\x}\right\}-\tfrac12\E^{0}\{\mathfrak{d}[\eta][\eta_{1}-\eta_{0}]\}.
\end{align*}
In the above formula, the function $\mathfrak{b}$ is local and uniformly bounded. We note that (2.9) in \cite{Y23} does not include the last term in the second line above; this is because $\mathscr{R}_{2,3}$ is equal to $\mathrm{R}_{2,3}$ in \cite{Y23} plus $-\frac12\E^{0}\{\mathfrak{d}[\eta][\eta_{1}-\eta_{0}]\}$. If we now combine the previous four displays, we arrive at the following:
\begin{align*}
\Phi^{\mathrm{A}}[\tau_{\x}\eta_{\t}]+N^{\frac12}\mathscr{R}_{2,1}+\mathscr{R}_{2,2}+\mathscr{R}_{2,3}&=-N^{\frac12}\overline{\mathfrak{q}}[\tau_{\x}\eta_{\t}]-N\overline{\mathfrak{d}}\grad_{-1}^{\mathbf{X}}\mathbf{Z}^{N}_{\t,\x}-\mathfrak{s}[\tau_{\x}\eta_{\t}]\mathbf{Z}^{N}_{\t,\x}\\
&+N^{\frac12}\grad^{\mathbf{X}}_{-2\mathfrak{l}_{\mathfrak{d}}}\left\{\mathfrak{b}[\tau_{\x}\eta_{\t}]\mathbf{Z}^{N}_{\t,\x}\right\}-\tfrac12\E^{0}\{\mathfrak{d}[\eta][\eta_{1}-\eta_{0}]\}\\
&+\tfrac18N\mathrm{T}^{+,N}\mathfrak{d}[\tau_{\x}\eta_{\t}](\eta_{\t,\x+1}-\eta_{\t,\x}).
\end{align*}
By the previous asymptotic for $\mathrm{T}^{+,N}$, the last two terms on the RHS combine to $\mathfrak{g}[\tau_{\x}\eta_{\t}]$. Thus, it we combine the previous display with \eqref{eq:msheI1}, we obtain the desired equation \eqref{eq:msheIa}-\eqref{eq:msheIb}.
\end{proof}
As an immediate consequence of Lemma \ref{lemma:mshe}, we get another equation via the Duhamel principle. Let us start to make this precise with the following semigroup associated to $\mathscr{T}_{N}$.
\begin{definition}\label{definition:heat}
Let {\small$\mathbf{H}^{N}_{\s,\t,\x,\y}$}, for $(\s,\t,\x,\y)\in[0,\infty)^{2}\times\mathbb{T}_{N}^{2}$, be the heat kernel for $\mathscr{T}_{N}$, i.e.
\begin{align*}
\partial_{\t}\mathbf{H}^{N}_{\s,\t,\x,\y}=\mathscr{T}_{N}\mathbf{H}^{N}_{\s,\t,\x,\y}\quad\mathrm{and}\quad\mathbf{H}^{N}_{\s,\s,\x,\y}=\mathbf{1}_{\x=\y},
\end{align*}
where $\mathscr{T}_{N}$ acts on the $\x$-variable. Following standard notation, for any $\varphi:\mathbb{T}_{N}\to\R$ and any $0\leq\s\leq\t$, we set 
\begin{align}
\mathrm{e}^{[\t-\s]\mathscr{T}_{N}}[\varphi_{\cdot}]_{\x}:=\mathrm{e}^{[\t-\s]\mathscr{T}_{N}}[\varphi_{\y}]_{\x}:=\sum_{\y\in\mathbb{T}_{N}}\mathbf{H}^{N}_{\s,\t,\x,\y}\varphi_{\y}.
\end{align}
We will use both notations given above; the notation with subscript $\y$ will be used when we want to emphasize or clarify the summation variable on the far RHS. We also use the notation {\small$\exp\{[\t-\s]\mathscr{T}_{N}\}$} to mean {\small$\mathrm{e}^{[\t-\s]\mathscr{T}_{N}}$}.
\end{definition}
\begin{corollary}\label{corollary:heat}
\fsp Recall the notation from Lemma \ref{lemma:mshe}. We have the integral equation
\begin{align}
\mathbf{Z}^{N}_{\t,\x}&=\mathrm{e}^{\t\mathscr{T}_{N}}[\mathbf{Z}^{N}_{0,\cdot}]_{\x}+{\textstyle\int_{0}^{\t}}\mathrm{e}^{[\t-\s]\mathscr{T}_{N}}[\mathbf{Z}^{N}_{\s,\cdot}\d\xi^{N}_{\s,\cdot}]_{\x}\label{eq:heatIa}\\
&-{\textstyle\int_{0}^{\t}}\mathrm{e}^{[\t-\s]\mathscr{T}_{N}}[N^{\frac12}\overline{\mathfrak{q}}[\tau_{\cdot}\eta_{\s}]\mathbf{Z}^{N}_{\s,\cdot}]_{\x}\d\s\label{eq:heatIb}\\
&+{\textstyle\int_{0}^{\t}}\mathrm{e}^{[\t-\s]\mathscr{T}_{N}}[(\mathfrak{g}[\tau_{\cdot}\eta_{\s}]-\mathfrak{s}[\tau_{\cdot}\eta_{\s}])\mathbf{Z}^{N}_{\s,\cdot}]_{\x}\d\s\label{eq:heatIc}\\
&+{\textstyle\int_{0}^{\t}}\mathrm{e}^{[\t-\s]\mathscr{T}_{N}}[N^{-\frac12}\mathfrak{b}_{1}[\tau_{\cdot}\eta_{\s}]\mathbf{Z}^{N}_{\s,\cdot}]_{\x}\d\s\label{eq:heatId}\\
&+{\textstyle\int_{0}^{\t}}\mathrm{e}^{[\t-\s]\mathscr{T}_{N}}[N^{\frac12}\grad^{\mathbf{X}}_{-2\mathfrak{l}_{\mathfrak{d}}}\{\mathfrak{b}_{2}[\tau_{\cdot}\eta_{\s}]\mathbf{Z}^{N}_{\s,\cdot}\}]_{\x}\d\s.\label{eq:heatIe}
\end{align}
\end{corollary}
\eqref{eq:heatIa} is a microscopic version of \eqref{eq:she}, and the rest of the display above are error terms. \eqref{eq:heatIb} is the main error term; \eqref{eq:heatIc} also requires probabilistic inputs to control, but it is not as serious because \eqref{eq:heatIb} has a factor of {\small$N^{1/2}$}. \eqref{eq:heatId}-\eqref{eq:heatIe} will be controlled by deterministic, analytic means.
\subsection{Proving Theorem \ref{theorem:main}}
Analysis of the main error \eqref{eq:heatIb} is quite delicate, and standard Bertini-Giacomin moment-style bounds (which are the basis of many works \cite{BG,CGST,CST,CT,CTIn}, for example) do not work. Instead, we use a more sophisticated stochastic analysis approach via stopping times to control $\mathbf{Z}^{N}$.

We start with an a priori size estimate. Fix $\e_{\mathrm{ap}}>0$ arbitrarily small but positive uniformly in $N$ such that $N^{\e_{\mathrm{ap}}}$ is an integer (this constraint is purely for convenience later on and is ultimately unimportant). Define
\begin{align}
\mathfrak{t}_{\mathrm{ap}}&:=\inf\left\{\t\in[0,1]: \|\mathbf{Z}^{N}\|_{\mathrm{L}^{\infty}([0,\t]\times\mathbb{T}_{N})}+\|\tfrac{1}{\mathbf{Z}^{N}}\|_{\mathrm{L}^{\infty}([0,\t]\times\mathbb{T}_{N})}\gtrsim N^{\e_{\mathrm{ap}}}\right\}\wedge1.\label{eq:tap}
\end{align}
(Above, $\gtrsim N^{\e_{\mathrm{ap}}}$ means $\geq CN^{\e_{\mathrm{ap}}}$ for some large constant $C>0$; to avoid unnecessary notation, we use $\gtrsim$.)

By Theorem \ref{theorem:main}, $\mathbf{Z}^{N}$ is uniformly bounded (in probability), so $\mathfrak{t}_{\mathrm{ap}}=1$ should hold with probability approaching $1$. Of course, this is circular reasoning. However, we will legitimately show $\mathfrak{t}_{\mathrm{ap}}=1$ with probability approaching $1$ as $N\to\infty$. The same is true for the other stopping times introduced below as well. We move to stopping times for space regularity of $\mathbf{Z}^{N}$. Let $\e_{\mathrm{reg}}>0$ be a small constant satisfying $\e_{\mathrm{reg}}\geq C\e_{\mathrm{ap}}$ for $C>0$ fixed, large, and $\mathrm{O}(1)$. Set $\mathfrak{l}_{\mathrm{reg}}=N^{1/3+\e_{\mathrm{reg}}}$. Recall space-gradients {\small$\grad^{\mathbf{X}}_{\mathfrak{l}}$} from Lemma \ref{lemma:mshe}. Define
\begin{align}
\mathfrak{t}^{\mathrm{space}}_{\mathrm{reg}}&:=\inf\left\{\t\in[0,1]:\sup_{-2\mathfrak{l}_{\mathrm{reg}}\leq\mathfrak{l}\leq2\mathfrak{l}_{\mathrm{reg}}}\frac{N^{\frac12}|\mathfrak{l}|^{-\frac12}\|\grad^{\mathbf{X}}_{\mathfrak{l}}\mathbf{Z}^{N}\|_{\mathrm{L}^{\infty}([0,\t]\times\mathbb{T}_{N})}}{1+\|\mathbf{Z}^{N}\|_{\mathrm{L}^{\infty}([0,\t]\times\mathbb{T}_{N})}^{2}}\gtrsim N^{2\e_{\mathrm{ap}}}\right\}.\label{eq:tregX}
\end{align}
In words, {\small$\mathfrak{t}^{\mathrm{space}}_{\mathrm{reg}}$} gives a priori H\"{o}lder-{\small$\frac12$} regularity in space up to the mesoscopic scale $\mathfrak{l}_{\mathrm{reg}}$. (Although \eqref{eq:she} only has H\"{o}lder regularity of exponent {\small$\frac12-\e$} for any $\e>0$, because we are working with a discretization of scale $N$, it is the same as exponent {\small$\frac12$} as long as we include $\log N$ factors; since we have {\small$N^{2\e_{\mathrm{ap}}}$} in \eqref{eq:tregX}, we still expect {\small$\mathfrak{t}^{\mathrm{space}}_{\mathrm{reg}}=1$} with probability tending to $1$. Also, the exponent $2\e_{\mathrm{ap}}$ is purely technical.) 

Next, we define a time-gradient {\small$\grad^{\mathbf{T}}_{\s}\psi_{\t}:=\psi_{[1\wedge(\t+\s)]\vee0}-\psi_{\t}$} for functions $\psi:[0,1]\to\R$; here, $[1\wedge(\t+\s)]\vee0$ is meant to keep the time-variable inside the interval $[0,1]$. We also define the following, which gives us a priori H\"{o}lder-$1/4$ time-regularity (agreeing essentially with time-regularity of \eqref{eq:kpz}) on mesoscopic time-scales:
\begin{align}
\mathfrak{t}^{\mathrm{time}}_{\mathrm{reg}}:=\inf\left\{\t\in[0,1]:\sup_{\s\in[0,1/N]}\frac{\{N^{-2}\vee\s\}^{-\frac14}\|\grad^{\mathbf{T}}_{-\s}\mathbf{Z}^{N}\|_{\mathrm{L}^{\infty}([0,\t]\times\mathbb{T}_{N})}}{1+\|\mathbf{Z}^{N}\|_{\mathrm{L}^{\infty}([0,\t]\times\mathbb{T}_{N})}^{2}}\gtrsim N^{2\e_{\mathrm{ap}}}\right\}\wedge1.\label{eq:tregT}
\end{align}
Since time-gradients in \eqref{eq:tregT} are in the backwards direction, {\small$\mathfrak{t}^{\mathrm{time}}_{\mathrm{reg}}$} is a stopping time. Also, the use of $N^{-2}\vee$ is because $\mathbf{Z}^{N}$ has jumps of speed $N^{2}$, so it is not H\"{o}lder continuous on ``microscopic" time-scales. Finally, set
\begin{align}
\mathfrak{t}_{\mathrm{stop}}:=\mathfrak{t}_{\mathrm{ap}}\wedge\mathfrak{t}^{\mathrm{space}}_{\mathrm{reg}}\wedge\mathfrak{t}^{\mathrm{time}}_{\mathrm{reg}} \quad\mathrm{and}\quad \mathbf{Y}^{N}_{\t,\x}:=\mathbf{Z}^{N}_{\t,\x}\mathbf{1}_{\t\leq\mathfrak{t}_{\mathrm{stop}}}.\label{eq:tst}
\end{align}
To justify using this stopping time, we need to show that {\small$\mathfrak{t}_{\mathrm{stop}}=1$} with probability tending to $1$ as $N\to\infty$. This is the result below; first, we introduce a notion of (very) high probability to be used frequently in this paper.
\begin{definition}\label{definition:hp}
\fsp We say an event $\mathcal{E}$ holds with high probability if $\mathbb{P}[\mathcal{E}^{C}]\leq \delta+C_{\delta}\mathrm{o}(1)$, where $\delta>0$ is arbitrary and $C_{\delta}$ depends only on $\delta$. We say $\mathcal{E}$ holds with very high probability if $\mathbb{P}[\mathcal{E}^{C}]\lesssim_{D}N^{-D}$ for any $D>0$.
\end{definition}
\begin{prop}\label{prop:tst=1}
\fsp We have $\mathfrak{t}_{\mathrm{stop}}=1$ with high probability.
\end{prop}
We defer the proof of Proposition \ref{prop:tst=1} to Section \ref{section:propproofs}. Now, let $\mathbf{U}^{N}:[0,\infty)\times\mathbb{T}_{N}\to\R$ solve the same SDE solved by $\mathbf{Z}^{N}$, but replace $\mathbf{Z}^{N}$ by $\mathbf{Y}^{N}$ in all the error terms \eqref{eq:heatIb}-\eqref{eq:heatIe}. We let $\mathbf{Q}^{N}:[0,\infty)\times\mathbb{T}_{N}\to\R$ solve the same SDE as $\mathbf{Z}^{N}$ but removing \eqref{eq:heatIb}-\eqref{eq:heatIe} entirely. In particular, 
\begin{align}
\mathbf{U}^{N}_{\t,\x}&=\mathrm{e}^{\t\mathscr{T}_{N}}[\mathbf{Z}^{N}_{0,\cdot}]_{\x}+{\textstyle\int_{0}^{\t}}\mathrm{e}^{[\t-\s]\mathscr{T}_{N}}[\mathbf{U}^{N}_{\s,\cdot}\d\xi^{N}_{\s,\cdot}]_{\x}\label{eq:uheatIa}\\
&-{\textstyle\int_{0}^{\t}}\mathrm{e}^{[\t-\s]\mathscr{T}_{N}}[N^{\frac12}\overline{\mathfrak{q}}[\tau_{\cdot}\eta_{\s}]\mathbf{Y}^{N}_{\s,\cdot}]_{\x}\d\s\label{eq:uheatIb}\\
&+{\textstyle\int_{0}^{\t}}\mathrm{e}^{[\t-\s]\mathscr{T}_{N}}[(\mathfrak{g}[\tau_{\cdot}\eta_{\s}]-\mathfrak{s}[\tau_{\cdot}\eta_{\s}])\mathbf{Y}^{N}_{\s,\cdot}]_{\x}\d\s\label{eq:uheatIc}\\
&+{\textstyle\int_{0}^{\t}}\mathrm{e}^{[\t-\s]\mathscr{T}_{N}}[N^{-\frac12}\mathfrak{b}_{1}[\tau_{\cdot}\eta_{\s}]\mathbf{Y}^{N}_{\s,\cdot}]_{\x}\d\s\label{eq:uheatId}\\
&+{\textstyle\int_{0}^{\t}}\mathrm{e}^{[\t-\s]\mathscr{T}_{N}}[N^{\frac12}\grad^{\mathbf{X}}_{-2\mathfrak{l}_{\mathfrak{d}}}\{\mathfrak{b}_{2}[\tau_{\cdot}\eta_{\s}]\mathbf{Y}^{N}_{\s,\cdot}\}]_{\x}\d\s.\label{eq:uheatIe}
\end{align}
and 
\begin{align}
\mathbf{Q}^{N}_{\t,\x}&=\mathrm{e}^{\t\mathscr{T}_{N}}[\mathbf{Z}^{N}_{0,\cdot}]_{\x}+{\textstyle\int_{0}^{\t}}\mathrm{e}^{[\t-\s]\mathscr{T}_{N}}[\mathbf{Q}^{N}_{\s,\cdot}\d\xi^{N}_{\s,\cdot}]_{\x}.\label{eq:qheatIa}
\end{align}
Since $\mathbf{U}^{N}$ and $\mathbf{Z}^{N}$ satisfy the same SDE until time {\small$\mathfrak{t}_{\mathrm{stop}}$} with the same initial data, we know $\mathbf{U}^{N}=\mathbf{Z}^{N}$ until time {\small$\mathfrak{t}_{\mathrm{stop}}$}. So, the proof of Theorem \ref{theorem:main} will follow from Proposition \ref{prop:tst=1} and the two results below (the second of which is more or less a standard tightness and martingale problem argument a la \cite{BG}).
\begin{prop}\label{prop:uq}
\fsp There exists a fixed $\beta>0$ so that {\small$\|\mathbf{U}^{N}-\mathbf{Q}^{N}\|_{\mathrm{L}^{\infty}([0,1]\times\mathbb{T}_{N})}\lesssim N^{-\beta}$} with high probability. 
\end{prop}
\begin{prop}\label{prop:qconvergence}
\fsp Theorem \ref{theorem:main} holds for $\mathbf{Q}^{N}$ in place of $\mathbf{Z}^{N}$.
\end{prop}
We are left to prove Propositions \ref{prop:tst=1}, \ref{prop:uq}, and \ref{prop:qconvergence}. We first provide important estimates for the error terms in \eqref{eq:uheatIb}-\eqref{eq:uheatIe} in Section \ref{section:bgp}, and then we prove Propositions \ref{prop:uq}, \ref{prop:qconvergence}, and \ref{prop:tst=1} in that order in Section \ref{section:propproofs}.
%
%
%
\section{Homogenization and Boltzmann-Gibbs principles}\label{section:bgp}
The key ingredient to proving Propositions \ref{prop:tst=1}, \ref{prop:uq}, and \ref{prop:qconvergence} is known as a \emph{Boltzmann-Gibbs principle}; it is a bound on \eqref{eq:uheatIb} (whose proof leads also to a bound on \eqref{eq:uheatIc}, which itself is a classical hydrodynamic limit argument). To clarify the exposition,  we let the ``Boltzmann-Gibbs" and ``hydrodynamic limit" terms be
\begin{align}
\Upsilon^{\mathrm{BG}}_{\t,\x}&:={\textstyle\int_{0}^{\t}}\mathrm{e}^{[\t-\s]\mathscr{T}_{N}}[N^{\frac12}\overline{\mathfrak{q}}[\tau_{\cdot}\eta_{\s}]\mathbf{Y}^{N}_{\s,\cdot}]_{\x}\d\s\label{eq:upsilonbg}\\
\Upsilon^{\mathrm{HL}}_{\t,\x}&:={\textstyle\int_{0}^{\t}}\mathrm{e}^{[\t-\s]\mathscr{T}_{N}}[(\mathfrak{g}[\tau_{\cdot}\eta_{\s}]-\mathfrak{s}[\tau_{\cdot}\eta_{\s}])\mathbf{Y}^{N}_{\s,\cdot}]_{\x}\d\s.\label{eq:upsilonhl}
\end{align}
%
\begin{theorem}\label{theorem:bgp}
\fsp There exists a universal constant $\beta_{\mathrm{BG}}>0$ such that
\begin{align}
\E\|\Upsilon^{\mathrm{BG}}\|_{\mathrm{L}^{\infty}([0,1]\times\mathbb{T}_{N})}\lesssim N^{-\beta_{\mathrm{BG}}}.\label{eq:bgp}
\end{align}
\end{theorem}
\begin{prop}\label{prop:hl}
\fsp There exists a universal constant $\beta_{\mathrm{HL}}>0$ such that
\begin{align}
\E\|\Upsilon^{\mathrm{HL}}\|_{\mathrm{L}^{\infty}([0,1]\times\mathbb{T}_{N})}\lesssim N^{-\beta_{\mathrm{HL}}}.\label{eq:hl}
\end{align}
\end{prop}
We now reap the benefits of Theorem \ref{theorem:bgp} and Proposition \ref{prop:hl}. Recall $\mathfrak{l}_{\mathrm{reg}}=N^{1/3+\e_{\mathrm{reg}}}$ from before \eqref{eq:tregX}.
\begin{lemma}\label{lemma:bgpreg}
\fsp We have the following estimate (which is also true if we replace $\Upsilon^{\mathrm{BG}}$ by $\Upsilon^{\mathrm{HL}}$):
\begin{align}
\E\left\{\left\|\sup_{-2\mathfrak{l}_{\mathrm{reg}}\leq\mathfrak{l}\leq2\mathfrak{l}_{\mathrm{reg}}}N^{\frac12}|\mathfrak{l}|^{-\frac12}|\grad^{\mathbf{X}}_{\mathfrak{l}}\Upsilon^{\mathrm{BG}}|\right\|_{\mathrm{L}^{\infty}([0,1]\times\mathbb{T}_{N})}\right\}\lesssim N^{-\frac{1}{13}}.\label{eq:bgpreg}
\end{align}
\end{lemma}
\begin{proof}
First, we fix any integer $\mathfrak{l}$ between $-2\mathfrak{l}_{\mathrm{reg}},2\mathfrak{l}_{\mathrm{reg}}$. Next, we observe the elementary bounds {\small$N^{1/2}|\mathfrak{l}|^{-1/2}\lesssim N^{1/2}|\mathfrak{l}|^{1/2}|\mathfrak{l}|^{-1}\lesssim N^{1/2}N^{1/6+\e_{\mathrm{reg}}}|\mathfrak{l}|^{-1}\lesssim N^{2/3+\e_{\mathrm{reg}}}|\mathfrak{l}|^{-1}$}. Thus, since {\small$\frac23+\frac{1}{13}<\frac34$}, if we take $\e_{\mathrm{reg}}$ small enough, then to prove \eqref{eq:bgpreg}, it suffices to show 
\begin{align}
\E\left\{\left\|\sup_{-2\mathfrak{l}_{\mathrm{reg}}\leq\mathfrak{l}\leq2\mathfrak{l}_{\mathrm{reg}}}|\mathfrak{l}|^{-1}|\grad^{\mathbf{X}}_{\mathfrak{l}}\Upsilon^{\mathrm{BG}}|\right\|_{\mathrm{L}^{\infty}([0,1]\times\mathbb{T}_{N})}\right\}\lesssim N^{-\frac34}.\label{eq:bgpregaux}
\end{align}
Fix any $\t\geq0$ and $\x\in\mathbb{T}_{N}$. We write
\begin{align}
|\mathfrak{l}|^{-1}\grad^{\mathbf{X}}_{\mathfrak{l}}\Upsilon^{\mathrm{BG}}_{\t,\x}&={\textstyle\int_{0}^{\t-\mathrm{T}_{N}}}|\mathfrak{l}|^{-1}\grad^{\mathbf{X}}_{\mathfrak{l}}\mathrm{e}^{[\t-\s]\mathscr{T}_{N}}[N^{\frac12}\overline{\mathfrak{q}}[\tau_{\cdot}\eta_{\s}]\mathbf{Y}^{N}_{\s,\cdot}]_{\x}\d\s\label{eq:bgpreg1}\\
&+{\textstyle\int_{\t-\mathrm{T}_{N}}^{\t}}|\mathfrak{l}|^{-1}\grad^{\mathbf{X}}_{\mathfrak{l}}\mathrm{e}^{[\t-\s]\mathscr{T}_{N}}[N^{\frac12}\overline{\mathfrak{q}}[\tau_{\cdot}\eta_{\s}]\mathbf{Y}^{N}_{\s,\cdot}]_{\x}\d\s,\nonumber
\end{align}
where {\small$\mathrm{T}_{N}:=N^{-1/2-2\e_{\mathrm{ap}}}\wedge\t$}. By spatial regularity of the heat kernel $\mathbf{H}^{N}$ (see Proposition \ref{prop:hk}) and the bounds $|\overline{\mathfrak{q}}|\lesssim1$ and $|\mathbf{Y}^{N}|\lesssim N^{\e_{\mathrm{ap}}}$ (for the latter bound, see \eqref{eq:tap} and \eqref{eq:tst}), we have 
\begin{align*}
{\textstyle\int_{\t-\mathrm{T}_{N}}^{\t}}|\mathfrak{l}|^{-1}\grad^{\mathbf{X}}_{\mathfrak{l}}\mathrm{e}^{[\t-\s]\mathscr{T}_{N}}[N^{\frac12}\overline{\mathfrak{q}}[\tau_{\cdot}\eta_{\s}]\mathbf{Y}^{N}_{\s,\cdot}]_{\x}\d\s&\lesssim{\textstyle\int_{\t-\mathrm{T}_{N}}^{\t}}N^{-1}N^{\frac12+\e_{\mathrm{ap}}}|\t-\s|^{-\frac12}\d\s\\
&\lesssim N^{-1}N^{\frac12+\e_{\mathrm{ap}}}N^{-\frac14-\e_{\mathrm{ap}}}\lesssim N^{-\frac34}.
\end{align*}
For the first term on the RHS of \eqref{eq:bgpreg1}, we use the semigroup property to write
\begin{align*}
{\textstyle\int_{0}^{\t-\mathrm{T}_{N}}}|\mathfrak{l}|^{-1}\grad^{\mathbf{X}}_{\mathfrak{l}}\mathrm{e}^{[\t-\s]\mathscr{T}_{N}}[N^{\frac12}\overline{\mathfrak{q}}[\tau_{\cdot}\eta_{\s}]\mathbf{Y}^{N}_{\s,\cdot}]_{\x}\d\s&={\textstyle\int_{0}^{\t-\mathrm{T}_{N}}}|\mathfrak{l}|^{-1}\grad^{\mathbf{X}}_{\mathfrak{l}}\mathrm{e}^{\mathrm{T}_{N}\mathscr{T}_{N}}\circ\mathrm{e}^{[\t-\s-\mathrm{T}_{N}]\mathscr{T}_{N}}[N^{\frac12}\overline{\mathfrak{q}}[\tau_{\cdot}\eta_{\s}]\mathbf{Y}^{N}_{\s,\cdot}]_{\x}\d\s.
\end{align*}
We use linearity of the semigroup operator to pull {\small$|\mathfrak{l}|^{-1}\grad^{\mathbf{X}}_{\mathfrak{l}}\mathrm{e}^{\mathrm{T}_{N}\mathscr{T}_{N}}$} out of the $\d\s$-integration. Then, we use space regularity of the heat kernel $\mathbf{H}^{N}$ again (see Proposition \ref{prop:hk}) to get that {\small$|\mathfrak{l}|^{-1}\grad^{\mathbf{X}}_{\mathfrak{l}}\mathrm{e}^{\mathrm{T}_{N}\mathscr{T}_{N}}:\mathrm{L}^{\infty}(\mathbb{T}_{N})\to\mathrm{L}^{\infty}(\mathbb{T}_{N})$} has operator norm {\small$\lesssim N^{-1}\mathrm{T}_{N}^{-1/2}=N^{-3/4+\e_{\mathrm{ap}}}$}. We ultimately get the following for the first term on the RHS of \eqref{eq:bgpreg1}:
\begin{align*}
&\sup_{\x\in\mathbb{T}_{N}}\sup_{-2\mathfrak{l}_{\mathrm{reg}}\leq\mathfrak{l}\leq2\mathfrak{l}_{\mathrm{reg}}}|{\textstyle\int_{0}^{\t-\mathrm{T}_{N}}}|\mathfrak{l}|^{-1}\grad^{\mathbf{X}}_{\mathfrak{l}}\mathrm{e}^{[\t-\s]\mathscr{T}_{N}}[N^{\frac12}\overline{\mathfrak{q}}[\tau_{\cdot}\eta_{\s}]\mathbf{Y}^{N}_{\s,\cdot}]_{\x}\d\s|\\
&\lesssim N^{-\frac34+\e_{\mathrm{ap}}}\sup_{\x\in\mathbb{T}_{N}}|{\textstyle\int_{0}^{\t-\mathrm{T}_{N}}}\mathrm{e}^{[\t-\s-\mathrm{T}_{N}]\mathscr{T}_{N}}[N^{\frac12}\overline{\mathfrak{q}}[\tau_{\cdot}\eta_{\s}]\mathbf{Y}^{N}_{\s,\cdot}]_{\x}\d\s|.
\end{align*}
The term in the $\x$-sup in the second line is just {\small$\Upsilon^{\mathrm{BG}}_{\t-\mathrm{T}_{N},\x}$}. So, we can take sup over $\t\in[0,1]$, take expectation and apply Theorem \ref{theorem:bgp} to deduce the previous display is $\lesssim N^{-3/4+\e_{\mathrm{ap}}}N^{-\beta_{\mathrm{BG}}}$ for some universal $\beta_{\mathrm{BG}}>0$. Since we can take $\e_{\mathrm{ap}}$ smaller than $\beta_{\mathrm{BG}}$, we ultimately deduce that the first term on the RHS of \eqref{eq:bgpreg1} is $\lesssim N^{-3/4}$. This gives \eqref{eq:bgpregaux}. The argument is the same for $\Upsilon^{\mathrm{HL}}$ except we replace {\small$N^{1/2}\overline{\mathfrak{q}}[\tau_{\cdot}\eta_{\s}]$} by {\small$\mathfrak{g}[\tau_{\cdot}\eta_{\s}]-\mathfrak{s}[\tau_{\cdot}\eta_{\s}]$} and we use Proposition \ref{prop:hl} instead of Theorem \ref{theorem:bgp}.
\end{proof}
We now record a time-regularity estimate. Although the following is entirely deterministic and does not use Theorem \ref{theorem:bgp} or Proposition \ref{prop:hl}, we still give it below to make this section a central hub for error terms.
\begin{lemma}\label{lemma:timeregbgpterms}
\fsp We have the following deterministic estimate (which also holds if we replace $\Upsilon^{\mathrm{BG}}$ by $\Upsilon^{\mathrm{HL}}$):
\begin{align}
\left\|\sup_{\r\in[0,1/N]}\r^{-\frac14}|\grad^{\mathbf{T}}_{-\r}\Upsilon^{\mathrm{BG}}|\right\|_{\mathrm{L}^{\infty}([0,1]\times\mathbb{T}_{N})}\lesssim 1.\label{eq:timeregbgpterms}
\end{align}
\end{lemma}
\begin{proof}
Intuitively, the scale-$\r$ time-gradient $\r$; when combined with $N^{1/2}$ in \eqref{eq:upsilonbg}, this gives $\lesssim N^{1/2}\r$, which is $\lesssim\r^{1/4}$ if $\r\lesssim N^{-1}$. Let us give the details below. Fix $\r\in[0,1/N]$ and $(\t,\x)$. In what follows, we always interpret $\t-\r$ as its positive part (but do not write it to ease notation). By \eqref{eq:upsilonbg}, 
\begin{align}
\grad^{\mathbf{T}}_{-\r}\Upsilon^{\mathrm{BG}}_{\t,\x}&={\int_{\t-\r}^{\t}}\sum_{\y\in\mathbb{T}_{N}}\mathbf{H}^{N}_{\s,\t,\x,\y}N^{\frac12}\overline{\mathfrak{q}}[\tau_{\y}\eta_{\s}]\mathbf{Y}^{N}_{\s,\y}\d\s\label{eq:timeregbgpterms1}\\
&+\int_{0}^{\t-\r}\sum_{\y\in\mathbb{T}_{N}}\{\mathbf{H}^{N}_{\s,\t,\x,\y}-\mathbf{H}^{N}_{\s,\t-\r,\x,\y}\}N^{\frac12}\overline{\mathfrak{q}}[\tau_{\y}\eta_{\s}]\mathbf{Y}^{N}_{\s,\y}\d\s.\nonumber
\end{align}
Because the semigroup operator {\small$\exp\{[\t-\s]\mathscr{T}_{N}\}:\mathrm{L}^{\infty}(\mathbb{T}_{N})\to\mathrm{L}^{\infty}(\mathbb{T}_{N})$} has operator norm $\mathrm{O}(1)$, the RHS of the first line of \eqref{eq:timeregbgpterms1} is $\lesssim N^{1/2+\e_{\mathrm{ap}}}\r$ (again, we use $|\mathbf{Y}^{N}|\lesssim N^{\e_{\mathrm{ap}}}$ for this). For the second line, we use Proposition \ref{prop:hk} to get that the $\mathrm{L}^{1}(\mathbb{T}_{N})$-norm of {\small$\mathbf{H}^{N}_{\s,\t,\x,\y}-\mathbf{H}^{N}_{\s,\t-\r,\x,\y}$} in $\y$ is {\small$\lesssim_{\e} \r^{1-\e}|\t-\r-\s|^{-1+\e}$} for $\e>0$ small. So, the second line of \eqref{eq:timeregbgpterms1} is {\small$\lesssim_{\e}\r^{1-\e}N^{1/2+\e_{\mathrm{ap}}}\int_{0}^{\t-\r}|\t-\r-\s|^{-1+\e}\d\s\lesssim_{\e}N^{1/2+\e_{\mathrm{ap}}}\r^{1-\e}$}. Ultimately, we deduce 
\begin{align*}
\r^{-\frac14}|\grad^{\mathbf{T}}_{-\r}\Upsilon^{\mathrm{BG}}_{\t,\x}|&\lesssim_{\e} N^{\frac12+\e_{\mathrm{ap}}}\r^{\frac34-\e}\lesssim N^{-\frac14+\e_{\mathrm{ap}}+\e},
\end{align*}
since $\r\lesssim N^{-1}$. If we choose $\e_{\mathrm{ap}},\e$ small enough, the RHS of the previous display is $\lesssim1$. Since this bound is uniform over $(\t,\x)$-variables on the LHS, \eqref{eq:timeregbgpterms} follows. The argument for $\Upsilon^{\mathrm{HL}}$ is the same (even easier); just replace {\small$N^{1/2}\overline{\mathfrak{q}}[\tau_{\cdot}\eta_{\s}]$} by {\small$\mathfrak{g}[\tau_{\cdot}\eta_{\s}]-\mathfrak{s}[\tau_{\cdot}\eta_{\s}]$}. 
\end{proof}
\subsection{Other error terms}
Again, for the sake of collecting all error estimates in this section, let us also bound \eqref{eq:uheatId} and \eqref{eq:uheatIe}. By boundedness of the $\mathscr{T}_{N}$-semigroup on $\mathrm{L}^{\infty}(\mathbb{T}_{N})$ (see Proposition \ref{prop:hk}) and $|\mathbf{Y}^{N}|\leq N^{\e_{\mathrm{ap}}}$ (see \eqref{eq:tst}) and $\mathfrak{b}_{1}=\mathrm{O}(1)$ (see Lemma \ref{lemma:mshe}), we have 
\begin{align}
\sup_{\t\in[0,1]}\sup_{\x\in\mathbb{T}_{N}}|{\textstyle\int_{0}^{\t}}\mathrm{e}^{[\t-\s]\mathscr{T}_{N}}[N^{-\frac12}\mathfrak{b}_{1}[\tau_{\cdot}\eta_{\s}]\mathbf{Y}^{N}_{\s,\cdot}]_{\x}\d\s|\lesssim N^{-\frac12+\e_{\mathrm{ap}}}.\label{eq:err1}
\end{align}
We now claim that the operator norm of {\small$\exp\{[\t-\s]\mathscr{T}_{N}\}\circ\grad^{\mathbf{X}}_{\mathfrak{l}}$} on $\mathrm{L}^{\infty}(\mathbb{T}_{N})$ is {\small$\lesssim |\mathfrak{l}|N^{-1}|\t-\s|^{-1/2}$}. This follows because {\small$\exp\{[\t-\s]\mathscr{T}_{N}\}\circ\grad^{\mathbf{X}}_{\mathfrak{l}}=\grad^{\mathbf{X}}_{\mathfrak{l}}\circ\exp\{[\t-\s]\mathscr{T}_{N}\}$}, as space-gradients commute with $\mathscr{T}_{N}$, and by Proposition \ref{prop:hk}. By this and the same bounds $|\mathbf{Y}^{N}|\leq N^{\e_{\mathrm{ap}}}$ and $\mathfrak{b}_{1}=\mathrm{O}(1)$, we get
\begin{align}
&\sup_{\t\in[0,1]}\sup_{\x\in\mathbb{T}_{N}}|{\textstyle\int_{0}^{\t}}\mathrm{e}^{[\t-\s]\mathscr{T}_{N}}[N^{\frac12}\grad^{\mathbf{X}}_{-2\mathfrak{l}_{\mathfrak{d}}}\{\mathfrak{b}_{2}[\tau_{\cdot}\eta_{\s}]\mathbf{Y}^{N}_{\s,\cdot}\}]_{\x}\d\s|\nonumber\\
&\lesssim\sup_{\t\in[0,1]}\sup_{\x\in\mathbb{T}_{N}}{\textstyle\int_{0}^{\t}}N^{-1}|\t-\s|^{-\frac12}N^{\frac12+\e_{\mathrm{ap}}}\d\s\lesssim N^{-\frac12+\e_{\mathrm{ap}}}.\label{eq:err2}
\end{align}
Finally, we remark that the same bounds would be true if we instead take the space gradient of the term inside absolute values on the LHS of any estimate \eqref{eq:err1}-\eqref{eq:err2} above. Indeed, space gradients are bounded (in operator norm) on $\mathrm{L}^{\infty}([0,1]\times\mathbb{T}_{N})$ (we just do not gain any regularity factor). Precisely, we have
\begin{align}
\sup_{\t\in[0,1]}\sup_{\x\in\mathbb{T}_{N}}\sup_{\mathfrak{l}\neq0}N^{\frac12}|\mathfrak{l}|^{-1}|\grad^{\mathbf{X}}_{\mathfrak{l}}{\textstyle\int_{0}^{\t}}\mathrm{e}^{[\t-\s]\mathscr{T}_{N}}[N^{-\frac12}\mathfrak{b}_{1}[\tau_{\cdot}\eta_{\s}]\mathbf{Y}^{N}_{\s,\cdot}]_{\x}\d\s|&\lesssim N^{\e_{\mathrm{ap}}}, \label{eq:err1regX}\\
\sup_{\t\in[0,1]}\sup_{\x\in\mathbb{T}_{N}}\sup_{\mathfrak{l}\neq0}N^{\frac12}|\mathfrak{l}|^{-1}|\grad^{\mathbf{X}}_{\mathfrak{l}}{\textstyle\int_{0}^{\t}}\mathrm{e}^{[\t-\s]\mathscr{T}_{N}}[N^{\frac12}\grad^{\mathbf{X}}_{-2\mathfrak{l}_{\mathfrak{d}}}\{\mathfrak{b}_{2}[\tau_{\cdot}\eta_{\s}]\mathbf{Y}^{N}_{\s,\cdot}\}]_{\x}\d\s|&\lesssim N^{\e_{\mathrm{ap}}}.\label{eq:err2regX}
\end{align}
Also, the proof of Lemma \ref{lemma:timeregbgpterms} gives the following (if we replace $N^{1/2}\overline{\mathfrak{q}}\mathbf{Y}^{N}$ in said proof by $N^{-1/2}\mathfrak{b}_{1}\mathbf{Y}^{N}$ or {\small$N^{1/2}\grad^{\mathbf{X}}_{-2\mathfrak{l}_{\mathfrak{d}}}[\mathfrak{b}_{2}\mathbf{Y}^{N}]$}, the key point being {\small$\mathfrak{b}_{1}\mathbf{Y}^{N}$} and {\small$\grad^{\mathbf{X}}_{-2\mathfrak{l}_{\mathfrak{d}}}[\mathfrak{b}_{2}\mathbf{Y}^{N}]$} are both $\mathrm{O}(N^{\e_{\mathrm{ap}}})$ deterministically):
\begin{align}
\sup_{\t\in[0,1]}\sup_{\x\in\mathbb{T}_{N}}\sup_{\r\in[0,1/N]}\r^{-\frac14}|\grad^{\mathbf{T}}_{-\r}{\textstyle\int_{0}^{\t}}\mathrm{e}^{[\t-\s]\mathscr{T}_{N}}[N^{-\frac12}\mathfrak{b}_{1}[\tau_{\cdot}\eta_{\s}]\mathbf{Y}^{N}_{\s,\cdot}]_{\x}\d\s|&\lesssim 1, \label{eq:err1regT}\\
\sup_{\t\in[0,1]}\sup_{\x\in\mathbb{T}_{N}}\sup_{\r\in[0,1/N]}\r^{-\frac14}|\grad^{\mathbf{T}}_{-\r}{\textstyle\int_{0}^{\t}}\mathrm{e}^{[\t-\s]\mathscr{T}_{N}}[N^{\frac12}\grad^{\mathbf{X}}_{-2\mathfrak{l}_{\mathfrak{d}}}\{\mathfrak{b}_{2}[\tau_{\cdot}\eta_{\s}]\mathbf{Y}^{N}_{\s,\cdot}\}]_{\x}\d\s|&\lesssim 1.\label{eq:err2regT}
\end{align}
%
%
%
%
\section{Proofs of Proposition \ref{prop:uq}, \ref{prop:qconvergence}, and \ref{prop:tst=1} in that order}\label{section:propproofs}
\subsection{Proof of Proposition \ref{prop:uq}}
For convenience, we let $\mathbf{D}^{N}:=\mathbf{U}^{N}-\mathbf{Q}^{N}$. It is straightforward to check that
\begin{align*}
\mathbf{D}^{N}_{\t,\x}={\textstyle\int_{0}^{\t}}\mathrm{e}^{[\t-\s]\mathscr{T}_{N}}[\mathbf{D}^{N}_{\s,\cdot}\d\xi^{N}_{\s,\cdot}]_{\x}+\Phi_{\t,\x},
\end{align*}
where $\Phi$ is given by \eqref{eq:uheatIb}-\eqref{eq:uheatIe}. Differentiating this equation in $\t$ also gives the following, in which $\Psi_{\t,\x}$ is equal to the $\d\s$-integrand in \eqref{eq:uheatIb}-\eqref{eq:uheatIe} evaluated at $\s=\t$:
\begin{align*}
\d\mathbf{D}^{N}_{\t,\x}&=\mathscr{T}_{N}\mathbf{D}^{N}_{\t,\x}\d\t+\mathbf{D}^{N}_{\t,\x}\d\xi^{N}_{\t,\x}+\Psi_{\t,\x}\d\t.
\end{align*}
(Note that $\Psi_{\t,\x}$ is uniformly bounded by $\mathrm{O}(N^{2})$ since $\mathbf{Y}^{N}=\mathrm{O}(N)$ by construction; see \eqref{eq:tst}.) The operator $\mathscr{T}_{N}:\mathrm{L}^{\infty}(\mathbb{T}_{N})\to\mathrm{L}^{\infty}(\mathbb{T}_{N})$ has norm $\mathrm{O}(N^{2})$ (it is discrete differentiation scaled by $\mathrm{O}(N^{2})$). The noise $\xi^{N}$ is a sum of compensated Poisson clocks of speed $\mathrm{O}(N^{2})$. Thus, standard short-time estimates (see Lemma A.6 in \cite{Y23}, for example) allow us to control $\mathbf{D}^{N}$ on $[0,1]\times\mathbb{T}_{N}$ in terms of its sup-norm on a very fine discretization of $[0,1]\times\mathbb{T}_{N}$. In particular, for any $\beta>0$ independent of $N$, we have 
\begin{align*}
\mathbb{P}[\|\mathbf{D}^{N}\|_{\mathrm{L}^{\infty}([0,1]\times\mathbb{T}_{N})}\gtrsim N^{-\beta}]\leq\mathbb{P}\left\{{\sup_{\substack{\t\in[0,1]\cap N^{-100}\Z\\ \x\in\mathbb{T}_{N}}}}|\mathbf{D}^{N}_{\t,\x}|\gtrsim N^{-\beta}\right\}+\mathrm{o}(1),
\end{align*}
where the additional $\mathrm{o}(1)$ comes from the probability that one of the $\mathrm{O}(N)$-many Poisson clocks rings more than $N^{\e}$-many times in at least one of the intervals of length $N^{-100}$ between adjacent points in $[0,1]\cap N^{-100}\Z$. (Such an event has probability that is exponentially small in $N$. Thus, a union bound over the $N^{\mathrm{O}(1)}$-many such clocks and time-intervals still produces $\mathrm{o}(1)$; again, see Lemma A.6 in \cite{Y23} for details.) 

Let $\mathcal{E}$ denote the event in which {\small$|\Phi_{\t,\x}|\lesssim N^{-\beta/2}$} uniformly in $(\t,\x)\in[0,1]\times\mathbb{T}_{N}$, where  $\beta=\beta_{\mathrm{BG}}\wedge\beta_{\mathrm{HL}}$ (see Theorem \ref{theorem:bgp}, Proposition \ref{prop:hl} for these constants). By Theorem \ref{theorem:bgp}, Proposition \ref{prop:hl}, the Markov inequality, \eqref{eq:err1} and \eqref{eq:err2}, we know that $\mathcal{E}$ holds with high probability. So, if we let $\mathfrak{t}^{\Phi}$ be the first time where {\small$|\Phi_{\mathfrak{t}^{\Phi},\x}|\gtrsim N^{-\beta/2}$} for some $\x\in\mathbb{T}_{N}$, we know that $\mathfrak{t}^{\Phi}\wedge1=1$ with high probability. (Note that $\Phi$ is continuous in time; thus {\small$|\Phi_{\t,\x}|\lesssim N^{-\beta/2}$} for all $\t\leq\mathfrak{t}^{\Phi}$.) In particular, if we set {\small$\Phi^{\wedge}_{\t,\x}:=\Phi_{\t\wedge\mathfrak{t}^{\Phi},\x}$}, and if we let $\mathbf{D}^{N,\wedge}$ solve
\begin{align*}
\mathbf{D}^{N,\wedge}_{\t,\x}&={\textstyle\int_{0}^{\t}}\mathrm{e}^{[\t-\s]\mathscr{T}_{N}}[\mathbf{D}^{N,\wedge}_{\s,\cdot}\d\xi^{N}_{\s,\cdot}]_{\x}+\Phi^{\wedge}_{\t,\x},
\end{align*}
then by uniqueness of solutions to SDEs, we know that $\mathbf{D}^{N}=\mathbf{D}^{N,\wedge}$ until time $\mathfrak{t}^{\Phi}$, and thus up to time $1$ with high probability. (The advantage with $\mathbf{D}^{N,\wedge}$ is that the forcing term $\Phi^{\wedge}$ is deterministically small.) In particular, it suffices to now show the following for some fixed $\beta_{2}>0$, we have
\begin{align*}
\mathbb{P}\left\{{\sup_{\substack{\t\in[0,1]\cap N^{-100}\Z\\ \x\in\mathbb{T}_{N}}}}|\mathbf{D}^{N,\wedge}_{\t,\x}|\gtrsim N^{-\beta}\right\}\leq\sum_{\substack{\t\in[0,1]\cap N^{-100}\Z\\\x\in\mathbb{T}_{N}}}\mathbb{P}[|\mathbf{D}^{N,\wedge}_{\t,\x}|\gtrsim N^{-\beta_{2}}]=\mathrm{o}(1).
\end{align*}
The first is union bound, so we show the second. Since the sum has $\mathrm{O}(N^{101})$-many terms, it suffices to bound each summand as follows, where the first bound is just Chebyshev for $p\geq1$ to be determined:
\begin{align}
\mathbb{P}[|\mathbf{D}^{N,\wedge}_{\t,\x}|\gtrsim N^{-\beta_{2}}]\lesssim N^{2p\beta_{2}}\E|\mathbf{D}^{N,\wedge}_{\t,\x}|^{2p}\lesssim_{D} N^{-D}.\label{eq:uq1}
\end{align}
Since the heat kernel $\mathbf{H}^{N}$ for the $\mathscr{T}_{N}$-semigroup satisfies the usual pointwise bound {\small$\mathbf{H}^{N}_{\s,\t,\x,\y}\lesssim N^{-1}|\t-\s|^{-1/2}$} for $\s\leq\t$ (see Proposition \ref{prop:hk}), standard Burkholder-Davis-Gundy-type martingale inequalities, i.e. the ones that show the standard Bertini-Giacomin-style bounds in \cite{BG}, give the following for $\t=\mathrm{O}(1)$ (see Lemmas 7.6 and A.4 in \cite{Y23}, for example):
\begin{align}
\{\E|{\textstyle\int_{0}^{\t}}\mathrm{e}^{[\t-\s]\mathscr{T}_{N}}[\mathbf{D}^{N,\wedge}_{\s,\cdot}\d\xi^{N}_{\s,\cdot}]_{\x}|^{2p}\}^{1/p}&\lesssim_{p}{\textstyle\int_{0}^{\t}}|\t-\s|^{-\frac12}\sup_{\y\in\mathbb{T}_{N}}\{\E|\mathbf{D}^{N,\wedge}_{\s,\y}|^{2p}\}^{1/p}\d\s.\label{eq:uq1b}
\end{align}
(The $p^{-1}$-th power of the $2p$-th moment is perhaps more naturally interpreted as the squared $2p$-th norm with respect to randomness. This type of estimate is standard in the literature, so we omit the details.) Since the RHS of the previous display is independent of the $\x$-variable on the LHS, we can put a supremum over $\x\in\mathbb{T}_{N}$ on the LHS and the bound remains true. If we combine this with the deterministic bound $|\Phi^{\wedge}|\lesssim N^{-\beta/2}$, then we get 
\begin{align}
\sup_{\x\in\mathbb{T}_{N}}\{\E|\mathbf{D}^{N,\wedge}_{\t,\x}|^{2p}\}^{\frac{1}{p}}&\lesssim_{p}{\textstyle\int_{0}^{\t}}|\t-\s|^{-\frac12}\sup_{\y\in\mathbb{T}_{N}}\{\E|\mathbf{D}^{N,\wedge}_{\s,\y}|^{2p}\}^{1/p}\d\s+N^{-\beta}.\label{eq:uq1c}
\end{align}
Gronwall now implies the estimate {\small$\E|\mathbf{D}^{N,\wedge}_{\t,\x}|^{2p}\lesssim_{p}N^{-p\beta}\exp\{C_{p}\int_{0}^{\t}|\t-\s|^{-1/2}\d\s\}\lesssim N^{-p\beta}$}. If $\beta_{2}=\frac13\beta$, then {\small$N^{p\beta_{2}}\E|\mathbf{D}^{N,\wedge}_{\t,\x}|^{2p}\lesssim_{p}N^{-p\beta/6}$} for any $p\geq1$. Choosing $p$ large gives the second bound in \eqref{eq:uq1}, so we are done. \qed
\subsection{Proof of Proposition \ref{prop:qconvergence}}\label{subsection:qconvergence}
We follow the strategy of tightness and identifying limit points. First, however, we give important bounds for $\mathbf{Q}^{N}$; these facilitate both the proof of tightness and the proof of Proposition \ref{prop:tst=1}. 
\subsubsection{A priori estimates}
Recall that $\mathbf{Q}^{N}$ solves an exact microscopic version of \eqref{eq:she}; see \eqref{eq:qheatIa}. As in the proof of Proposition \ref{prop:uq} (see \eqref{eq:uq1b}-\eqref{eq:uq1c}),  we first have the standard Bertini-Giacomin-style bound \cite{BG}
\begin{align*}
\sup_{\x\in\mathbb{T}_{N}}\{\E|\mathbf{Q}^{N}_{\t,\x}|^{2p}\}^{\frac{1}{p}}&\lesssim_{p}\sup_{\x\in\mathbb{T}_{N}}\{\E|\mathbf{Z}^{N}_{0,\x}|^{2p}\}^{\frac{1}{p}}+{\textstyle\int_{0}^{\t}}|\t-\s|^{-\frac12}\sup_{\y\in\mathbb{T}_{N}}\{\E|\mathbf{Q}^{N}_{\s,\y}|^{2p}\}^{1/p}\d\s.
\end{align*}
By assumption, the first term on the RHS is $\mathrm{O}(1)$. By Gronwall, as in after \eqref{eq:uq1c}, we have {\small$\E|\mathbf{Q}^{N}_{\t,\x}|^{2p}\lesssim_{p}1$} for all $(\t,\x)\in[0,1]\times\mathbb{T}_{N}$. We can also follow the net-and-Chebyshev argument leading up to \eqref{eq:uq1} verbatim to get that with very high probability, we have the following for any fixed $\delta>0$:
\begin{align}
\|\mathbf{Q}^{N}\|_{\mathrm{L}^{\infty}([0,1]\times\mathbb{T}_{N})}\lesssim_{\delta} N^{\delta}.\label{eq:quniform}
\end{align}
Let us now record a high probability lower bound for $\mathbf{Q}^{N}$; it is just Lemma 7.2 in \cite{Y23}.
\begin{lemma}\label{lemma:qlower}
\fsp Fix any $\delta>0$. There exists $\e_{\delta}>0$ (depending only on $\delta$) so that with probability at least $1-\delta$, we have {\small$\mathbf{Q}^{N}_{\t,\x}\geq\e_{\delta}$} simultaneously for all $(\t,\x)\in[0,1]\times\mathbb{T}_{N}$.
\end{lemma}
\subsubsection{Tightness}
We have already shown {\small$\E|\mathbf{Q}^{N}_{\t,\x}|^{2p}\lesssim_{p}1$} for all $(\t,\x)\in[0,1]\times\mathbb{T}_{N}$ and $p\geq1$ (see the paragraph before \eqref{eq:quniform}). Thus, as in Proposition 3.2 and Corollary 3.3 in \cite{DT}, it suffices to show the following to prove tightness of $\mathbf{Q}^{N}$ in $\mathscr{D}([0,1],\mathscr{C}([0,1]))$ for all $\t,\mathfrak{t}\geq0$ and $\x,\y\in\mathbb{T}_{N}$ and $p\geq1$ and $\zeta>0$:
\begin{align}
\E|\mathbf{Q}^{N}_{\t,\x}-\mathbf{Q}^{N}_{\t,\y}|^{2p}&\lesssim_{p}N^{-p+p\zeta}|\x-\y|^{p-p\zeta},\label{eq:qtight1}\\
\E|\mathbf{Q}^{N}_{\t+\mathfrak{t},\x}-\mathbf{Q}^{N}_{\t,\x}|^{2p}&\lesssim_{p}\mathfrak{t}^{\frac{p}{2}-p\zeta}+N^{-p+p\zeta}.\label{eq:qtight2}
\end{align}
The proof of \eqref{eq:qtight1}-\eqref{eq:qtight2} are the exact same as the proof of the estimates in Proposition 3.2 of \cite{DT}; indeed, all we require are the heat kernel estimates in Proposition \ref{prop:hk}. (The point is that $\mathbf{Q}^{N}$ solves an exact microscopic version of \eqref{eq:she}; see \eqref{eq:qheatIa}. Also, although \cite{DT} does not include $\zeta$-corrections in the exponents, we can still use Kolmogorov continuity with \eqref{eq:qtight1}-\eqref{eq:qtight2} to deduce tightness.)

\subsubsection{Identification of limit points}
The limit \eqref{eq:she} can be characterized in terms of a martingale problem; see Lemma 3.16 in \cite{Y23}. Let us introduce this martingale problem to start.
\begin{definition}\label{definition:mgproblem}
\fsp We say a continuous process $\mathscr{Z}\in\mathscr{C}([0,1],\mathscr{C}(\mathbb{T}))$ satisfies the \emph{martingale problem} associated to \eqref{eq:she} if for any $\varphi\in\mathscr{C}^{\infty}(\mathbb{T})$, the processes below (as functions of $\t\in[0,1]$)
\begin{align*}
\mathscr{M}_{\t}&:={\textstyle\int_{\mathbb{T}}}\mathscr{Z}_{\t,\x}\varphi_{\x}\d\x-{\textstyle\int_{\mathbb{T}}}\mathscr{Z}_{0,\x}\varphi_{\x}\d\x-{\textstyle\int_{0}^{\t}\int_{\mathbb{T}}}\mathscr{Z}_{\s,\x}(\tfrac12\partial_{\x}^{2}-\overline{\mathfrak{d}}\partial_{\x})\varphi_{\x}\d\x\d\s,\\
\mathscr{N}_{\t}&:=\mathscr{M}_{\t}^{2}-\tfrac12{\textstyle\int_{0}^{\t}\int_{\mathbb{T}}}\mathscr{Z}_{\s,\x}^{2}\varphi_{\x}^{2}\d\x\d\s
\end{align*}
are martingales with respect to the filtration generated by $\mathscr{Z}$. (The operator {\small$\frac12\partial_{\x}^{2}-\overline{\mathfrak{d}}\partial_{\x}$} is the adjoint with respect to Lebesgue measure on $\mathbb{T}$ of the operator {\small$\frac12\partial_{\x}^{2}+\overline{\mathfrak{d}}\partial_{\x}$} in \eqref{eq:she}.)
\end{definition}
To identify subsequential limits of the tight sequence {\small$\mathbf{Z}^{N}_{\cdot,N\cdot}$}, again, by Lemma 3.16 in \cite{DT}, we must show that all subsequential limits solve the above martingale problem. By \eqref{eq:qheatIa}, for any $\varphi\in\mathscr{C}^{\infty}(\mathbb{T})$, the process 
\begin{align*}
\mathscr{M}^{N}_{\t}:=\tfrac{1}{N}\sum_{\x\in\mathbb{T}_{N}}\mathbf{Q}^{N}_{\t,\x}\varphi_{N^{-1}\x}-\tfrac{1}{N}\sum_{\x\in\mathbb{T}_{N}}\mathbf{Q}^{N}_{0,\x}\varphi_{N^{-1}\x}-{\textstyle\int_{0}^{\t}}\tfrac{1}{N}\sum_{\x\in\mathbb{T}_{N}}\mathscr{T}_{N}\mathbf{Q}^{N}_{\s,\x}\varphi_{N^{-1}\x}\d\s
\end{align*}
is a martingale. Since $\mathbf{Q}^{N}$ converges to some function $\mathscr{Z}$ in $\mathscr{D}([0,1],\mathscr{C}(\mathbb{T}))$ after rescaling in space from $\mathbb{T}_{N}$ to $\mathbb{T}\cap N^{-1}\Z$ (and linearly interpolating from $\mathbb{T}\cap N^{-1}\Z$ to $\mathbb{T}$), we know 
\begin{align*}
\tfrac{1}{N}\sum_{\x\in\mathbb{T}_{N}}\mathbf{Q}^{N}_{\t,\x}\varphi_{N^{-1}\x}-\tfrac{1}{N}\sum_{\x\in\mathbb{T}_{N}}\mathbf{Q}^{N}_{0,\x}\varphi_{N^{-1}\x}\to_{N\to\infty}{\textstyle\int_{\mathbb{T}}}\mathscr{Z}_{\t,\x}\varphi_{\x}\d\x-{\textstyle\int_{\mathbb{T}}}\mathscr{Z}_{0,\x}\varphi_{\x}\d\x
\end{align*}
just by Riemann sum approximations. Next, recalling {\small$\mathscr{T}_{N}=\frac12N^{2}\Delta-\overline{\mathfrak{d}}\grad^{\mathbf{X}}_{-1}$}, summation-by-parts shows that the adjoint of {\small$\mathscr{T}_{N}$} with respect to uniform measure on $\mathbb{T}_{N}$ is {\small$\mathscr{T}_{N}=\frac12N^{2}\Delta-\overline{\mathfrak{d}}\grad^{\mathbf{X}}_{+1}$} (i.e. the time-reversal of a random walk with drift to the left is a random walk with the same drift but to the right). So, by smoothness of $\varphi$, we get
\begin{align*}
{\textstyle\int_{0}^{\t}}\tfrac{1}{N}\sum_{\x\in\mathbb{T}_{N}}\mathscr{T}_{N}\mathbf{Q}^{N}_{\s,\x}\varphi_{N^{-1}\x}\d\s\to_{N\to\infty}{\textstyle\int_{0}^{\t}\int_{\mathbb{T}}}\mathscr{Z}_{\s,\x}(\tfrac12\partial_{\x}^{2}-\overline{\mathfrak{d}}\partial_{\x})\varphi_{\x}\d\x\d\s.
\end{align*}
Combining the above three displays gets $\mathscr{M}^{N}\to\mathscr{M}$ uniformly in $\t\in[0,1]$, where $\mathscr{M}$ is from Definition \ref{definition:mgproblem}. We use uniform control on {\small$\E|\mathscr{M}_{\t}^{N}|^{2p}$} for $p\geq10$, for example, to justify that the limit $\mathscr{M}$ is a martingale as well. (Such a bound follows by the moment bounds on $\mathbf{Q}^{N}$ in the paragraph before \eqref{eq:quniform}.)

Now, let $\langle\mathscr{M}^{N}\rangle$ be the predictable bracket process of $\mathscr{M}^{N}$, so that {\small$\mathscr{N}^{N}_{\t}:=|\mathscr{M}^{N}_{\t}|^{2}-\langle\mathscr{M}^{N}_{\t}\rangle$} is a martingale in $\t$. We claim that {\small$\mathscr{N}^{N}_{\t}\to\mathscr{N_{\t}}$} with $\mathscr{N}$ from Definition \ref{definition:mgproblem}, i.e. that {\small$\langle\mathscr{M}^{N}_{\t}\rangle\to\frac12\int_{0}^{\t}\int_{\mathbb{T}}\mathscr{Z}_{\s,\x}^{2}\varphi_{\x}^{2}\d\x\d\s$} uniformly in $\t$. Assuming this claim, uniform control on {\small$\E|\mathscr{M}_{\t}^{N}|^{2p}$} for $p\geq10$, again, shows that $\mathscr{N}$ must itself be a martingale as well. Thus, it remains to prove the claim. The proof of Proposition 1.5 in \cite{DT} shows 
\begin{align*}
\langle\mathscr{M}^{N}_{\t}\rangle={\int_{0}^{\t}}\tfrac{1+\mathrm{o}(1)}{2N}\sum_{\x\in\mathbb{T}_{N}}|\mathbf{Q}^{N}_{\s,\x}|^{2}\varphi_{N^{-1}\x}^{2}\d\s+\int_{0}^{\t}\tfrac{1+\mathrm{o}(1)}{N}\sum_{\x\in\mathbb{T}_{N}}\mathfrak{w}[\tau_{\x}\eta_{\s}]|\mathbf{Q}^{N}_{\s,\x}|^{2}\varphi_{N^{-1}\x}^{2}\d\s+\mathrm{o}(1),
\end{align*}
where $\mathfrak{w}$ is local and $\E^{0}\mathfrak{w}=0$. (We use results in \cite{DT} for jump length $1$. The models in this paper differ from such a model in \cite{DT} by clocks at each point in $\mathbb{T}_{N}$ of speed $\mathrm{O}(N)$; since only clocks of speed $\mathrm{O}(N^{2})$ have non-vanishing size, the contribution from our additional $\mathrm{O}(N)$ clocks can be absorbed into the $\mathrm{o}(1)$ terms. We also note that the $\mathcal{E}$-terms in \cite{DT} in the formula that we used to get the above display is the same as our $\mathfrak{w}$.) Riemann sums show that the first term on the RHS converges to {\small$\frac12\int_{0}^{\t}\int_{\mathbb{T}}\mathscr{Z}^{2}_{\s,\x}\varphi_{\x}^{2}\d\x\d\s$}. Thus, it suffices to show
\begin{align*}
\lim_{N\to\infty}\int_{0}^{\t}\tfrac{1}{N}\sum_{\x\in\mathbb{T}_{N}}\mathfrak{w}[\tau_{\x}\eta_{\s}]|\mathbf{Q}^{N}_{\s,\x}|^{2}\varphi_{N^{-1}\x}^{2}\d\s=0.
\end{align*}
(It is enough to do so pointwise in $\t$, since moment bounds for $\mathbf{Q}^{N}$ show that term inside the limit on the LHS, as a function of $\t$, is uniformly equicontinuous in $N$.) The proof of this estimate follows the proof of Lemma 2.5 in \cite{DT}, the entropy production estimate of Lemma \ref{lemma:eprod}, and Proposition \ref{prop:tst=1}. (In a nutshell, Lemma 2.5 in \cite{DT} and Lemma \ref{lemma:eprod} replaces {\small$\mathfrak{w}[\tau_{\x}\eta_{\s}]$} by {\small$\mathfrak{l}^{-1}(\eta_{\s,\x+1}+\ldots+\eta_{\s,\x+\mathfrak{l}})$} for $\mathfrak{l}=N^{1/3}$; any $\mathfrak{l}\leq N^{1/2-\delta}$ is allowed in the proof of Lemma 2.5 in \cite{DT}. But {\small$\mathfrak{l}^{-1}(\eta_{\s,\x+1}+\ldots+\eta_{\s,\x+\mathfrak{l}})=N^{-1/2}\mathfrak{l}^{-1}(\mathbf{h}^{N}_{\s,\x+\mathfrak{l}}-\mathbf{h}^{N}_{\s,\x})$}. By Proposition \ref{prop:tst=1}, with high probability, we have {\small$N^{-1/2}\mathfrak{l}^{-1}(\mathbf{h}^{N}_{\s,\x+\mathfrak{l}}-\mathbf{h}^{N}_{\s,\x})=\mathrm{O}(N^{-1/2}\mathfrak{l}^{-1}N^{1/2+2\e_{\mathrm{ap}}}\mathfrak{l}^{1/2})=\mathrm{O}(\mathfrak{l}^{-1/2})\to0$}.) So Proposition \ref{prop:qconvergence} holds. (We note that the proof of Proposition \ref{prop:tst=1} will depend only on \eqref{eq:quniform}, Lemma \ref{lemma:qlower}, and estimates from the proof of tightness; it will not use identification of limit points. In particular, there is no circular reasoning.) \qed
\subsection{Proof of Proposition \ref{prop:tst=1}}
We claim that
\begin{align}
\mathbb{P}\left[\mathcal{E}^{\mathbf{X}}(\mathfrak{t})\right]+\mathbb{P}\left[\mathcal{E}^{\mathbf{T}}(\mathfrak{t})\right]\lesssim N^{-\beta},\label{eq:hpreg}
\end{align}
where $\mathfrak{t}$ is any stopping time valued in $[0,1]$, where $\beta>0$ is fixed, and where, for $\mathfrak{l}_{\mathrm{reg}}=N^{1/3+\e_{\mathrm{reg}}}$ from \eqref{eq:tregX},
\begin{align*}
\mathcal{E}^{\mathbf{X}}(\mathfrak{t})&:=\left\{\sup_{1\leq|\mathfrak{l}|\leq2\mathfrak{l}_{\mathrm{reg}}}\frac{N^{\frac12}|\mathfrak{l}|^{-\frac12}\|\grad^{\mathbf{X}}_{\mathfrak{l}}\mathbf{U}^{N}\|_{\mathrm{L}^{\infty}([0,\mathfrak{t}]\times\mathbb{T}_{N})}}{1+\|\mathbf{U}^{N}\|_{\mathrm{L}^{\infty}([0,\mathfrak{t}]\times\mathbb{T}_{N})}^{2}}\gtrsim N^{\e_{\mathrm{ap}}}\right\},\\
\mathcal{E}^{\mathbf{T}}(\mathfrak{t})&:=\left\{\sup_{\s\in[0,1/N]}\frac{\{N^{-2}\vee\s\}^{-\frac14}\|\grad^{\mathbf{T}}_{\s}\mathbf{U}^{N}\|_{\mathrm{L}^{\infty}([0,\mathfrak{t}]\times\mathbb{T}_{N})}}{1+\|\mathbf{U}^{N}\|_{\mathrm{L}^{\infty}([0,\mathfrak{t}]\times\mathbb{T}_{N})}^{2}}\gtrsim N^{\e_{\mathrm{ap}}}\right\}.
\end{align*}
Intuitively, \eqref{eq:hpreg} shows that for $\mathfrak{t}=\mathfrak{t}_{\mathrm{stop}}$ (see \eqref{eq:tst}), the regularity estimates in the stopping times \eqref{eq:tap}, \eqref{eq:tregX}, and \eqref{eq:tregT} are ``self-propagating". (I.e., the bounds we get from \eqref{eq:hpreg} are better than the estimates directly implied by working before $\mathfrak{t}_{\mathrm{stop}}$.) Let us make this paragraph precise. In particular, we assume \eqref{eq:hpreg} and complete the proof of Proposition \ref{prop:tst=1}; we prove \eqref{eq:hpreg} afterwards.
\begin{proof}[Proof of Proposition \ref{prop:tst=1}]
We start with the union bound estimate
\begin{align}
\mathbb{P}[\mathfrak{t}_{\mathrm{stop}}<1]&\leq\mathbb{P}[\{\mathfrak{t}_{\mathrm{stop}}=\mathfrak{t}_{\mathrm{ap}}<1\}]+\mathbb{P}[\{\mathfrak{t}_{\mathrm{stop}}=\mathfrak{t}_{\mathrm{reg}}^{\mathrm{space}}<1\}]+\mathbb{P}[\{\mathfrak{t}_{\mathrm{stop}}=\mathfrak{t}_{\mathrm{reg}}^{\mathrm{time}}<1\}].\label{eq:tst=1unionbound}
\end{align}
We show that each term on the RHS of \eqref{eq:tst=1unionbound} is $\mathrm{o}(1)$. Take the first term and suppose we are on the event in this probability. By Proposition \ref{prop:uq}, \eqref{eq:quniform}, and Lemma \ref{lemma:qlower}, we get $\mathbf{U}^{N}=\mathbf{Q}^{N}+\mathrm{o}(1)$ and {\small$N^{-\e_{\mathrm{ap}}/2}\lesssim\mathbf{Q}^{N}\lesssim N^{\e_{\mathrm{ap}}/2}$} with high probability. This gives {\small$N^{-\e_{\mathrm{ap}}}\lesssim\mathbf{U}^{N}\lesssim N^{\e_{\mathrm{ap}}}$} with high probability. Since $\mathbf{Z}^{N}=\mathbf{U}^{N}$ until $\mathfrak{t}_{\mathrm{stop}}$ (see after \eqref{eq:qheatIa}), we get {\small$N^{-\e_{\mathrm{ap}}/2}\lesssim\mathbf{Z}^{N}\lesssim N^{\e_{\mathrm{ap}}/2}$} until time $\mathfrak{t}_{\mathrm{stop}}$ with high probability. Since we conditioned on $\mathfrak{t}_{\mathrm{stop}}=\mathfrak{t}_{\mathrm{ap}}$, with high probability, we have {\small$N^{-\e_{\mathrm{ap}}/2}\lesssim\mathbf{Z}^{N}\lesssim N^{\e_{\mathrm{ap}}/2}$} until time $\mathfrak{t}_{\mathrm{ap}}$. This is a contradiction since $\mathfrak{t}_{\mathrm{ap}}<1$. (Indeed, choose $\mathfrak{t}_{N}>0$ small enough so that $\mathfrak{t}_{\mathrm{ap}}+\mathfrak{t}_{N}<1$ and there are no jumps after $\mathfrak{t}_{\mathrm{ap}}$ and up to $\mathfrak{t}_{\mathrm{ap}}+\mathfrak{t}_{N}$; we can do this because $\mathfrak{t}_{\mathrm{ap}}$ is a stopping time. In this case, in the time-interval $\mathfrak{t}_{\mathrm{ap}}+[0,\mathfrak{t}_{N}]$, the function $\mathbf{Z}^{N}$ evolves continuously in time. Thus, if $\mathfrak{t}_{N}>0$ is small enough, then we maintain {\small$N^{-\e_{\mathrm{ap}}/2}\lesssim\mathbf{Z}^{N}\lesssim N^{\e_{\mathrm{ap}}/2}$} until time $\mathfrak{t}_{\mathrm{ap}}+\mathfrak{t}_{N}<1$. This contradicts the definition of $\mathfrak{t}_{\mathrm{ap}}$ in \eqref{eq:tap}.) Thus, the first term on the RHS of \eqref{eq:tst=1unionbound} is $\mathrm{o}(1)$. Showing the rest of the RHS of \eqref{eq:tst=1unionbound} is $\mathrm{o}(1)$ follows by the same argument; we use \eqref{eq:hpreg} to show that $\mathbf{U}^{N}$ has better regularity than what \eqref{eq:tregX} and \eqref{eq:tregT} ask for, transfer the same regularity estimates from $\mathbf{U}^{N}$ to $\mathbf{Z}^{N}$ until time $\mathfrak{t}_{\mathrm{stop}}$, use the fact that {\small$\mathfrak{t}^{\mathrm{space}}_{\mathrm{reg}}$} and {\small$\mathfrak{t}^{\mathrm{time}}_{\mathrm{reg}}$} are stopping times to let continuity propagate the estimates past {\small$\mathfrak{t}^{\mathrm{space}}_{\mathrm{reg}}$} and {\small$\mathfrak{t}^{\mathrm{time}}_{\mathrm{reg}}$}, respectively, and get a contradiction with high probability. 
\end{proof}
\subsubsection{Proof of \eqref{eq:hpreg}}
We first borrow the following from the proof of Proposition 3.2 in \cite{DT} (for any $\zeta>0$):
\begin{align}
|\grad^{\mathbf{X}}_{\mathfrak{l}}\mathrm{e}^{\t\mathscr{T}_{N}}[\mathbf{Z}^{N}_{0,\cdot}]_{\x}|\lesssim_{\zeta} N^{-\frac12+\zeta}|\mathfrak{l}|^{\frac12-\zeta} \quad\mathrm{and}\quad |\grad^{\mathbf{T}}_{-\s}\mathrm{e}^{\t\mathscr{T}_{N}}[\mathbf{Z}^{N}_{0,\cdot}]_{\x}|&\lesssim_{\zeta}\s^{\frac{1}{4}-\zeta}+N^{-\frac12}\lesssim N^{\zeta}\{N^{-2}\vee\s\}^{\frac14}. \label{eq:hpregbound1}
\end{align}
(In the proof of Proposition 3.2 in \cite{DT}, there were expectations, but since $\mathbf{Z}^{N}$ has deterministic initial data by assumption.) We now claim that with very high probability and for any $\zeta>0$, we have H\"{o}lder-$1/2$ regularity in space and H\"{o}lder-$1/4$ regularity in time for stochastic integrals (similar to \eqref{eq:she} itself but with additional $N^{\zeta}$ factors). In particular, we claim that
\begin{align*}
\sup_{\t\in[0,\mathfrak{t}]}\sup_{\x\in\mathbb{T}_{N}}|\grad^{\mathbf{X}}_{\mathfrak{l}}{\textstyle\int_{0}^{\t}}\mathrm{e}^{[\t-\r]\mathscr{T}_{N}}[\mathbf{U}^{N}_{\r,\cdot}\d\xi^{N}_{\r,\cdot}]_{\x}|&\lesssim_{\zeta}N^{-\frac12+\zeta}|\mathfrak{l}|^{\frac12}\{1+\|\mathbf{U}^{N}\|_{\mathrm{L}^{\infty}([0,\mathfrak{t}]\times\mathbb{T}_{N})}^{2}\}\\
\sup_{\t\in[0,\mathfrak{t}]}\sup_{\x\in\mathbb{T}_{N}}|\grad^{\mathbf{T}}_{-\s}{\textstyle\int_{0}^{\t}}\mathrm{e}^{[\t-\r]\mathscr{T}_{N}}[\mathbf{U}^{N}_{\r,\cdot}\d\xi^{N}_{\r,\cdot}]_{\x}|&\lesssim_{\zeta}N^{\zeta}\{N^{-2}\vee\s\}^{\frac14}\{1+\|\mathbf{U}^{N}\|_{\mathrm{L}^{\infty}([0,\mathfrak{t}]\times\mathbb{T}_{N})}^{2}\},
\end{align*}
where $\mathfrak{t}$ is the stopping time taken in \eqref{eq:hpreg}. These follow by Lemmas 6.4 and 6.8 in \cite{Y23}; we do not reproduce the lengthy details here. (Intuitively, we perform a high-moment estimate; because of the extra $N^{\zeta}$-factor on the RHS of each estimate, high moments and Chebyshev give very high probability statements.) We can now use a union bound to deduce that with very high probability, the same is true simultaneously over all $\mathfrak{l}\in[-N,N]\cap\Z$ and $\s$ in a discretization of $[0,1/N]$ with mesh $\mathrm{O}(N^{100})$ with very high probability. We can now use short-time continuity (on very small scales $\mathrm{O}(N^{-100})$) as in the proof of Proposition \ref{prop:uq} to get the estimates with very high probability simultaneously over all $\s\in[0,1/N]$; see Lemma A.6 in \cite{Y23}. Combining this with \eqref{eq:hpregbound1}, with Lemmas \ref{lemma:bgpreg} and \ref{lemma:timeregbgpterms}, and with \eqref{eq:err1regX}-\eqref{eq:err2regT} now gives \eqref{eq:hpreg}, so the proof is complete (the bounds we cited just control space-time regularity of each term in \eqref{eq:uheatIa}-\eqref{eq:uheatIe}). \qed
%
%
%
\section{Homogenization tools}\label{section:tools}
\subsection{Entropy production}
The goal of this subsection is controlling how close to $\mathbb{P}^{\sigma}$-measures the process {\small$\eta^{}_{\t}$} is at local scales after space-time averaging. Before we state the key ingredient, we need some notation. Recall that $\mathbb{P}^{\sigma}$ is product measure on $\{\pm1\}^{\mathbb{T}_{N}}$ such that {\small$\E^{\sigma}\eta_{\x}=\sigma$} for all $\x$. For any function $\mathfrak{f}:\{\pm1\}^{\mathbb{T}_{N}}\to\R$, we define its \emph{Dirichlet form} with respect to $\mathbb{P}^{\sigma}$ as 
\begin{align}
\mathfrak{D}^{\sigma}[\mathfrak{f}]:=\sum_{\x\in\mathbb{T}_{N}}\E^{\sigma}|\mathfrak{f}[\eta^{\x,\x+1}]-\mathfrak{f}[\eta]|^{2}=\sum_{\x\in\mathbb{T}_{N}}\E^{\sigma}|\mathscr{L}_{\x}\mathfrak{f}|^{2}=-\tfrac12\sum_{\x\in\mathbb{T}_{N}}\E^{\sigma}\mathfrak{f}\mathscr{L}_{\x}\mathfrak{f}\label{eq:dform}
\end{align}
where $\eta^{\x,\x+1}$ again means swapping {\small$\eta_{\x},\eta_{\x+1}$}, and the last identity is a standard Dirichlet form-$\mathrm{L}^{2}$ calculation; it follows since $\E^{\sigma}$ is invariant under swapping {\small$\eta_{\x},\eta_{\x+1}$}, and $\mathscr{L}_{\x}\mathfrak{f}$ picks up a sign in this swapping (see Appendix 1.9 in \cite{KL} for this computation.) Finally, for any probability density $\mathfrak{p}_{0}$ with respect to $\mathbb{P}^{0}$ and any deterministic $\t\geq0$, we let {\small$\mathfrak{p}_{\t}^{}$} be the probability density with respect to $\mathbb{P}^{0}$ for the law of {\small$\eta^{}_{\t}$}, assuming the law of {\small$\eta^{}_{0}$} is {\small$\mathfrak{p}_{0}\d\mathbb{P}^{0}$}.
\begin{lemma}\label{lemma:eprod}
\fsp For any probability density $\mathfrak{p}_{0}$ with respect to $\mathbb{P}^{0}$ and any deterministic $\mathfrak{t}=\mathrm{O}(1)$, we have 
\begin{align}
{\textstyle\int_{0}^{\t}}\mathfrak{D}^{0}[\sqrt{\mathfrak{p}^{}_{\s}}]\d\s&\lesssim N^{-2}\E^{0}[\mathfrak{p}_{0}\log\mathfrak{p}_{0}]+N^{-1}\mathfrak{t}\lesssim N^{-1}.\label{eq:eprod}
\end{align}
\end{lemma}
\begin{proof}
It is standard to check that $\E^{0}[\mathfrak{p}_{0}\log\mathfrak{p}_{0}]\lesssim N$, so the second estimate follows immediately. We prove the first estimate. By the Ito formula, we have 
\begin{align*}
\tfrac{\d}{\d\t}\E^{0}[\mathfrak{p}^{}_{\t}\log\mathfrak{p}^{}_{\t}]&=\E^{0}[\mathfrak{p}^{}_{\t}\mathscr{L}_{N}\log\mathfrak{p}^{}_{\t}]+\E^{0}[\mathfrak{p}^{}_{\t}\partial_{\t}\log\mathfrak{p}^{}_{\t}],
\end{align*}
where $\mathscr{L}^{}_{N}$ is the generator of {\small$\t\mapsto\eta_{\t}$}. The second term on the far RHS of the above display equals $\E^{0}[\partial_{\t}\mathfrak{p}^{}_{\t}]=\frac{\d}{\d\t}\E^{0}\mathfrak{p}^{}_{\t}=0$, since $\mathfrak{p}^{}_{\t}$ is a probability density with respect to $\mathbb{P}^{0}$. Thus, the previous display equals
\begin{align*}
&\sum_{\x\in\mathbb{T}_{N}}\E^{0}\left[\mathbf{1}_{\substack{\eta_{\x}=-1\\\eta_{\x+1}=1}}\left(\tfrac12N^{2}+\tfrac12N^{\frac32}+\tfrac12N\mathfrak{d}[\tau_{\x}\eta]\right)\mathfrak{p}^{}_{\t}[\eta]\left\{\log\mathfrak{p}^{}_{\t}[\eta^{\x,\x+1}]-\log\mathfrak{p}^{}_{\t}[\eta]\right\}\right]\\
&+\sum_{\x\in\mathbb{T}_{N}}\E^{0}\left[\mathbf{1}_{\substack{\eta_{\x}=1\\\eta_{\x+1}=-1}}\left(\tfrac12N^{2}-\tfrac12N^{\frac32}-\tfrac12N\mathfrak{d}[\tau_{\x}\eta]\right)\mathfrak{p}^{}_{\t}[\eta]\left\{\log\mathfrak{p}^{}_{\t}[\eta^{\x,\x+1}]-\log\mathfrak{p}^{}_{\t}[\eta]\right\}\right].
\end{align*}
We now use the inequality $a(\log b-\log a)\leq 2a^{1/2}(b^{1/2}-a^{1/2})$ (see Appendix 1.9 in \cite{KL}) and positivity of $N^{2}\pm N^{3/2}\pm N\mathfrak{d}[\tau_{\x}\eta]$ for $N$ large enough. This turns the previous display into
\begin{align*}
\tfrac{\d}{\d\t}\E^{0}[\mathfrak{p}_{\t}\log\mathfrak{p}_{\t}]&\leq2\sum_{\x\in\mathbb{T}_{N}}\E^{0}\left[\left(\tfrac12N^{2}+\tfrac12N^{\frac32}\mathbf{1}_{\substack{\eta_{\x}=-1\\\eta_{\x+1}=1}}-\tfrac12N^{\frac32}\mathbf{1}_{\substack{\eta_{\x}=1\\\eta_{\x+1}=-1}}\right)\mathfrak{p}^{}_{\t}[\eta]^{\frac12}\left(\mathfrak{p}^{}_{\t}[\eta^{\x,\x+1}]^{\frac12}-\mathfrak{p}^{}_{\t}[\eta]^{\frac12}\right)\right]\\
&+2\sum_{\x\in\mathbb{T}_{N}}\E^{0}\left[\left(\mathbf{1}_{\substack{\eta_{\x}=-1\\\eta_{\x+1}=1}}\tfrac12N\mathfrak{d}[\tau_{\x}\eta]-\mathbf{1}_{\substack{\eta_{\x}=-1\\\eta_{\x+1}=1}}\tfrac12N^{}\mathfrak{d}[\tau_{\x}\eta]\right)\mathfrak{p}^{}_{\t}[\eta]^{\frac12}\left(\mathfrak{p}^{}_{\t}[\eta^{\x,\x+1}]^{\frac12}-\mathfrak{p}^{}_{\t}[\eta]^{\frac12}\right)\right].
\end{align*}
The first line of the previous display is $\leq-\kappa N^{2}\mathfrak{D}^{0}[\sqrt{\mathfrak{p}^{}_{\t}}]$ for some $\kappa>0$ fixed; this is because the generator $\mathscr{L}_{N}$ without $\mathfrak{d}$ has $\mathbb{P}^{0}$ as an invariant measure (see Appendix 1.9 in \cite{KL} for this standard computation). By Schwarz, the second line is 
\begin{align*}
&\lesssim\delta^{-1}\sum_{\x\in\mathbb{T}_{N}}\E^{0}\mathfrak{p}^{}_{\t}+\delta\sum_{\x\in\mathbb{T}_{N}}N^{2}\E^{0}\left[\left(\mathfrak{p}^{}_{\t}[\eta^{\x,\x+1}]^{\frac12}-\mathfrak{p}^{}_{\t}[\eta]^{\frac12}\right)^{2}\right]\lesssim \delta^{-1}N+\delta N^{2}\mathfrak{D}^{0}[\sqrt{\mathfrak{p}^{}_{\t}}] \quad \text{for any} \ \delta>0.
\end{align*}
We deduce $\frac{\d}{\d\t}\E^{0}[\mathfrak{p}^{}_{\t}\log\mathfrak{p}^{}_{\t}]\leq -K_{1}N^{2}\mathfrak{D}^{0}[\sqrt{\mathfrak{p}^{}_{\t}}]+K_{2}\delta^{-1}N+K_{2}\delta N^{2}\mathfrak{D}^{0}[\sqrt{\mathfrak{p}^{}_{\t}}]$; here $K_{1},K_{2}>0$ are fixed. If $\delta$ is small, then for some $K>0$ fixed, we have
\begin{align*}
\E^{0}[\mathfrak{p}^{}_{\mathfrak{t}}\log\mathfrak{p}^{}_{\mathfrak{t}}]-\E^{0}[\mathfrak{p}^{}_{0}\log\mathfrak{p}^{}_{0}]\leq -KN^{2}{\textstyle\int_{0}^{\mathfrak{t}}}\mathfrak{D}^{0}[\sqrt{\mathfrak{p}^{}_{\mathfrak{s}}}]\d\s+\mathrm{O}(N).
\end{align*}
\eqref{eq:eprod} follows by rearranging and using the standard bound $\E^{0}[\mathfrak{p}_{\mathfrak{t}}\log\mathfrak{p}_{\mathfrak{t}}]\geq0$ (see Appendix 1.8 of \cite{KL}).
\end{proof}
We now show that local statistics in the bulk are close to $\mathbb{P}^{\sigma}$ measures after large-scale space-time averaging; it will be a consequence of Lemma \ref{lemma:eprod} and standard one-block ideas in \cite{GPV}. Before we state these results, for convenience, we define a \emph{sub-interval} in $\mathbb{T}_{N}$ to be a subset of the form $\{\x,\x+1,\ldots,\x+\mathfrak{l}\}$ for some $\mathfrak{l}\geq0$. Also, we must define the following probability measures (which are generally known as ``canonical ensembles").
\begin{definition}\label{definition:can}
\fsp Fix any subset $\mathbb{I}\subseteq\mathbb{T}_{N}$ and $\sigma\in[-1,1]$. We define $\mathbb{P}^{\sigma,\mathbb{I}}$ to be the uniform measure on 
\begin{align}
\left\{\eta\in\{\pm1\}^{\mathbb{I}}: |\mathbb{I}|^{-1}\sum_{\x\in\mathbb{I}}\eta_{\x}=\sigma\right\}\subseteq\{\pm1\}^{\mathbb{I}}.\label{eq:canonicalhyperplane}
\end{align}
(We call this the \emph{canonical ensemble} on $\mathbb{I}$ with density $\sigma$.) We let $\E^{\sigma,\mathbb{I}}$ be expectation with respect to $\mathbb{P}^{\sigma,\mathbb{I}}$.
\end{definition}
\begin{lemma}\label{lemma:localreduc}
\fsp Suppose $\mathfrak{f}:\{\pm1\}^{\mathbb{T}_{N}}\to\R$ is so that for any $\eta$, $\mathfrak{f}[\eta]$ depends only on $\eta_{\y}$ for $\y$ in some sub-interval $\mathbb{I}\subseteq\mathbb{T}_{N}$. For any $\kappa>0$ and $0<\mathfrak{t}=\mathrm{O}(1)$, we have 
\begin{align}
\int_{0}^{\mathfrak{t}}\tfrac{1}{N}\sum_{\y\in\mathbb{T}_{N}}\E[|\mathfrak{f}[\tau_{\y}\eta^{}_{\s}]|]\d\s\lesssim \kappa^{-1}N^{-2}|\mathbb{I}|^{3}+\kappa^{-1}\sup_{\sigma\in[-1,1]}\log\E^{\sigma,\mathbb{I}}\mathrm{e}^{\kappa|\mathfrak{f}[\eta]|}.\label{eq:localreduc}
\end{align}
In particular, if $\kappa^{-1}\lesssim\|\mathfrak{f}\|_{\infty}\lesssim\kappa^{-1}$, where $\|\|_{\infty}$ means sup-norm on $\{\pm1\}^{\mathbb{T}_{N}}$, then 
\begin{align}
\int_{0}^{\mathfrak{t}}\tfrac{1}{N}\sum_{\y\in\mathbb{T}_{N}}\E[|\mathfrak{f}[\tau_{\y}\eta^{}_{\s}]|]\d\s\lesssim  \kappa^{-1}N^{-2}|\mathbb{I}|^{3}+\sup_{\sigma\in[-1,1]}\E^{\sigma,\mathbb{I}}|\mathfrak{f}[\eta]|.\label{eq:localreducII}
\end{align}
\end{lemma}
\begin{proof}
\eqref{eq:localreducII} follows by \eqref{eq:localreduc} since if $\kappa|\mathfrak{f}[\eta]|\lesssim1$, then {\small$\log\E^{\sigma,\mathbb{I}}\exp\{\kappa|\mathfrak{f}[\eta]|\}\leq\log[1+\mathrm{O}(\kappa\E|\mathfrak{f}[\eta]|)]\lesssim\kappa\E|\mathfrak{f}[\eta]|$}. The proof of \eqref{eq:localreduc} will follow from the classical one-block scheme \cite{GPV} plus a log-Sobolev inequality for spin-swaps \cite{Yau}. Recall that $\mathfrak{p}_{\s}$ is the probability density for the law of $\eta_{\s}$ with respect to $\mathbb{P}^{0}$. We have 
\begin{align*}
\int_{0}^{1}\tfrac{1}{N}\sum_{\y\in\mathbb{T}_{N}}\E[|\mathfrak{f}[\tau_{\y}\eta^{}_{\s}]|]\d\s&=\int_{0}^{1}\tfrac{1}{N}\sum_{\y\in\mathbb{T}_{N}}\E^{0}[\mathfrak{p}^{}_{\s}|\mathfrak{f}[\tau_{\y}\eta]|]\d\s=\int_{0}^{1}\tfrac{1}{N}\sum_{\y\in\mathbb{T}_{N}}\E^{0}[\Pi^{\tau_{\y}\mathbb{I}}\mathfrak{p}^{}_{\s}|\mathfrak{f}[\tau_{\y}\eta]|]\d\s,
\end{align*}
where {\small$\Pi^{\tau_{\y}\mathbb{I}}$} is projection onto the marginal distribution of $\eta_{\x}$ for $\x+\y\in\mathbb{I}$. Indeed, $\mathfrak{f}[\tau_{\y}\eta]$ depends only on these spins. Now, let {\small$\Pi^{\sigma,\mathbb{I}}$} be projection {\small$\Pi^{\mathbb{I}}$} composed with conditioning on \eqref{eq:canonicalhyperplane}. By the law of total probability, 
\begin{align*}
\E^{0}[\Pi^{\tau_{\y}\mathbb{I}}\mathfrak{p}^{}_{\s}|\mathfrak{f}[\tau_{\y}\eta]|]&=\sum_{\sigma\in[-1,1]}p_{\sigma}\E^{\sigma,\tau_{\y}\mathbb{I}}[(\Pi^{\sigma,\tau_{\y}\mathbb{I}}\mathfrak{p}^{}_{\s})|\mathfrak{f}[\tau_{\y}\eta]|],
\end{align*}
where the sum is only over some finite set of possible values $\sigma$ and $p_{\sigma}$ sum over $\sigma$ to $1$. By the entropy inequality (see Lemma \ref{lemma:entropyinequality}), the RHS of the previous display is
\begin{align*}
\lesssim\sum_{\sigma\in[-1,1]}p_{\sigma}\kappa^{-1}\E^{\sigma,\tau_{\y}\mathbb{I}}[(\Pi^{\sigma,\tau_{\y}\mathbb{I}}\mathfrak{p}^{}_{\s})\log(\Pi^{\sigma,\tau_{\y}\mathbb{I}}\mathfrak{p}^{}_{\s})]+\sum_{\sigma\in[-1,1]}p_{\sigma}\kappa^{-1}\log\E^{\sigma,\tau_{\y}\mathbb{I}}\mathrm{e}^{\kappa|\mathfrak{f}[\tau_{\y}\eta]|]}, \quad \text{for any} \ \kappa>0.
\end{align*}
We can take supremum over $\sigma$ of the second term in the previous display instead of averaging against {\small$p_{\sigma}$}, giving the last term in \eqref{eq:localreduc} after we note that {\small$\E^{\sigma,\tau_{\y}\mathbb{I}}\mathrm{e}^{\kappa|\mathfrak{f}[\tau_{\y}\eta]|]}$} is independent of the shift $\y$ (by translation). For the first term, we apply the log-Sobolev inequality for spin-swaps in \cite{Yau} (the LSI constant is $\lesssim|\mathbb{I}|^{2}$); this gives
\begin{align*}
\E^{\sigma,\tau_{\y}\mathbb{I}}[(\Pi^{\sigma,\tau_{\y}\mathbb{I}}\mathfrak{p}^{}_{\s})\log(\Pi^{\sigma,\tau_{\y}\mathbb{I}}\mathfrak{p}^{}_{\s})]&\lesssim|\mathbb{I}|^{2}\sum_{\x,\x+1\in\tau_{\y}\mathbb{I}}\E^{\sigma,\tau_{\y}\mathbb{I}}|\mathscr{L}_{\x}\sqrt{\Pi^{\sigma,\tau_{\y}\mathbb{I}}\mathfrak{p}^{}_{\s}}|^{2}.
\end{align*}
When we multiply by $p_{\sigma}$ and sum over all $\sigma$, then the law of total expectation gives
\begin{align*}
\kappa^{-1}\sum_{\sigma\in[-1,1]}p_{\sigma}|\mathbb{I}|^{2}\sum_{\x,\x+1\in\tau_{\y}\mathbb{I}}\E^{\sigma,\tau_{\y}\mathbb{I}}|\mathscr{L}_{\x}\sqrt{\Pi^{\sigma,\tau_{\y}\mathbb{I}}\mathfrak{p}^{}_{\s}}|^{2}&\lesssim\kappa^{-1} |\mathbb{I}|^{2}\sum_{\x,\x+1\in\tau_{\y}\mathbb{I}}\E^{0}|\mathscr{L}_{\x}\sqrt{\Pi^{\tau_{\y}\mathbb{I}}\mathfrak{p}^{}_{\s}}|^{2}.
\end{align*}
By convexity of the functional on the RHS of the previous display in the parameter {\small$\Pi^{\tau_{\y}\mathbb{I}}\mathfrak{p}^{}_{\s}$} (see Appendix 1.10 in \cite{KL}), we can drop the averaging operator $\Pi^{\tau_{\y}\mathbb{I}}$. In particular, we get
\begin{align*}
\tfrac{1}{N}\sum_{\y\in\mathbb{T}_{N}}\kappa^{-1}|\mathbb{I}|^{2}\sum_{\x,\x+1\in\tau_{\y}\mathbb{I}}\E^{0}|\mathscr{L}_{\x}\sqrt{\Pi^{\tau_{\y}\mathbb{I}}\mathfrak{p}^{}_{\s}}|^{2}&\leq\tfrac{1}{N}\sum_{\y\in\mathbb{T}_{N}}\kappa^{-1}|\mathbb{I}|^{2}\sum_{\x,\x+1\in\tau_{\y}\mathbb{I}}\E^{0}|\mathscr{L}_{\x}\sqrt{\mathfrak{p}^{}_{\s}}|^{2}\lesssim \kappa^{-1}\tfrac{1}{N}|\mathbb{I}|^{3}\mathfrak{D}^{0}[\sqrt{\mathfrak{p}^{}_{\s}}].
\end{align*}
(The additional $|\mathbb{I}|$ factor comes from a redundancy in the number of times we sum over any given bond $\x,\x+1\in\mathbb{T}_{N}$.) If we integrate over $\s\in[0,1]$ and use Lemma \ref{lemma:eprod}, then the RHS of the above display is $\lesssim$ the first term on the RHS of \eqref{eq:localreduc}. This completes the proof.
\end{proof}
\subsection{Kipnis-Varadhan-type estimate for bulk statistics}\label{subsection:kv}
We now prove an estimate which shows that space-time averages of local \emph{fluctuating} statistics (i.e. mean-zero with respect to the canonical measures from Definition \ref{definition:can}) exhibit square-root cancellation. Because we look at local statistics, we will consider an exclusion process on a sub-interval in $\mathbb{T}_{N}$ with its own periodic boundary conditions.

In particular, take a sub-interval $\mathbb{L}$ of length $2\ell+1$, i.e. of the form $\y+\{-\ell,\ldots,\ell\}$. (We will always take $\ell=N^{\upsilon}$ for some $0<\upsilon\leq1/3+\beta$ for $\beta>0$ small but fixed.) Throughout this subsection, addition between points, if at least one belongs to $\mathbb{L}$, will be taken with respect to periodic boundary conditions on $\mathbb{L}$, i.e. identifying $\y+\ell+1=\y-\ell$. 

We introduce the aforementioned exclusion process  but now on $\mathbb{L}$ with periodic boundary conditions. Let {\small$\t\mapsto\eta^{\mathbb{L}}_{\t}$} be the process valued in $\{\pm1\}^{\mathbb{L}}$ with generator {\small$\mathscr{L}^{\mathbb{L}}_{N}:=\mathscr{L}^{\mathbb{L}}_{N,\mathrm{S}}+\mathscr{L}^{\mathbb{L}}_{N,\mathrm{A}}$}, where
\begin{align}
\mathscr{L}_{N,\mathrm{S}}^{\mathbb{L}}&:=\tfrac12N^{2}\sum_{\x\in\mathbb{L}}\mathscr{L}_{\x}\\
\mathscr{L}_{N,\mathrm{A}}^{\mathbb{L}}&:=\tfrac12N^{\frac32}\sum_{\x\in\mathbb{L}}\left(\mathbf{1}_{\substack{\eta_{\x}=-1\\\eta_{\x+1}=1}}-\mathbf{1}_{\substack{\eta_{\x}=1\\\eta_{\x+1}=-1}}\right)\left(1+N^{-\frac12}\mathfrak{d}[\tau_{\x}\eta]\right)\mathscr{L}_{\x}.
\end{align}
We now give the main estimate. Recall $\mathbb{P}^{\sigma,\mathbb{I}}$ and $\E^{\sigma,\mathbb{I}}$ from Definition \ref{definition:can}.
\begin{prop}\label{prop:kv}
\fsp Suppose {\small$\{\mathfrak{f}_{i}\}_{i=1}^{n}$} is a collection of functions $\{\pm1\}^{\mathbb{L}}\to\R$ satisfying the following properties.
\begin{enumerate}
\item For any $\mathfrak{f}_{i}$, we have $\E^{\sigma,\mathbb{L}}\mathfrak{f}_{i}=0$ for all $\sigma\in[-1,1]$ and $\sup_{i}|\mathfrak{f}_{i}[\eta]|\lesssim1$ for all $\eta\in\{\pm1\}^{\mathbb{L}}$.
\item There exist pairwise disjoint sub-intervals {\small$\{\mathbb{I}_{i}\}_{i=1}^{n}$} in $\mathbb{L}$ such that $\mathfrak{f}_{i}[\eta]$ depends only on $\eta_{\z}$ for $\z\in\mathbb{I}_{i}$.
\end{enumerate}
Now, fix $0<\tau\lesssim N^{-4/3}$ and assume that {\small$|\mathbb{L}|\lesssim N^{1/3+\e_{\mathrm{reg}}}$}. Assume that the distribution of {\small$\eta^{\mathbb{L}}_{\s}$} at time $\s=0$ is given by $\mathbb{P}^{\sigma,\mathbb{L}}$ for some $\sigma\in[-1,1]$. For any $\delta>0$, we have
\begin{align}
\E\left(\left\{\tau^{-1}\int_{0}^{\tau}\tfrac{1}{n}\sum_{i=1}^{n}\mathfrak{f}_{i}[\eta^{\mathbb{L}}_{\s}]\d\s\right\}^{2}\right)\lesssim_{\delta} N^{\delta}N^{-2}\tau^{-1}n^{-1}\sup_{i=1,\ldots,n}|\mathbb{I}_{i}|^{2}.\label{eq:kv}
\end{align}
\end{prop}
The factor $\tau^{-1}n^{-1}$ comes from square-root cancellation in space and time. The factor $N^{-2}$ comes from the fact that the speed of the particle system is order $N^{2}$, and a faster speed means more cancellations. We also note that the factor of $N^{\delta}$ is completely technical and harmless since $\delta>0$ is small.

The main difficulty is in the fact that $\eta^{\mathbb{L}}$ is not stationary, nor do we know anything about its invariant measure. Indeed, if $\eta^{\mathbb{L}}$ had $\mathbb{P}^{\sigma,\mathbb{L}}$ as an invariant measure, then Proposition \ref{prop:kv} would follow immediately from Corollary 1 in \cite{GJ15}. However, the invariant measure assumption is crucial in the proof of Corollary 1 in \cite{GJ15}. Let us also note that if {\small$\tau=N^{-4/3}$} and {\small$|\mathbb{L}|=N^{1/3+\e_{\mathrm{reg}}}$}, the sum of all of the clocks in the $\mathfrak{d}[\cdot]$-part of the $\eta^{\mathbb{L}}$ dynamics is heuristically a Poisson distribution of parameter {\small$\mathrm{O}(N|\mathbb{L}|\tau)=\mathrm{O}(N^{\e_{\mathrm{reg}}})$}. Since we need $\e_{\mathrm{reg}}>0$, this is divergent. Thus, Proposition \ref{prop:kv} is \emph{far from being} a dynamically perturbative result, and new technical innovations are necessary. (Indeed, the aforementioned Poisson clock speed parameter is barely $\gg1$, but this is still far from being what we want, because \eqref{eq:kv} gives an estimate that is much more refined than just $\mathrm{o}(1)$.)

To start, we introduce a process obtained by setting $\mathfrak{d}$ equal to zero. We call it a ``free process" since $\mathfrak{d}$ includes the (interesting) interactions in general. In particular, we define {\small$\mathscr{L}^{\mathrm{free},\mathbb{L}}_{N}:=\mathscr{L}^{\mathbb{L}}_{N,\mathrm{S}}+\mathscr{L}^{\mathrm{free},\mathbb{L}}_{N,\mathrm{A}}$}, where
\begin{align*}
\mathscr{L}_{N,\mathrm{A}}^{\mathrm{free},\mathbb{L}}&:=\tfrac12N^{\frac32}\sum_{\x\in\mathbb{L}}\left(\mathbf{1}_{\substack{\eta_{\x}=-1\\\eta_{\x+1}=1}}-\mathbf{1}_{\substack{\eta_{\x}=1\\\eta_{\x+1}=-1}}\right)\mathscr{L}_{\x}.
\end{align*}
%
%
%
\subsubsection{Preliminary estimates}
The general strategy to prove \eqref{eq:kv} is based on the Ito formula as in Appendix 1.6 of \cite{KL} for stationary processes. However, because $\eta^{\mathbb{L}}$ is not a stationary process in Proposition \ref{prop:kv} for general driving functions $\mathfrak{d}[\cdot]$, we require a stability estimate for the law of $\eta^{\mathbb{L}}$ with respect to the initial distribution $\mathbb{P}^{\sigma,\mathbb{L}}$. Before we state the estimate, we need some notation, which is essentially that before Lemma \ref{lemma:eprod} but localized to $\mathbb{L}$. For any $\t\geq0$, we let {\small$\mathfrak{p}^{\mathbb{L}}_{\t}$} be the probability density with respect to $\mathbb{P}^{\sigma,\mathbb{L}}$ for the law of {\small$\eta^{\mathbb{L}}_{\t}$}. (Note that {\small$\mathfrak{p}^{\mathbb{L}}_{0}\equiv1$} by assumption in Proposition \ref{prop:kv}.) We also define the localized Dirichlet form as 
\begin{align}
\mathfrak{D}^{\sigma,\mathbb{L}}[\mathfrak{f}]:=\sum_{\x\in\mathbb{L}}\E^{\sigma,\mathbb{L}}|\mathfrak{f}[\eta^{\x,\x+1}]-\mathfrak{f}[\eta]|^{2}\label{eq:dformloc}
\end{align}
for any function $\mathfrak{f}:\{\pm1\}^{\mathbb{L}}\to\R$, where $\x+1$ means addition with respect to $\mathbb{L}$ as a torus.
\begin{lemma}\label{lemma:lpprod}
\fsp Fix deterministic $p\geq1$ and $\tau\lesssim N^{-4/3}$. We have {\small$\E^{\sigma,\mathbb{L}}|\mathfrak{p}^{\mathbb{L}}_{\t}|^{p}\lesssim_{p}1$} uniformly over $\t\in[0,\tau]$.
\end{lemma}
\begin{remark}\label{remark:lpprod}
\fsp Before we show Lemma \ref{lemma:lpprod}, we note that the strategy used in the proof is based on differentiating {\small$\E^{\sigma,\mathbb{L}}|\mathfrak{p}^{\mathbb{L}}_{\t}|^{p}$} in $\t$. In particular, if $p=2$, it is a standard energy dissipation for parabolic equations. It also drastically simplifies for $p=2$, but our applications of Lemma \ref{lemma:lpprod} to prove Proposition \ref{prop:kv} need large $p=\mathrm{O}(1)$. (A main difficulty in the proof is that we must analyze {\small$\E^{\sigma,\mathbb{L}}\mathfrak{p}^{\mathbb{L}}_{\t}\mathscr{L}^{\mathrm{free},\mathbb{L}}_{N,\mathrm{A}}|\mathfrak{p}^{\mathbb{L}}_{\t}|^{p-1}$}. If $p=2$, we use anti-symmetry of {\small$\mathscr{L}^{\mathrm{free},\mathbb{L}}_{N,\mathrm{A}}$} to show this term is $0$. For general $p\geq1$, the aim is to exploit anti-symmetry in a somewhat similar way. We must also control the $\mathfrak{d}[\cdot]$-part of {\small$\mathscr{L}^{\mathbb{L}}_{N}$}. It has no invariance with respect to $\mathbb{P}^{\sigma,\mathbb{L}}$, but it can be done perturbatively, because its speed is $\mathrm{O}(N)$ and thus small enough compared to the leading-order speed $N^{2}$.) 
\end{remark}
\begin{proof}
Because $\mathrm{L}^{p}$-norms are monotone non-decreasing in $p$, it suffices to possibly enlarge $p$ so that $p$ is even. (This is not necessary, but it makes factorizing {\small$\alpha^{p/2}-\beta^{p/2}$} much more convenient in terms of formulas, and such factorizing is important for this argument.) By the Kolmogorov equation, we have {\small$\partial_{\t}\mathfrak{p}^{\mathbb{L}}_{\t}=(\mathscr{L}^{\mathbb{L}}_{N})^{\ast}\mathfrak{p}^{\mathbb{L}}_{\t}$}, where the star superscript denotes adjoint with respect to $\mathbb{P}^{\sigma,\mathbb{L}}$, so that
\begin{align*}
\tfrac{\d}{\d\t}\E^{\sigma,\mathbb{L}}|\mathfrak{p}^{\mathbb{L}}_{\t}|^{p}&=p\E^{\sigma,\mathbb{L}}|\mathfrak{p}^{\mathbb{L}}_{\t}|^{p-1}(\mathscr{L}^{\mathbb{L}}_{N})^{\ast}\mathfrak{p}^{\mathbb{L}}_{\t}=p\E^{\sigma,\mathbb{L}}\mathfrak{p}^{\mathbb{L}}_{\t}\mathscr{L}^{\mathbb{L}}_{N}|\mathfrak{p}^{\mathbb{L}}_{\t}|^{p-1}\\
&=p\E^{\sigma,\mathbb{L}}\mathfrak{p}^{\mathbb{L}}_{\t}\mathscr{L}^{\mathbb{L}}_{N,\mathrm{S}}|\mathfrak{p}^{\mathbb{L}}_{\t}|^{p-1}+p\E^{\sigma,\mathbb{L}}\mathfrak{p}^{\mathbb{L}}_{\t}\mathscr{L}^{\mathrm{free},\mathbb{L}}_{N,\mathrm{A}}|\mathfrak{p}^{\mathbb{L}}_{\t}|^{p-1}+p\E^{\sigma,\mathbb{L}}\mathfrak{p}^{\mathbb{L}}_{\t}\mathscr{L}^{\mathfrak{d},\mathbb{L}}_{N,\mathrm{A}}|\mathfrak{p}^{\mathbb{L}}_{\t}|^{p-1},
\end{align*}
where
\begin{align*}
\mathscr{L}_{N,\mathrm{A}}^{\mathfrak{d},\mathbb{L}}&:=\tfrac12N\sum_{\x\in\mathbb{L}}\left(\mathbf{1}_{\substack{\eta_{\x}=-1\\\eta_{\x+1}=1}}-\mathbf{1}_{\substack{\eta_{\x}=1\\\eta_{\x+1}=-1}}\right)\mathfrak{d}[\tau_{\x}\eta]\mathscr{L}_{\x}.
\end{align*}
We first claim the following computation holds for some $\kappa>0$ universal:
\begin{align*}
\E^{\sigma,\mathbb{L}}\mathfrak{p}^{\mathbb{L}}_{\t}\mathscr{L}^{\mathbb{L}}_{N,\mathrm{S}}|\mathfrak{p}^{\mathbb{L}}_{\t}|^{p-1}&=\tfrac12N^{2}{\textstyle\sum_{\x\in\mathbb{L}}}\E^{\sigma,\mathbb{L}}\left[\mathfrak{p}^{\mathbb{L}}_{\t}[\eta]\left\{|\mathfrak{p}^{\mathbb{L}}_{\t}[\eta^{\x,\x+1}]|^{p-1}-|\mathfrak{p}^{\mathbb{L}}_{\t}[\eta]|^{p-1}\right\}\right]\\
&=-\tfrac14N^{2}{\textstyle\sum_{\x\in\mathbb{L}}}\E^{\sigma,\mathbb{L}}\left[\left\{\mathfrak{p}^{\mathbb{L}}_{\t}[\eta^{\x,\x+1}]-\mathfrak{p}^{\mathbb{L}}_{\t}[\eta]\right\}\left\{|\mathfrak{p}^{\mathbb{L}}_{\t}[\eta^{\x,\x+1}]|^{p-1}-|\mathfrak{p}^{\mathbb{L}}_{\t}[\eta]|^{p-1}\right\}\right]\\
&\leq-\kappa p^{-1}N^{2}{\textstyle\sum_{\x\in\mathbb{L}}}\E^{\sigma,\mathbb{L}}|\mathfrak{p}^{\mathbb{L}}_{\t}[\eta^{\x,\x+1}]^{\frac{p}{2}}-\mathfrak{p}^{\mathbb{L}}_{\t}[\eta]^{\frac{p}{2}}|^{2}=-\kappa p^{-1}N^{2}\mathfrak{D}^{\sigma,\mathbb{L}}[|\mathfrak{p}^{\mathbb{L}}_{\t}|^{p/2}].
\end{align*}
The first line is by definition of {\small$\mathscr{L}^{\mathbb{L}}_{N,\mathrm{S}}$}. To justify the second line, we first write {\small$\frac12\mathfrak{p}^{\mathbb{L}}_{\t}[\eta]=\frac14\mathfrak{p}^{\mathbb{L}}_{\t}[\eta]+\frac14\mathfrak{p}^{\mathbb{L}}_{\t}[\eta]$}. For the first copy of {\small$\frac14\mathfrak{p}^{\mathbb{L}}_{\t}[\eta]$}, we make the change-of-variables {\small$\eta\mapsto\eta^{\x,\x+1}$}. This gives the difference {\small$|\mathfrak{p}^{\mathbb{L}}_{\t}[\eta^{\x,\x+1}]|^{p-1}-|\mathfrak{p}^{\mathbb{L}}_{\t}[\eta]|^{p-1}$} an additional sign, which leads to the second line above. To justify the last line, we use the inequality {\small$(\alpha-\beta)(\alpha^{p-1}-\beta^{p-1})\geq\kappa p^{-1}(\alpha^{p/2}-\beta^{p/2})^{2}$}, which can be checked after factoring {\small$(\alpha-\beta)^{2}$} from both sides and using {\small$\alpha^{q}-\beta^{q}=(\alpha-\beta)(\alpha^{q-1}+\ldots+\beta^{q-1})$} for any integer $q>0$. Next, we claim that, for some uniformly bounded $\mathfrak{c}$ and any small $\delta>0$, 
\begin{align*}
\E^{\sigma,\mathbb{L}}\mathfrak{p}^{\mathbb{L}}_{\t}\mathscr{L}^{\mathfrak{d},\mathbb{L}}_{N,\mathrm{A}}|\mathfrak{p}^{\mathbb{L}}_{\t}|^{p-1}&=\tfrac12N\sum_{\x\in\mathbb{L}}\E^{\sigma,\mathbb{L}}\left\{\mathfrak{c}[\tau_{\x}\eta]\mathfrak{p}^{\mathbb{L}}_{\t}[\eta]\left(\mathfrak{p}^{\mathbb{L}}_{\t}[\eta^{\x,\x+1}]^{p-1}-\mathfrak{p}^{\mathbb{L}}_{\t}[\eta]^{p-1}\right)\right\}\\
&=\tfrac12N\sum_{\x\in\mathbb{L}}\E^{\sigma,\mathbb{L}}\left\{\mathfrak{c}[\tau_{\x}\eta]\frac{\mathfrak{p}^{\mathbb{L}}_{\t}[\eta]\sum_{a=0}^{p-2}\mathfrak{p}^{\mathbb{L}}_{\t}[\eta]^{a}\mathfrak{p}^{\mathbb{L}}_{\t}[\eta^{\x,\x+1}]^{p-2-a}}{\sum_{b=0}^{p/2-1}\mathfrak{p}^{\mathbb{L}}_{\t}[\eta]^{b}\mathfrak{p}^{\mathbb{L}}_{\t}[\eta^{\x,\x+1}]^{\frac{p}{2}-1-b}}\left(\mathfrak{p}^{\mathbb{L}}_{\t}[\eta^{\x,\x+1}]^{\frac{p}{2}}-\mathfrak{p}^{\mathbb{L}}_{\t}[\eta]^{\frac{p}{2}}\right)\right\}\\
&\lesssim_{p}N\sum_{\x\in\mathbb{L}}\E^{\sigma,\mathbb{L}}\left\{\mathfrak{c}[\tau_{\x}\eta]\frac{\mathfrak{p}^{\mathbb{L}}_{\t}[\eta]^{p-1}+\mathfrak{p}^{\mathbb{L}}_{\t}[\eta^{\x,\x+1}]^{p-1}}{\mathfrak{p}^{\mathbb{L}}_{\t}[\eta]^{p/2-1}+\mathfrak{p}^{\mathbb{L}}_{\t}[\eta^{\x,\x+1}]^{p/2-1}}\left(\mathfrak{p}^{\mathbb{L}}_{\t}[\eta^{\x,\x+1}]^{\frac{p}{2}}-\mathfrak{p}^{\mathbb{L}}_{\t}[\eta]^{\frac{p}{2}}\right)\right\}\\
&\lesssim_{p}N\sum_{\x\in\mathbb{L}}\E^{\sigma,\mathbb{L}}\left\{\left(\mathfrak{p}^{\mathbb{L}}_{\t}[\eta^{\x,\x+1}]^{\frac{p}{2}}+\mathfrak{p}^{\mathbb{L}}_{\t}[\eta]^{\frac{p}{2}}\right)\left|\mathfrak{p}^{\mathbb{L}}_{\t}[\eta^{\x,\x+1}]^{\frac{p}{2}}-\mathfrak{p}^{\mathbb{L}}_{\t}[\eta]^{\frac{p}{2}}\right|\right\}\\
&\leq \delta^{-1}|\mathbb{L}|\E^{\sigma,\mathbb{L}}|\mathfrak{p}^{\mathbb{L}}_{\t}|^{p}+\delta N^{2}\sum_{\x\in\mathbb{L}}\E^{\sigma,\mathbb{L}}|\mathfrak{p}^{\mathbb{L}}_{\t}[\eta^{\x,\x+1}]^{\frac{p}{2}}-\mathfrak{p}^{\mathbb{L}}_{\t}[\eta]^{\frac{p}{2}}|^{2}\\
&\lesssim\delta^{-1}|\mathbb{L}|\E^{\sigma,\mathbb{L}}|\mathfrak{p}^{\mathbb{L}}_{\t}|^{p}+\delta N^{2}\mathfrak{D}^{\sigma,\mathbb{L}}[|\mathfrak{p}^{\mathbb{L}}_{\t}|^{p/2}].
\end{align*}
The first line follows by definition of {\small$\mathscr{L}^{\mathfrak{d},\mathbb{L}}_{N,\mathrm{A}}$}. The second line follows by the identities {\small$\alpha^{q}-\beta^{q}=(\alpha-\beta)(\alpha^{q-1}+\ldots+\beta^{q-1})$} for $q=p-1$ and $q=p/2$. (This is where we use the fact that $p/2$ is a positive integer.) The third line follows by two inequalities. The first is {\small$\mathfrak{p}^{\mathbb{L}}_{\t}[\eta]\sum_{a=0}^{p-2}\mathfrak{p}^{\mathbb{L}}_{\t}[\eta]^{a}\mathfrak{p}^{\mathbb{L}}_{\t}[\eta^{\x,\x+1}]^{p-2-a}\lesssim \mathfrak{p}^{\mathbb{L}}_{\t}[\eta^{\x,\x+1}]^{p-1}+\mathfrak{p}^{\mathbb{L}}_{\t}[\eta]^{p-1}$} by Young's inequality (intuitively, one can check the highest powers of {\small$\mathfrak{p}^{\mathbb{L}}_{\t}[\eta^{\x,\x+1}]$} and {\small$\mathfrak{p}^{\mathbb{L}}_{\t}[\eta]$} appearing on the LHS of this inequality). The second is {\small$\sum_{b=0}^{p/2-1}\mathfrak{p}^{\mathbb{L}}_{\t}[\eta]^{b}\mathfrak{p}^{\mathbb{L}}_{\t}[\eta^{\x,\x+1}]^{\frac{p}{2}-1-b}\gtrsim\mathfrak{p}^{\mathbb{L}}_{\t}[\eta^{\x,\x+1}]^{p/2-1}+\mathfrak{p}^{\mathbb{L}}_{\t}[\eta]^{p/2-1}$} by taking the summands at $b=0$ and $b=p/2-1$. The fourth line follows from the inequality {\small$(\alpha^{p-1}+\beta^{p-1})(\alpha^{p/2-1}+\beta^{p/2-1})^{-1}\leq \alpha^{p/2}+\beta^{p/2}$}. The fifth line follows by Schwarz, and the last line by definition of $\mathfrak{D}^{\sigma,\mathbb{L}}$. Finally, we claim that 
\begin{align}
|\E^{\sigma,\mathbb{L}}\mathfrak{p}^{\mathbb{L}}_{\t}\mathscr{L}^{\mathrm{free},\mathbb{L}}_{N,\mathrm{A}}|\mathfrak{p}^{\mathbb{L}}_{\t}|^{p-1}|&\lesssim_{p} N^{\frac32}\mathfrak{D}^{\sigma,\mathbb{L}}[|\mathfrak{p}^{\mathbb{L}}_{\t}|^{p/2}].\label{eq:lpprod1}
\end{align}
If we take \eqref{eq:lpprod1} for granted for now, then if we combine all displays thus far in this argument, we have 
\begin{align*}
\tfrac{\d}{\d\t}\E^{\sigma,\mathbb{L}}|\mathfrak{p}^{\mathbb{L}}_{\t}|^{p}&\leq-\kappa N^{2}\mathfrak{D}^{\sigma,\mathbb{L}}[|\mathfrak{p}^{\mathbb{L}}_{\t}|^{p/2}]+\mathrm{O}_{p}(\delta N^{2}\mathfrak{D}^{\sigma,\mathbb{L}}[|\mathfrak{p}^{\mathbb{L}}_{\t}|^{p/2}]+N^{\frac32}\mathfrak{D}^{\sigma,\mathbb{L}}[|\mathfrak{p}^{\mathbb{L}}_{\t}|^{p/2}]+\delta^{-1}|\mathbb{L}|\E^{\sigma,\mathbb{L}}|\mathfrak{p}^{\mathbb{L}}_{\t}|^{p}).
\end{align*}
If we take $\delta$ small enough depending only on the universal constant $\kappa$, then the first three terms on the RHS sum to something which is negative. We therefore have {\small$\tfrac{\d}{\d\t}\E^{\sigma,\mathbb{L}}|\mathfrak{p}^{\mathbb{L}}_{\t}|^{p}\lesssim_{p}\delta^{-1}|\mathbb{L}|\E^{\sigma,\mathbb{L}}|\mathfrak{p}^{\mathbb{L}}_{\t}|^{p}$}. If we combine this with the Gronwall inequality, we deduce {\small$\E^{\sigma,\mathbb{L}}|\mathfrak{p}^{\mathbb{L}}_{\t}|^{p}\lesssim\exp\{\mathrm{O}(|\mathbb{L}|\t)\}\E^{\sigma,\mathbb{L}}|\mathfrak{p}^{\mathbb{L}}_{0}|^{p}$}. We now recall {\small$\mathfrak{p}^{\mathbb{L}}_{0}\equiv1$} and use $|\mathbb{L}|\lesssim N^{1/2}$ and $\t\leq\tau\lesssim N^{-4/3}$ to deduce the desired bound {\small$\E^{\sigma,\mathbb{L}}|\mathfrak{p}^{\mathbb{L}}_{\t}|^{p}\lesssim1$}. Thus, it suffices to prove \eqref{eq:lpprod1}. We first note the following, which follows just from definition of {\small$\mathscr{L}^{\mathrm{free},\mathbb{L}}_{N,\mathrm{A}}$} (where {\small$\mathfrak{a}[\eta]=\mathbf{1}_{\substack{\eta_{\x}=-1\\\eta_{\x+1}=1}}-\mathbf{1}_{\substack{\eta_{\x}=1\\\eta_{\x+1}=-1}}$} for convenience):
\begin{align*}
\E^{\sigma,\mathbb{L}}\mathfrak{p}^{\mathbb{L}}_{\t}\mathscr{L}^{\mathrm{free},\mathbb{L}}_{N,\mathrm{A}}\mathfrak{p}^{\mathbb{L}}_{\t}|\mathfrak{p}^{\mathbb{L}}_{\t}|^{p-1}&=\tfrac12N^{\frac32}\sum_{\x\in\mathbb{L}}\E\left\{\mathfrak{a}[\tau_{\x}\eta]\mathfrak{p}^{\mathbb{L}}_{\t}[\eta]\left(\mathfrak{p}^{\mathbb{L}}_{\t}[\eta^{\x,\x+1}]^{p-1}-\mathfrak{p}^{\mathbb{L}}_{\t}[\eta]^{p-1}\right)\right\}.
\end{align*}
Similar to our computation of {\small$\E^{\sigma,\mathbb{L}}\mathfrak{p}^{\mathbb{L}}_{\t}\mathscr{L}^{\mathfrak{d},\mathbb{L}}_{N,\mathrm{A}}|\mathfrak{p}^{\mathbb{L}}_{\t}|^{p-1}$} above, we can expand the RHS to get
\begin{align*}
\E^{\sigma,\mathbb{L}}\mathfrak{p}^{\mathbb{L}}_{\t}\mathscr{L}^{\mathrm{free},\mathbb{L}}_{N,\mathrm{A}}|\mathfrak{p}^{\mathbb{L}}_{\t}|^{p-1}&=\tfrac12N^{\frac32}\sum_{\x\in\mathbb{L}}\E\left\{\mathfrak{a}[\tau_{\x}\eta]\frac{\mathfrak{p}^{\mathbb{L}}_{\t}[\eta]\sum_{a=0}^{p-2}\mathfrak{p}^{\mathbb{L}}_{\t}[\eta]^{a}\mathfrak{p}^{\mathbb{L}}_{\t}[\eta^{\x,\x+1}]^{p-2-a}}{\sum_{b=0}^{p/2-1}\mathfrak{p}^{\mathbb{L}}_{\t}[\eta]^{b}\mathfrak{p}^{\mathbb{L}}_{\t}[\eta^{\x,\x+1}]^{\frac{p}{2}-1-b}}\left(\mathfrak{p}^{\mathbb{L}}_{\t}[\eta^{\x,\x+1}]^{\frac{p}{2}}-\mathfrak{p}^{\mathbb{L}}_{\t}[\eta]^{\frac{p}{2}}\right)\right\}.
\end{align*}
Just for this proof, for efficiency, we set {\small$\alpha:=\mathfrak{p}^{\mathbb{L}}_{\t}[\eta]$} and {\small$\beta:=\mathfrak{p}^{\mathbb{L}}_{\t}[\eta^{\x,\x+1}]$}. We now claim the expansion below for the fraction on the RHS:
\begin{align*}
&\frac{\alpha\sum_{a=0}^{p-2}\alpha^{a}\beta^{p-2-a}}{\sum_{b=0}^{p/2-1}\alpha^{b}\beta^{\frac{p}{2}-1-b}}=\frac{(p-1)\alpha^{p-1}}{\sum_{b=0}^{p/2-1}\alpha^{b}\beta^{\frac{p}{2}-1-b}}+\frac{\alpha\sum_{a=0}^{p-2}\alpha^{a}(\beta^{p-2-a}-\alpha^{p-2-a})}{\sum_{b=0}^{p/2-1}\alpha^{b}\beta^{\frac{p}{2}-1-b}}\\
&=\frac{(p-1)\alpha^{p-1}}{\sum_{b=0}^{p/2-1}\alpha^{b}\beta^{\frac{p}{2}-1-b}}+\frac{\alpha\sum_{a=0}^{p-2}\alpha^{a}\sum_{c=0}^{p-3-a}\alpha^{c}\beta^{p-3-a-c}}{\left(\sum_{b=0}^{p/2-1}\alpha^{b}\beta^{\frac{p}{2}-1-b}\right)^{2}}(\beta^{\frac{p}{2}}-\alpha^{\frac{p}{2}})\\
&=\frac{(p-1)\alpha^{p-1}}{\sum_{b=0}^{p/2-1}\alpha^{b}\beta^{\frac{p}{2}-1-b}}+\mathrm{O}_{p}(|\beta^{\frac{p}{2}}-\alpha^{\frac{p}{2}}|).
\end{align*}
The first line is by adding and subtracting $\alpha^{p-2-a}$ for every $a$ in the sum. The second line follows by another application of {\small$\alpha^{q}-\beta^{q}=(\alpha-\beta)(\alpha^{q-1}+\ldots+\beta^{q-1})$} for $q=p-2-a$ and $q=p/2$. (Note that in the second line, we can avoid taking $a=p-2$ in the $a$-summation, because the interior $c$-summation would be empty and thus zero.) The third line above follows by another application of {\small$\alpha\sum_{a=0}^{p-2}\alpha^{a}\sum_{c=0}^{p-3-a}\alpha^{c}\beta^{p-3-a-c}\lesssim_{p}\alpha^{p-2}+\beta^{p-2}$} (for any $a=0,\ldots,p-3$) via Young's inequality and of the lower bound {\small$\sum_{b=0}^{p/2-1}\alpha^{b}\beta^{p/2-1-b}\gtrsim \alpha^{p/2-1}+\beta^{p/2-1}$} by taking the $b=0$ and $b=p/2-1$ summands. (Again, for the upper bound that we used Young's inequality for, check the highest powers of $\alpha,\beta$, respectively, that appear on the LHS while recalling that $a=0,\ldots,p-3$, so that $a\neq p-2$.) Next, we note
\begin{align*}
\frac{(p-1)\alpha^{p-1}}{\sum_{b=0}^{p/2-1}\alpha^{b}\beta^{\frac{p}{2}-1-b}}&=\tfrac{2(p-1)}{p}\alpha^{p/2}+(p-1)\alpha^{p-1}\left\{\left(\sum_{b=0}^{p/2-1}\alpha^{b}\beta^{\frac{p}{2}-1-b}\right)^{-1}-\left(\sum_{b=0}^{p/2-1}\alpha^{b}\alpha^{\frac{p}{2}-1-b}\right)^{-1}\right\}.
\end{align*}
Another similar and elementary calculation using {\small$\alpha^{q}-\beta^{q}=(\alpha-\beta)(\alpha^{q-1}+\ldots+\beta^{q-1})$} and Young's inequality shows the second term on the RHS above is $\lesssim_{p}|\beta^{p/2}-\alpha^{p/2}|$. (An intuitive way to see this bound is that the curly-bracketed term above vanishes if $\beta=\alpha$. Moreover, if we scale $(\alpha,\beta)\mapsto(\mathrm{c}\alpha,\mathrm{c}\beta)$ for any $\mathrm{c}>0$, then the second term on the RHS above scales with factor {\small$\mathrm{c}^{p/2}$}.) We combine this with the above three displays:
\begin{align*}
\E^{\sigma,\mathbb{L}}\mathfrak{p}^{\mathbb{L}}_{\t}\mathscr{L}^{\mathrm{free},\mathbb{L}}_{N,\mathrm{A}}|\mathfrak{p}^{\mathbb{L}}_{\t}|^{p-1}&=\tfrac{p-1}{p}N^{\frac32}\sum_{\x\in\mathbb{L}}\E\left\{\mathfrak{a}[\tau_{\x}\eta]\mathfrak{p}^{\mathbb{L}}_{\t}[\eta]^{\frac{p}{2}}\left(\mathfrak{p}^{\mathbb{L}}_{\t}[\eta^{\x,\x+1}]^{\frac{p}{2}}-\mathfrak{p}^{\mathbb{L}}_{\t}[\eta]^{\frac{p}{2}}\right)\right\}\\
&+\tfrac12N^{\frac32}\sum_{\x\in\mathbb{L}}\E\left\{\mathrm{O}_{p}\left(\mathfrak{p}^{\mathbb{L}}_{\t}[\eta^{\x,\x+1}]^{\frac{p}{2}}-\mathfrak{p}^{\mathbb{L}}_{\t}[\eta]^{\frac{p}{2}}\right)^{2}\right\}.
\end{align*}
The first term on the RHS is equal to a constant times {\small$\E^{\sigma,\mathbb{L}}|\mathfrak{p}^{\mathbb{L}}_{\t}|^{p/2}\mathscr{L}^{\mathrm{free},\mathbb{L}}_{N,\mathrm{A}}|\mathfrak{p}^{\mathbb{L}}_{\t}|^{p/2}$}. Since {\small$\mathscr{L}^{\mathrm{free},\mathbb{L}}_{N,\mathrm{A}}$} is anti-symmetric with respect to $\E^{\sigma,\mathbb{L}}$, this term is equal to $0$. The second term on the RHS above is {\small$\lesssim_{p}N^{3/2}\mathfrak{D}^{\sigma,\mathbb{L}}[|\mathfrak{p}^{\mathbb{L}}_{\t}|^{p/2}]$}. The proposed bound \eqref{eq:lpprod1} therefore follows.
\end{proof}
Although Lemma \ref{lemma:lpprod} is a powerful stability estimate for {\small$\mathfrak{p}^{\mathbb{L}}_{\t}$}, it cannot itself reduce the proof of Proposition \ref{prop:kv} to the same estimate but for the ``free process" (obtained by setting $\mathfrak{d}$ equal to $0$), which has $\mathbb{P}^{\sigma,\mathbb{L}}$ as an invariant measure. The reason is because the LHS of \eqref{eq:kv} is a function of the \emph{dynamics} of $\eta^{\mathbb{L}}$, not just one-point statistics. We will explain how to use the Ito formula to reduce showing \eqref{eq:kv} to estimates for one-point statistics shortly. The next two estimates are needed to run said strategy.

The first is a classical spectral gap estimate. Before we state it, recall the Dirichlet form {\small$\mathfrak{D}^{\sigma,\mathbb{L}}$} from \eqref{eq:dformloc}.
\begin{lemma}\label{lemma:l2op}
\fsp Suppose $\lambda\geq0$, and suppose that $\mathfrak{f}:\{\pm1\}^{\mathbb{L}}\to\R$ is such that $\E^{\sigma,\mathbb{L}}\mathfrak{f}=0$. We have the estimates {\small$\|(\lambda-\mathscr{L}_{N}^{\mathrm{free},\mathbb{L}})^{-1}\mathfrak{f}\|_{2}^{2}\lesssim\lambda^{-1}\|\mathfrak{f}\|_{\mathrm{H}^{-1}}^{2}$} and {\small$N^{2}\mathfrak{D}^{\sigma,\mathbb{L}}[(\lambda-\mathscr{L}^{\mathrm{free},\mathbb{L}}_{N})^{-1}\mathfrak{f}]\lesssim\|\mathfrak{f}\|_{\mathrm{H}^{-1}}^{2}$}, where {\small$\|\cdot\|_{p}^{p}:=\E|\cdot|^{p}$} and
\begin{align*}
\|\mathfrak{f}\|_{\mathrm{H}^{-1}}^{2}:=\sup_{\mathfrak{b}:\{\pm1\}^{\mathbb{L}}\to\R}\left\{2\E^{\sigma,\mathbb{L}}\mathfrak{f}\mathfrak{b}+\E^{\sigma,\mathbb{L}}\mathfrak{b}\mathscr{L}^{\mathrm{free},\mathbb{L}}_{N}\mathfrak{b}\right\}.
\end{align*}
\end{lemma}
\begin{proof}
See the beginning of the proof of Proposition 6.1 in Appendix 1 of \cite{KL}.
\end{proof}
We now present a less standard estimate that we will need for technical reasons. It is an operator norm bound on {\small$\mathrm{L}^{\infty}(\{\pm1\}^{\mathbb{L}})$} whose proof is based on convexity and spectral gap. 
\begin{lemma}\label{lemma:linftyop}
\fsp Suppose $\lambda\geq0$, and suppose that $\mathfrak{f}:\{\pm1\}^{\mathbb{L}}\to\R$ is such that $\E^{\sigma,\mathbb{L}}\mathfrak{f}=0$. We have the estimate {\small$\sup_{p\geq1}\|(\lambda-\mathscr{L}^{\mathrm{free},\mathbb{L}}_{N})^{-1}\mathfrak{f}\|_{p}\lesssim N^{\mathrm{O}(1)}\|\mathfrak{f}\|_{\infty}$}, where the exponent $\mathrm{O}(1)$ does not depend on $\lambda$.
\end{lemma}
\begin{proof}
For $\mathfrak{t}=N^{\mathrm{D}}$ with $\mathrm{D}=\mathrm{O}(1)$ large, we have the resolvent identity
\begin{align*}
(\lambda-\mathscr{L}^{\mathrm{free},\mathbb{L}}_{N})^{-1}\mathfrak{f}&={\textstyle\int_{0}^{\infty}}\mathrm{e}^{-\lambda\r}[\mathrm{e}^{\r\mathscr{L}^{\mathrm{free},\mathbb{L}}_{N}}\mathfrak{f}]\d\r={\textstyle\int_{0}^{\mathfrak{t}}}\mathrm{e}^{-\lambda\r}[\mathrm{e}^{\r\mathscr{L}^{\mathrm{free},\mathbb{L}}_{N}}\mathfrak{f}]\d\r+{\textstyle\int_{\mathfrak{t}}^{\infty}}\mathrm{e}^{-\lambda\r}[\mathrm{e}^{\r\mathscr{L}^{\mathrm{free},\mathbb{L}}_{N}}\mathfrak{f}]\d\r.
\end{align*}
We recall that {\small$\mathscr{L}^{\mathrm{free},\mathbb{L}}_{N}$} has $\mathbb{P}^{\sigma,\mathbb{L}}$ as an invariant measure. So, it is contractive with respect to $\|\|_{p}$-norms. Since $\lambda\geq0$, we get the following estimate for the first term on the far RHS of the previous display:
\begin{align*}
\|{\textstyle\int_{0}^{\mathfrak{t}}}\mathrm{e}^{-\lambda\r}[\mathrm{e}^{\r\mathscr{L}^{\mathrm{free},\mathbb{L}}_{N}}\mathfrak{f}]\d\r\|_{p}&\leq{\textstyle\int_{0}^{\mathfrak{t}}}\|\mathrm{e}^{\r\mathscr{L}^{\mathrm{free},\mathbb{L}}_{N}}\mathfrak{f}\|_{p}\d\r\leq\mathfrak{t}\|\mathfrak{f}\|_{p}\leq N^{\mathrm{O}(1)}\|\mathfrak{f}\|_{\infty},
\end{align*}
where the last bound is by monotonicity in $p$ and $\mathfrak{t}=N^{\mathrm{O}(1)}$. Now, we claim {\small$\|\cdot\|_{\infty}\lesssim\exp\{\mathrm{O}(|\mathbb{L}|)\}\|\cdot\|_{2}$}. This follows because $\mathbb{P}^{\sigma,\mathbb{L}}$ is uniform measure on a subset of $\{\pm1\}^{\mathbb{L}}$, whose size is {\small$\lesssim\exp\{\mathrm{O}(|\mathbb{L}|)\}$}. Next, we claim that {\small$\|\exp\{\r\mathscr{L}^{\mathrm{free},\mathbb{L}}_{N}\}\mathfrak{f}\|_{2}\lesssim\exp\{-\kappa\r\}\|\mathfrak{f}\|_{2}$}. This follows because (the symmetric part of) {\small$\mathscr{L}^{\mathrm{free},\mathbb{L}}_{N}$} has a spectral gap of size $\gtrsim N^{2}|\mathbb{L}|^{-2}\gtrsim 1$, which, in turn, follows by the $N^{2}$-scaling in {\small$\mathscr{L}^{\mathrm{free},\mathbb{L}}_{N}$} and the log-Sobolev inequality with diffusive-scaling-constant in \cite{Y}. (Indeed, one can check {\small$\partial_{\r}\|\exp\{\r\mathscr{L}^{\mathrm{free},\mathbb{L}}_{N}\}\mathfrak{f}\|_{2}^{2}\leq-\kappa'N^{2}\mathfrak{D}^{\sigma,\mathbb{L}}[\exp\{\r\mathscr{L}^{\mathrm{free},\mathbb{L}}_{N}\}\mathfrak{f}]\leq-\kappa''N^{2}|\mathbb{L}|^{-2}\|\exp\{\r\mathscr{L}^{\mathrm{free},\mathbb{L}}_{N}\}\mathfrak{f}\|_{2}^{2}$}, where the spectral gap gives the last inequality, and where $\kappa',\kappa''\gtrsim1$ are universal constants. Now use Gronwall to get the desired exponential decay.) Ultimately, we get
\begin{align*}
\|{\textstyle\int_{\mathfrak{t}}^{\infty}}\mathrm{e}^{-\lambda\r}[\mathrm{e}^{\r\mathscr{L}^{\mathrm{free},\mathbb{L}}_{N}}\mathfrak{f}]\d\r\|_{\infty}&\leq{\textstyle\int_{\mathfrak{t}}^{\infty}}\|\mathrm{e}^{-\lambda\r}[\mathrm{e}^{\r\mathscr{L}^{\mathrm{free},\mathbb{L}}_{N}}\mathfrak{f}]\|_{\infty}\d\r\leq\mathrm{e}^{\mathrm{O}(|\mathbb{L}|)}{\textstyle\int_{\mathfrak{t}}^{\infty}}\|\mathrm{e}^{\r\mathscr{L}^{\mathrm{free},\mathbb{L}}_{N}}\mathfrak{f}\|_{2}\d\r\\
&\lesssim\mathrm{e}^{\mathrm{O}(|\mathbb{L}|)}{\textstyle\int_{\mathfrak{t}}^{\infty}}\mathrm{e}^{-\kappa\r}\|\mathfrak{f}\|_{2}\d\r\lesssim\mathrm{e}^{\mathrm{O}(|\mathbb{L}|)}\mathrm{e}^{-\kappa\mathfrak{t}}\|\mathfrak{f}\|_{2}\lesssim \mathrm{e}^{\mathrm{O}(|\mathbb{L}|)}\mathrm{e}^{-\kappa N^{\mathrm{D}}}\|\mathfrak{f}\|_{\infty},
\end{align*}
where the last bound is again by monotonicity of $\|\cdot\|_{p}$ in $p$ and $\mathfrak{t}=N^{\mathrm{D}}$. If we take the exponent $\mathrm{D}>0$ to be large enough, since $|\mathbb{L}|\leq N$, the last term is exponentially small in $N$. Combining this with the previous three displays completes the proof.
\end{proof}
\subsubsection{Proof of Proposition \ref{prop:kv}}
Our strategy, as in Appendix 1.6 of \cite{KL}, is based on the Ito formula to rewrite the time-integral on the LHS of \eqref{eq:kv} and control all terms by the $\mathrm{H}^{-1}$-norm of the average of $\mathfrak{f}_{i}$-terms. However, unlike Appendix 1.6 in \cite{KL}, we will need to isolate the ``free" generator {\small$\mathscr{L}^{\mathrm{free},\mathbb{L}}_{N}$} from the total generator {\small$\mathscr{L}^{\mathbb{L}}_{N}$} and deal with the fact that in principle, we do not know anything about the invariant measure of $\eta^{\mathbb{L}}$. 

First, we reduce to the case $\tau\gtrsim N^{-10}$. If $\tau\lesssim N^{-10}$, then the RHS of \eqref{eq:kv} is $\gtrsim N^{5}$, for example, since the number $n$ of disjoint subsets of $\mathbb{L}$ satisfies $n\leq|\mathbb{L}|\leq N$. But the LHS of \eqref{eq:kv} is $\mathrm{O}(1)$ since it is an average of $\mathrm{O}(1)$ objects. Now, for convenience, we set {\small$\mathfrak{A}[\eta^{\mathbb{L}}]:=n^{-1}(\mathfrak{f}_{1}[\eta^{\mathbb{L}}]+\ldots+\mathfrak{f}_{n}[\eta^{\mathbb{L}}])$} just for this proof. Fix $\lambda\asymp\tau^{-1}$. We claim
\begin{align*}
\tau^{-1}{\textstyle\int_{0}^{\tau}}\mathfrak{A}[\eta^{\mathbb{L}}_{\s}]\d\s&=\tau^{-1}{\textstyle\int_{0}^{\tau}}(\lambda-\mathscr{L}^{\mathrm{free},\mathbb{L}}_{N})(\lambda-\mathscr{L}^{\mathrm{free},\mathbb{L}}_{N})^{-1}\mathfrak{A}[\eta^{\mathbb{L}}_{\s}]\d\s\\
&=\tau^{-1}{\textstyle\int_{0}^{\tau}}\lambda(\lambda-\mathscr{L}^{\mathrm{free},\mathbb{L}}_{N})^{-1}\mathfrak{A}[\eta^{\mathbb{L}}_{\s}]\d\s-\tau^{-1}{\textstyle\int_{0}^{\tau}}\mathscr{L}^{\mathbb{L}}_{N}(\lambda-\mathscr{L}^{\mathrm{free},\mathbb{L}}_{N})^{-1}\mathfrak{A}[\eta^{\mathbb{L}}_{\s}]\d\s\\
&+\tau^{-1}{\textstyle\int_{0}^{\tau}}\mathscr{L}^{\mathfrak{d},\mathbb{L}}_{N,\mathrm{A}}(\lambda-\mathscr{L}^{\mathrm{free},\mathbb{L}}_{N})^{-1}\mathfrak{A}[\eta^{\mathbb{L}}_{\s}]\d\s.
\end{align*}
The second line follows from the first line and from {\small$\mathscr{L}^{\mathbb{L}}_{N}=\mathscr{L}^{\mathrm{free},\mathbb{L}}_{N}+\mathscr{L}^{\mathfrak{d},\mathbb{L}}_{N,\mathrm{A}}$} (see the display before Proposition \ref{prop:kv} and the display after Proposition \ref{prop:kv}). We now claim that for any $\delta>0$, we have
\begin{align}
\E\left(\tau^{-1}{\textstyle\int_{0}^{\tau}}\lambda(\lambda-\mathscr{L}^{\mathrm{free},\mathbb{L}}_{N})^{-1}\mathfrak{A}[\eta^{\mathbb{L}}_{\s}]\d\s\right)^{2}&\lesssim_{\delta}N^{\delta}N^{-2}\tau^{-1}n^{-1}\sup_{i=1,\ldots,n}|\mathbb{I}_{i}|^{2},\label{eq:kv1a}\\
\E\left(\tau^{-1}{\textstyle\int_{0}^{\tau}}\mathscr{L}^{\mathbb{L}}_{N}(\lambda-\mathscr{L}^{\mathrm{free},\mathbb{L}}_{N})^{-1}\mathfrak{A}[\eta^{\mathbb{L}}_{\s}]\d\s\right)^{2}&\lesssim_{\delta}N^{\delta}N^{-2}\tau^{-1}n^{-1}\sup_{i=1,\ldots,n}|\mathbb{I}_{i}|^{2},\label{eq:kv1b}\\
\E\left(\tau^{-1}{\textstyle\int_{0}^{\tau}}\mathscr{L}^{\mathfrak{d},\mathbb{L}}_{N,\mathrm{A}}(\lambda-\mathscr{L}^{\mathrm{free},\mathbb{L}}_{N})^{-1}\mathfrak{A}[\eta^{\mathbb{L}}_{\s}]\d\s\right)^{2}&\lesssim_{\delta}N^{\delta}N^{-2}\tau^{-1}n^{-1}\sup_{i=1,\ldots,n}|\mathbb{I}_{i}|^{2}.\label{eq:kv1c}
\end{align}
Assuming these bounds are true, Proposition \ref{prop:kv} would follow by the previous two displays. Thus, it suffices to prove each of \eqref{eq:kv1a}-\eqref{eq:kv1c}. We start with \eqref{eq:kv1a}. We claim that for some $C=\mathrm{O}(1)$ and any $\e>0$, we have
\begin{align*}
\mathrm{LHS}\eqref{eq:kv1a}&\lesssim\tau^{-1}{\textstyle\int_{0}^{\tau}}\E|\lambda(\lambda-\mathscr{L}^{\mathrm{free},\mathbb{L}}_{N})^{-1}\mathfrak{A}[\eta^{\mathbb{L}}_{\s}]|^{2}\d\s\\
&\lesssim_{\e}\tau^{-1}{\textstyle\int_{0}^{\tau}}\left\{\E^{\sigma,\mathbb{L}}|\lambda(\lambda-\mathscr{L}^{\mathrm{free},\mathbb{L}}_{N})^{-1}\mathfrak{A}|^{2+\e}\right\}^{\frac{2}{2+\e}}\d\s\\
&\lesssim\left\{N^{C\e}\lambda^{1+\e}\|\mathfrak{A}\|_{\mathrm{H}^{-1}}^{2}\right\}^{\frac{2}{2+\e}}\lesssim\left\{N^{{C}\e}N^{-2}\tau^{-1-\e}n^{-1}\sup_{i=1,\ldots,n}|\mathbb{I}_{i}|^{2}\right\}^{\frac{2}{2+\e}}.
\end{align*}
The first bound follows by Cauchy-Schwarz with respect to the $\d\s$-integration. The second bound follows from H\"{o}lder and Lemma \ref{lemma:lpprod} (which lets us change measure from the law of {\small$\eta^{\mathbb{L}}_{\s}$} to $\mathbb{P}^{\sigma,\mathbb{L}}$). To justify the third bound, we first use Lemma \ref{lemma:linftyop} and uniform boundedness of $\mathfrak{A}$ to get {\small$|\lambda(\lambda-\mathscr{L}^{\mathrm{free},\mathbb{L}}_{N})^{-1}\mathfrak{A}|^{2+\e}\lesssim\lambda^{2+\e}|(\lambda-\mathscr{L}^{\mathrm{free},\mathbb{L}}_{N})^{-1}\mathfrak{A}|^{2+\e}\lesssim N^{{C}\e}\lambda^{2+\e}|(\lambda-\mathscr{L}^{\mathrm{free},\mathbb{L}}_{N})^{-1}\mathfrak{A}|^{2}$}. We now take expectation and use Lemma \ref{lemma:l2op} to get the first bound in the last line. The last bound follows by {\small$\lambda\asymp\tau^{-1}$} and standard $\mathrm{H}^{-1}$-calculations; see Proposition 7 in \cite{GJ15}, for example. Since $n\geq1$ and $|\mathbb{I}_{i}|\geq1$ for all $i$ and {\small$\tau\gtrsim N^{-10}$}, the last term in the above display is {\small$\lesssim_{\e}N^{{C}\e}N^{-2}\tau^{-1}n^{-1}\sup_{i}|\mathbb{I}_{i}|^{2}$} for a possibly different $C=\mathrm{O}(1)$. Taking $\e$ small gives \eqref{eq:kv1a}. We now prove \eqref{eq:kv1c}. Similarly, by Cauchy-Schwarz in the $\d\s$-integration and by using H\"{o}lder and Lemma \ref{lemma:lpprod} to change measure as in our proof of \eqref{eq:kv1a}, we have (for any $\e>0$) that
\begin{align*}
\mathrm{LHS}\eqref{eq:kv1c}&\lesssim\tau^{-1}{\textstyle\int_{0}^{\tau}}\E|\mathscr{L}^{\mathfrak{d},\mathbb{L}}_{N,\mathrm{A}}(\lambda-\mathscr{L}^{\mathrm{free},\mathbb{L}}_{N})^{-1}\mathfrak{A}[\eta^{\mathbb{L}}_{\s}]|^{2}\d\s\\
&\lesssim_{\e}\tau^{-1}{\textstyle\int_{0}^{\tau}}\left\{\E^{\sigma,\mathbb{L}}|\mathscr{L}^{\mathfrak{d},\mathbb{L}}_{N,\mathrm{A}}(\lambda-\mathscr{L}^{\mathrm{free},\mathbb{L}}_{N})^{-1}\mathfrak{A}|^{2+\e}\right\}^{\frac{2}{2+\e}}.
\end{align*}
Next, we note that {\small$\mathscr{L}^{\mathfrak{d},\mathbb{L}}_{N,\mathrm{A}}$} is a bounded operator on {\small$\mathrm{L}^{\infty}(\{\pm1\}^{\mathbb{L}})\to\mathrm{L}^{\infty}(\{\pm1\}^{\mathbb{L}})$} with norm $\lesssim N^{\mathrm{O}(1)}$. Using this with Lemma \ref{lemma:linftyop} shows that {\small$\|\mathscr{L}^{\mathfrak{d},\mathbb{L}}_{N,\mathrm{A}}(\lambda-\mathscr{L}^{\mathrm{free},\mathbb{L}}_{N})^{-1}\mathfrak{A}\|_{\infty}\lesssim N^{\mathrm{O}(1)}$}. Now, we use this bound, the definition of {\small$\mathscr{L}^{\mathfrak{d},\mathbb{L}}_{N,\mathrm{A}}$}, and Cauchy-Schwarz for the $\x$-sum below to get
\begin{align*}
|\mathscr{L}^{\mathfrak{d},\mathbb{L}}_{N,\mathrm{A}}(\lambda-\mathscr{L}^{\mathrm{free},\mathbb{L}}_{N})^{-1}\mathfrak{A}|^{2+\e}&\lesssim N^{{C}\e}|\mathscr{L}^{\mathfrak{d},\mathbb{L}}_{N,\mathrm{A}}(\lambda-\mathscr{L}^{\mathrm{free},\mathbb{L}}_{N})^{-1}\mathfrak{A}|^{2}\\
&\lesssim N^{{C}\e}\left\{N\sum_{\x\in\mathbb{L}}|\mathscr{L}_{\x}(\lambda-\mathscr{L}^{\mathrm{free},\mathbb{L}}_{N})^{-1}\mathfrak{A}|\right\}^{2}\\
&\lesssim N^{{C}\e}|\mathbb{L}|N^{2}\sum_{\x\in\mathbb{L}}|\mathscr{L}_{\x}(\lambda-\mathscr{L}^{\mathrm{free},\mathbb{L}}_{N})^{-1}\mathfrak{A}|^{2}.
\end{align*}
If we take $\E^{\sigma,\mathbb{L}}$ of the last line above, we get {\small$\lesssim N^{{C}\e}|\mathbb{L}|N^{2}\mathfrak{D}^{\sigma,\mathbb{L}}[(\lambda-\mathscr{L}^{\mathrm{free},\mathbb{L}}_{N})^{-1}\mathfrak{A}]$} by \eqref{eq:dformloc}. By Lemma \ref{lemma:l2op}, this resulting upper bound is {\small$\lesssim N^{{C}\e}|\mathbb{L}|\|\mathfrak{A}\|_{\mathrm{H}^{-1}}^{2}$}. If we now combine this with the previous two displays, we get 
\begin{align*}
\mathrm{LHS}\eqref{eq:kv1c}&\lesssim_{\e}\tau^{-1}{\textstyle\int_{0}^{\tau}}\left\{N^{{C}\e}|\mathbb{L}|\|\mathfrak{A}\|_{\mathrm{H}^{-1}}^{2}\right\}^{\frac{2}{2+\e}}\d\s\lesssim\left\{N^{{C}\e}|\mathbb{L}|\|\mathfrak{A}\|_{\mathrm{H}^{-1}}^{2}\right\}^{\frac{2}{2+\e}}\\
&\lesssim\left\{N^{{C}\e}|\mathbb{L}|N^{-2}n^{-1}\sup_{i=1,\ldots,n}|\mathbb{I}_{i}|^{2}\right\}^{\frac{2}{2+\e}}
\end{align*}
where the second bound is because the time-integrand is now independent of the integration variable $\s$, and the last bound is by Proposition 7 in \cite{GJ15} as in our proof of \eqref{eq:kv1a} above. Now, since $|\mathbb{L}|\lesssim N^{1/2}\lesssim N^{4/3}\lesssim \tau^{-1}$ and $n,|\mathbb{I}_{i}|\geq1$ and {\small$\tau\gtrsim N^{-10}$}, the last term in the above display is {\small$\lesssim_{\e}N^{{C}\e}N^{-2}\tau^{-1}n^{-1}\sup_{i}|\mathbb{I}_{i}|^{2}$} for a possibly different $C=\mathrm{O}(1)$. Taking $\e>0$ small now gives \eqref{eq:kv1c}. We are left to prove \eqref{eq:kv1b}. By the Ito formula, since {\small$\mathscr{L}^{\mathbb{L}}_{N}$} is the generator for {\small$\eta^{\mathbb{L}}$}, we have 
\begin{align*}
\tau^{-1}{\textstyle\int_{0}^{\tau}}\mathscr{L}^{\mathbb{L}}_{N}(\lambda-\mathscr{L}^{\mathrm{free},\mathbb{L}}_{N})^{-1}\mathfrak{A}[\eta^{\mathbb{L}}_{\s}]\d\s&=\tau^{-1}(\lambda-\mathscr{L}^{\mathrm{free},\mathbb{L}}_{N})^{-1}\mathfrak{A}[\eta^{\mathbb{L}}_{\tau}]-\tau^{-1}(\lambda-\mathscr{L}^{\mathrm{free},\mathbb{L}}_{N})^{-1}\mathfrak{A}[\eta^{\mathbb{L}}_{0}]+\tau^{-1}\mathbf{m}_{\tau},
\end{align*}
where $\mathbf{m}_{\tau}$ is a martingale whose predictable bracket is given by (see Appendix 1.5 in \cite{KL})
\begin{align*}
[\mathbf{m}]_{\tau}={\textstyle\int_{0}^{\tau}}\left\{\mathscr{L}^{\mathbb{L}}_{N}|(\lambda-\mathscr{L}^{\mathrm{free},\mathbb{L}}_{N})^{-1}\mathfrak{A}[\eta^{\mathbb{L}}_{\s}]|^{2}-2\{(\lambda-\mathscr{L}^{\mathrm{free},\mathbb{L}}_{N})^{-1}\mathfrak{A}[\eta^{\mathbb{L}}_{\s}]\}\cdot\mathscr{L}^{\mathbb{L}}_{N}\{(\lambda-\mathscr{L}^{\mathrm{free},\mathbb{L}}_{N})^{-1}\mathfrak{A}[\eta^{\mathbb{L}}_{\s}]\}\right\}\d\s.
\end{align*}
The term in curly brackets is non-negative (it is the time-derivative of a predictable bracket, or alternatively, it is a Carre-du-Champ operator). So, to prove \eqref{eq:kv1b} and thereby complete the proof of the entire proposition, it suffices to prove the following two estimates for any $\delta>0$:
\begin{align}
\sup_{0\leq\t\leq\tau}\E|\tau^{-1}(\lambda-\mathscr{L}^{\mathrm{free},\mathbb{L}}_{N})^{-1}\mathfrak{A}[\eta^{\mathbb{L}}_{\t}]|^{2}&\lesssim_{\delta}N^{\delta}N^{-2}\tau^{-1}n^{-1}\sup_{i=1,\ldots,n}|\mathbb{I}_{i}|^{2},\label{eq:kv2a}\\
\E|\tau^{-1}\mathbf{m}_{\tau}|^{2}&\lesssim_{\delta}N^{\delta}N^{-2}\tau^{-1}n^{-1}\sup_{i=1,\ldots,n}|\mathbb{I}_{i}|^{2}.\label{eq:kv2b}
\end{align}
We start with \eqref{eq:kv2a}. We first claim the following calculation (which we justify afterwards):
\begin{align*}
\E|\tau^{-1}(\lambda-\mathscr{L}^{\mathrm{free},\mathbb{L}}_{N})^{-1}\mathfrak{A}[\eta^{\mathbb{L}}_{\t}]|^{2}&\lesssim_{\e}\{\E^{\sigma,\mathbb{L}}|\tau^{-1}(\lambda-\mathscr{L}^{\mathrm{free},\mathbb{L}}_{N})^{-1}\mathfrak{A}|^{2+\e}\}^{\frac{2}{2+\e}}\\
&\lesssim\{\E^{\sigma,\mathbb{L}}|\tau^{-1}(\lambda-\mathscr{L}^{\mathrm{free},\mathbb{L}}_{N})^{-1}\mathfrak{A}|^{2}\}^{\frac{2}{2+\e}}\\
&\lesssim\{\tau^{-1}\|\mathfrak{A}\|_{\mathrm{H}^{-1}}^{2}\}^{\frac{2}{2+\e}}\lesssim\{N^{-2}\tau^{-1}n^{-1}{\textstyle\sup_{i}}|\mathbb{I}_{i}|^{2}\}^{\frac{2}{2+\e}}
\end{align*}
The first bound follows by H\"{o}lder and Lemma \ref{lemma:lpprod} as before to change measure. The second line follows because the operator norm of {\small$(\lambda-\mathscr{L}^{\mathrm{free},\mathbb{L}}_{N})^{-1}:\mathrm{L}^{\infty}(\{\pm1\}^{\mathbb{L}})\to\mathrm{L}^{\infty}(\{\pm1\}^{\mathbb{L}})$} is $\lesssim\lambda^{-1}$. (This bound is a standard resolvent estimate for Markov generators; it can be deduced from the proof of Lemma \ref{lemma:linftyop} if we only use contractivity of {\small$\exp\{\r\mathscr{L}^{\mathrm{free},\mathbb{L}}_{N}\}:\mathrm{L}^{\infty}(\{\pm1\}^{\mathbb{L}})\to\mathrm{L}^{\infty}(\{\pm1\}^{\mathbb{L}})$} and integrate $\exp\{-\lambda\r\}\d\r$ over $\r\geq0$.) In particular, this implies that {\small$|\tau^{-1}(\lambda-\mathscr{L}^{\mathrm{free},\mathbb{L}}_{N})^{-1}\mathfrak{A}|^{2+\e}\lesssim|\tau^{-1}(\lambda-\mathscr{L}^{\mathrm{free},\mathbb{L}}_{N})^{-1}\mathfrak{A}|^{2}$} since $\tau\asymp\lambda^{-1}$. The last line follows by Lemma \ref{lemma:l2op} and $\tau\asymp\lambda^{-1}$, and then by the $\mathrm{H}^{-1}$-estimate from Proposition 7 of \cite{GJ15} as used earlier in this proof. Since $\e>0$ is arbitrary, we can take it small enough so that \eqref{eq:kv2a} follows. We now prove \eqref{eq:kv2b}. Let {\small$\Gamma_{1}[\eta^{\mathbb{L}}_{\s}]$} be the integrand in $[\mathbf{m}]_{\tau}$. (This is the usual notation for Carre-du-Champ.) We now claim (for some $C=\mathrm{O}(1)$ and any $\e>0$) that
\begin{align*}
\E|\tau^{-1}\mathbf{m}_{\tau}|^{2}&\lesssim\tau^{-2}\E[\mathbf{m}]_{\tau}\lesssim\tau^{-2}{\textstyle\int_{0}^{\tau}}\E\Gamma_{1}[\eta^{\mathbb{L}}_{\s}]\d\s\lesssim_{\e}\tau^{-2}{\textstyle\int_{0}^{\tau}}\left\{\E^{\sigma,\mathbb{L}}\Gamma_{1}^{1+\e}\right\}^{\frac{2}{2+\e}}\d\s\lesssim N^{{C}\e}\tau^{-1}\left\{\E^{\sigma,\mathbb{L}}\Gamma_{1}\right\}^{\frac{2}{2+\e}}.
\end{align*}
The first two bounds are by the Ito isometry (or definition of the predictable bracket) and formula for {\small$[\mathbf{m}]_{\tau}$}. The third bound is by H\"{o}lder and Lemma \ref{lemma:lpprod} as before. To get the last bound, first note that {\small$\E^{\sigma,\mathbb{L}}\Gamma_{1}^{1+\e}$} is independent of the $\s$-integration variable, and the $\tau^{-2}$ becomes $\tau^{-1}$ after integrating. Then, note that {\small$\Gamma\lesssim N^{\mathrm{O}(1)}$}; this follows because all $\mathscr{L}$-operators appearing in $\Gamma_{1}$ are bounded in operator norm (as maps on {\small$\mathrm{L}^{\infty}(\{\pm1\}^{\mathbb{L}})$}) by $N^{\mathrm{O}(1)}$, and by resolvent bounds in Lemma \ref{lemma:linftyop}. This gives {\small$\Gamma_{1}^{\e}\lesssim N^{{C}\e}$}, and the last bound above follows. We now write
\begin{align}
\Gamma_{1}&:=\Gamma_{1,\mathrm{free}}+\Gamma_{1,\mathfrak{d}},\label{eq:kv3a}\\
\Gamma_{1,\mathrm{free}}&:=\mathscr{L}^{\mathrm{free},\mathbb{L}}_{N}|(\lambda-\mathscr{L}^{\mathrm{free},\mathbb{L}}_{N})^{-1}\mathfrak{A}|^{2}-2\{(\lambda-\mathscr{L}^{\mathrm{free},\mathbb{L}}_{N})^{-1}\mathfrak{A}\}\cdot\mathscr{L}^{\mathrm{free},\mathbb{L}}_{N}\{(\lambda-\mathscr{L}^{\mathrm{free},\mathbb{L}}_{N})^{-1}\mathfrak{A}\},\label{eq:kv3b}\\
\Gamma_{1,\mathfrak{d}}&:=\mathscr{L}^{\mathfrak{d},\mathbb{L}}_{N,\mathrm{A}}|(\lambda-\mathscr{L}^{\mathrm{free},\mathbb{L}}_{N})^{-1}\mathfrak{A}|^{2}-2\{(\lambda-\mathscr{L}^{\mathrm{free},\mathbb{L}}_{N})^{-1}\mathfrak{A}\}\cdot\mathscr{L}^{\mathfrak{d},\mathbb{L}}_{N,\mathrm{A}}\{(\lambda-\mathscr{L}^{\mathrm{free},\mathbb{L}}_{N})^{-1}\mathfrak{A}\}.\label{eq:kv3c}
\end{align}
This follows just by the same decomposition {\small$\mathscr{L}^{\mathbb{L}}_{N}=\mathscr{L}^{\mathrm{free},\mathbb{L}}_{N}+\mathscr{L}^{\mathfrak{d},\mathbb{L}}_{N,\mathrm{A}}$} that we wrote before \eqref{eq:kv1a}. Let us control {\small$\E^{\sigma,\mathbb{L}}\Gamma_{1,\mathrm{free}}$}. We use an argument from the proof of Proposition 6.1 in Appendix 1 of \cite{KL} (which is, at this point, standard). Let us briefly explain it. Because {\small$\mathscr{L}^{\mathrm{free},\mathbb{L}}_{N}$} has $\mathbb{P}^{\sigma,\mathbb{L}}$ as an invariant measure, when we take $\E^{\sigma,\mathbb{L}}$ of the first term on the RHS of \eqref{eq:kv3b}, we get $0$. When we take $\E^{\sigma,\mathbb{L}}$ of the last term in \eqref{eq:kv3b}, we can replace {\small$\mathscr{L}^{\mathrm{free},\mathbb{L}}_{N}$} therein with its symmetric part {\small$\mathscr{L}^{\mathbb{L}}_{N,\mathrm{S}}$}. Thus, when we take $\E^{\sigma,\mathbb{L}}$ of the last term in \eqref{eq:kv3b}, we get a Dirichlet form scaled by the speed $N^{2}$ of {\small$\mathscr{L}^{\mathbb{L}}_{N,\mathrm{S}}$}. In particular, we get a bound of {\small$\lesssim N^{2}\mathfrak{D}[(\lambda-\mathscr{L}^{\mathrm{free},\mathbb{L}}_{N})^{-1}\mathfrak{A}]$}, which we can bound by {\small$\lesssim\|\mathfrak{A}\|_{\mathrm{H}^{-1}}^{2}$} via Lemma \ref{lemma:l2op}. Ultimately, we get
\begin{align}
\E^{\sigma,\mathbb{L}}|\Gamma_{1,\mathrm{free}}|&\lesssim\|\mathfrak{A}\|_{\mathrm{H}^{-1}}^{2}\lesssim N^{-2}n^{-1}\sup_{i=1,\ldots,n}|\mathbb{I}_{i}|^{2},\label{eq:kv4}
\end{align}
where the last inequality above is, again, by Proposition 7 in \cite{GJ15}. We now control $\Gamma_{1,\mathfrak{d}}$. To this end, to make the presentation clearer, for this argument only, we set {\small$\mathscr{R}_{\lambda}:=(\lambda-\mathscr{L}^{\mathrm{free},\mathbb{L}}_{N})^{-1}\mathfrak{A}$} as functions on $\{\pm1\}^{\mathbb{L}}$. By definition of {\small$\mathscr{L}^{\mathfrak{d},\mathbb{L}}_{N,\mathrm{A}}$} (see the proof of Lemma \ref{lemma:lpprod}) and the inequality $|a^{2}-b^{2}|\lesssim N|a-b|^{2}+N^{-1}|a+b|^{2}$, we have 
\begin{align*}
|\mathscr{L}^{\mathfrak{d},\mathbb{L}}_{N,\mathrm{A}}\mathscr{R}_{\lambda}^{2}|&\lesssim N\sum_{\x\in\mathbb{L}}|\mathscr{R}_{\lambda}[\eta^{\x,\x+1}]^{2}-\mathscr{R}_{\lambda}[\eta]^{2}|\lesssim N^{2}\sum_{\x\in\mathbb{L}}|\mathscr{R}_{\lambda}[\eta^{\x,\x+1}]-\mathscr{R}_{\lambda}[\eta]|^{2}+\sum_{\x\in\mathbb{L}}|\mathscr{R}_{\lambda}[\eta^{\x,\x+1}]+\mathscr{R}_{\lambda}[\eta]|^{2}.
\end{align*}
Take the first term on the far RHS above. When we take $\E^{\sigma,\mathbb{L}}$ of it, we get {\small$\lesssim N^{2}\mathfrak{D}^{\sigma,\mathbb{L}}[(\lambda-\mathscr{L}^{\mathrm{free},\mathbb{L}}_{N})^{-1}\mathfrak{A}]$}, and by Lemma \ref{lemma:l2op}, this is {\small$\lesssim\|\mathfrak{A}\|_{\mathrm{H}^{-1}}^{2}$}. Now take the last term in the previous display. When we take $\E^{\sigma,\mathbb{L}}$ of it, because $\mathbb{P}^{\sigma,\mathbb{L}}$ is invariant under spin-swaps, we just get {\small$\lesssim|\mathbb{L}|\E^{\sigma,\mathbb{L}}|(\lambda-\mathscr{L}^{\mathrm{free},\mathbb{L}}_{N})^{-1}\mathfrak{A}|^{2}$}, which is {\small$\lesssim|\mathbb{L}|\tau\|\mathfrak{A}\|_{\mathrm{H}^{-1}}^{2}$} by Lemma \ref{lemma:l2op} and $\lambda\asymp\tau^{-1}$. Ultimately, since $|\mathbb{L}|\lesssim N^{1/2}$ and $\tau\lesssim N^{-4/3}$ by assumption in Proposition \ref{prop:kv}, we have
\begin{align}
\E^{\sigma,\mathbb{L}}|\mathscr{L}^{\mathfrak{d},\mathbb{L}}_{N,\mathrm{A}}[(\lambda-\mathscr{L}^{\mathrm{free},\mathbb{L}}_{N})^{-1}\mathfrak{A}]^{2}|\lesssim\|\mathfrak{A}\|_{\mathrm{H}^{-1}}^{2}\lesssim N^{-2}n^{-1}\sup_{i=1,\ldots,n}|\mathbb{I}_{i}|^{2},\label{eq:kv5a}
\end{align}
where the last bound is again by Proposition 7 of \cite{GJ15}. Next, we claim
\begin{align*}
|\mathscr{R}_{\lambda}\mathscr{L}^{\mathfrak{d},\mathbb{L}}_{N,\mathrm{A}}\mathscr{R}_{\lambda}|&\lesssim N\sum_{\x\in\mathbb{L}}|\mathscr{R}_{\lambda}[\eta]||\mathscr{R}_{\lambda}[\eta^{\x,\x+1}]-\mathscr{R}_{\lambda}[\eta]|\lesssim |\mathbb{L}|\mathscr{R}_{\lambda}[\eta]^{2}+N^{2}\sum_{\x\in\mathbb{L}}|\mathscr{R}_{\lambda}[\eta^{\x,\x+1}]-\mathscr{R}_{\lambda}[\eta]|^{2}.
\end{align*}
The first bound above is by definition of {\small$\mathscr{L}^{\mathfrak{d},\mathbb{L}}_{N,\mathrm{A}}$} (see the proof of Lemma \ref{lemma:lpprod}) and triangle inequality. The second bound is Schwarz, i.e. the upper bound {\small$|ab|\lesssim N^{-1}|a|^{2}+N|b|^{2}$}. When we take $\E^{\sigma,\mathbb{L}}$ of the first term on the RHS, we get {\small$\lesssim|\mathbb{L}|\E^{\sigma,\mathbb{L}}|(\lambda-\mathscr{L}^{\mathrm{free},\mathbb{L}}_{N})^{-1}\mathfrak{A}|^{2}\lesssim|\mathbb{L}|\tau\|\mathfrak{A}\|_{\mathrm{H}^{-1}}^{2}$} by Lemma \ref{lemma:l2op} and by {\small$\lambda\asymp\tau^{-1}$}. This is {\small$\lesssim\|\mathfrak{A}\|_{\mathrm{H}^{-1}}^{2}$}, because $|\mathbb{L}|\lesssim N^{1/2}$ and $\tau\lesssim N^{-4/3}$ by assumption. If we take $\E^{\sigma,\mathbb{L}}$ on the last term on the RHS above, we again get {\small$\lesssim N^{2}\mathfrak{D}^{\sigma,\mathbb{L}}[(\lambda-\mathscr{L}^{\mathrm{free},\mathbb{L}}_{N})^{-1}\mathfrak{A}]$}, which by Lemma \ref{lemma:l2op} is {\small$\lesssim\|\mathfrak{A}\|_{\mathrm{H}^{-1}}^{2}$}. Ultimately, we get 
\begin{align*}
\E^{\sigma,\mathbb{L}}|\{(\lambda-\mathscr{L}^{\mathrm{free},\mathbb{L}}_{N})^{-1}\mathfrak{A}\}\cdot\mathscr{L}^{\mathfrak{d},\mathbb{L}}_{N,\mathrm{A}}\{(\lambda-\mathscr{L}^{\mathrm{free},\mathbb{L}}_{N})^{-1}\mathfrak{A}\}|\lesssim\|\mathfrak{A}\|_{\mathrm{H}^{-1}}^{2}\lesssim N^{-2}n^{-1}\sup_{i=1,\ldots,n}|\mathbb{I}_{i}|^{2},
\end{align*}
where the last bound is by Proposition 7 in \cite{GJ15}. If we combine this with \eqref{eq:kv3c} and \eqref{eq:kv5a}, we get {\small$\E^{\sigma,\mathbb{L}}|\Gamma_{1,\mathfrak{d}}|\lesssim N^{-2}n^{-1}\sup_{i}|\mathbb{I}_{i}|^{2}$}. If we combine this with \eqref{eq:kv3a} and \eqref{eq:kv4}, we get {\small$\E^{\sigma,\mathbb{L}}\Gamma_{1}\lesssim N^{-2}n^{-1}\sup_{i}|\mathbb{I}_{i}|^{2}$} and thus
\begin{align*}
\E|\tau^{-1}\mathbf{m}_{\tau}|^{2}&\lesssim_{\e}N^{C\e}\tau^{-1}\left\{N^{-2}n^{-1}\sup_{i=1,\ldots,n}|\mathbb{I}_{i}|^{2}\right\}^{\frac{2}{2+\e}}, \quad\mathrm{where} \ C=\mathrm{O}(1).
\end{align*}
Since $n,|\mathbb{I}_{i}|\geq1$ and {\small$\tau\gtrsim N^{-10}$}, the RHS of the above display is {\small$\lesssim N^{{C\e}}N^{-2}\tau^{-1}n^{-1}\sup_{i}|\mathbb{I}_{i}|^{2}$}. If we take $\e>0$ small, we get \eqref{eq:kv2b}. Thus, we have \eqref{eq:kv2a}-\eqref{eq:kv2b}, which completes the proof. \qed
\begin{remark}\label{remark:kvim}
\fsp Let us extend Proposition \ref{prop:kv} to get key ingredients for deriving \eqref{eq:kpz} when the speed of $\mathfrak{d}[\cdot]$ is $N^{\alpha}$ with $\alpha\in[1,3/2]$. In particular, for this remark only, we assume more generally that the speed of $\mathfrak{d}[\cdot]$ is $\mathrm{O}(N^{\alpha})$. As this remark will not be used in the proof of any results in this paper, the reader is invited to skip it, at least in a first reading. (For this reason also, we avoid a completely detailed discussion and focus on the main points.)

Set $p=2^{q}$ with $q\geq1$ an integer and {\small$\wt{\mathfrak{p}}^{\mathbb{L}}_{\t}:=\mathfrak{p}^{\mathbb{L}}_{\t}-1$}. As in the beginning of the proof of Lemma \ref{lemma:lpprod}, we have
\begin{align*}
\tfrac{\d}{\d\t}\E^{\sigma,\mathbb{L}}(\wt{\mathfrak{p}}^{\mathbb{L}}_{\t})^{p}&=p\E^{\sigma,\mathbb{L}}(\wt{\mathfrak{p}}^{\mathbb{L}}_{\t})^{p-1}\partial_{\t}\wt{\mathfrak{p}}^{\mathbb{L}}_{\t}=p\E^{\sigma,\mathbb{L}}(\wt{\mathfrak{p}}^{\mathbb{L}}_{\t})^{p-1}\partial_{\t}{\mathfrak{p}}^{\mathbb{L}}_{\t}\\
&=p\E^{\sigma,\mathbb{L}}\mathfrak{p}^{\mathbb{L}}_{\t}\mathscr{L}^{\mathbb{L}}_{N,\mathrm{S}}(\wt{\mathfrak{p}}^{\mathbb{L}}_{\t})^{p-1}+p\E^{\sigma,\mathbb{L}}\mathfrak{p}^{\mathbb{L}}_{\t}\mathscr{L}^{\mathrm{free},\mathbb{L}}_{N,\mathrm{A}}(\wt{\mathfrak{p}}^{\mathbb{L}}_{\t})^{p-1}+p\E^{\sigma,\mathbb{L}}\mathfrak{p}^{\mathbb{L}}_{\t}\mathscr{L}^{\mathfrak{d},\mathbb{L}}_{N,\mathrm{A}}(\wt{\mathfrak{p}}^{\mathbb{L}}_{\t})^{p-1},
\end{align*}
where we refer to the proof of Lemma \ref{lemma:lpprod} for notation. First, note {\small$\E^{\sigma,\mathbb{L}}\mathfrak{p}^{\mathbb{L}}_{\t}\mathscr{L}^{\mathbb{L}}_{N,\mathrm{S}}(\wt{\mathfrak{p}}^{\mathbb{L}}_{\t})^{p-1}=\E^{\sigma,\mathbb{L}}\wt{\mathfrak{p}}^{\mathbb{L}}_{\t}\mathscr{L}^{\mathbb{L}}_{N,\mathrm{S}}(\wt{\mathfrak{p}}^{\mathbb{L}}_{\t})^{p-1}$}, since {\small$\mathscr{L}^{\mathbb{L}}_{N,\mathrm{S}}$} is self-adjoint with respect to $\mathbb{P}^{\sigma,\mathbb{L}}$ and thus its adjoint is a generator that vanishes on constants. We also have {\small$\E^{\sigma,\mathbb{L}}\mathfrak{p}^{\mathbb{L}}_{\t}\mathscr{L}^{\mathbb{L}}_{N,\mathrm{A}}(\wt{\mathfrak{p}}^{\mathbb{L}}_{\t})^{p-1}=\E^{\sigma,\mathbb{L}}\wt{\mathfrak{p}}^{\mathbb{L}}_{\t}\mathscr{L}^{\mathbb{L}}_{N,\mathrm{S}}(\wt{\mathfrak{p}}^{\mathbb{L}}_{\t})^{p-1}$}, since {\small$\mathscr{L}^{\mathbb{L}}_{N,\mathrm{A}}$} also has an adjoint that vanishes on constants. Thus, if we follow the proof of Lemma \ref{lemma:lpprod}, we obtain the estimates below for some $\kappa\gtrsim1$ (which may change from line to line but will always remain uniformly bounded and positive):
\begin{align*}
p\E^{\sigma,\mathbb{L}}\mathfrak{p}^{\mathbb{L}}_{\t}\mathscr{L}^{\mathbb{L}}_{N,\mathrm{S}}(\wt{\mathfrak{p}}^{\mathbb{L}}_{\t})^{p-1}&\leq-\kappa N^{2}\mathfrak{D}^{\sigma,\mathbb{L}}[(\wt{\mathfrak{p}}^{\mathbb{L}}_{\t})^{p/2}]\\
|\E^{\sigma,\mathbb{L}}\wt{\mathfrak{p}}^{\mathbb{L}}_{\t}\mathscr{L}^{\mathrm{free},\mathbb{L}}_{N,\mathrm{A}}(\wt{\mathfrak{p}}^{\mathbb{L}}_{\t})^{p-1}|&\lesssim_{p}N^{\frac32}\mathfrak{D}^{\sigma,\mathbb{L}}[(\wt{\mathfrak{p}}^{\mathbb{L}}_{\t})^{p/2}].
\end{align*}
Similarly, if we follow the proof of Lemma \ref{lemma:lpprod}, then we get the following for any $\delta>0$ small:
\begin{align*}
|p\E^{\sigma,\mathbb{L}}\mathfrak{p}^{\mathbb{L}}_{\t}\mathscr{L}^{\mathfrak{d},\mathbb{L}}_{N,\mathrm{A}}(\wt{\mathfrak{p}}^{\mathbb{L}}_{\t})^{p-1}|&\lesssim_{p}\delta^{-1}N^{2\alpha-2}|\mathbb{L}|\E^{\sigma,\mathbb{L}}|\mathfrak{p}^{\mathbb{L}}_{\t}|^{p}+\delta N^{2}\mathfrak{D}^{\sigma,\mathbb{L}}[(\wt{\mathfrak{p}}^{\mathbb{L}}_{\t})^{p/2}].
\end{align*}
(Indeed, in the upper bound on {\small$\E^{\sigma,\mathbb{L}}\mathfrak{p}^{\mathbb{L}}_{\t}\mathscr{L}^{\mathfrak{d},\mathbb{L}}_{N,\mathrm{A}}(\wt{\mathfrak{p}}^{\mathbb{L}}_{\t})^{p-1}$} in the proof of Lemma \ref{lemma:lpprod}, the same argument works exactly, except the prefactor of $N$ there is replaced by $N^{\alpha}$, and we use Schwarz to give the Dirichlet form term $N^{2}$ and the $\mathrm{L}^{p}$-norm $N^{2\alpha-2}$.) Next, by the spectral gap for $\mathbb{P}^{\sigma,\mathbb{L}}$, we get {\small$\mathfrak{D}^{\sigma,\mathbb{L}}[(\wt{\mathfrak{p}}^{\mathbb{L}}_{\t})^{p/2}]\gtrsim|\mathbb{L}|^{-2}\E^{\sigma,\mathbb{L}}|(\wt{\mathfrak{p}}^{\mathbb{L}}_{\t})^{p/2}-\E^{\sigma,\mathbb{L}}(\wt{\mathfrak{p}}^{\mathbb{L}}_{\t})^{p/2}|^{2}$}. If we now let $|\mathbb{L}|\lesssim N^{\beta}$ with $\beta>0$ to be determined shortly, we ultimately get
\begin{align*}
\tfrac{\d}{\d\t}\E^{\sigma,\mathbb{L}}(\wt{\mathfrak{p}}^{\mathbb{L}}_{\t})^{p}&\leq-\kappa N^{2}\mathfrak{D}^{\sigma,\mathbb{L}}[(\wt{\mathfrak{p}}^{\mathbb{L}}_{\t})^{p/2}]+\mathrm{O}_{p}(N^{2\alpha-2}|\mathbb{L}|\E^{\sigma,\mathbb{L}}|\mathfrak{p}^{\mathbb{L}}_{\t}|^{p})\\
&\leq-\kappa N^{2}\mathfrak{D}^{\sigma,\mathbb{L}}[(\wt{\mathfrak{p}}^{\mathbb{L}}_{\t})^{p/2}]+\mathrm{O}_{p}(N^{2\alpha-2+\beta}\E^{\sigma,\mathbb{L}}|\wt{\mathfrak{p}}^{\mathbb{L}}_{\t}|^{p})+\mathrm{O}_{p}(N^{2\alpha-2+\beta})\\
&\leq-\kappa N^{2-2\beta}\E^{\sigma,\mathbb{L}}|\wt{\mathfrak{p}}^{\mathbb{L}}_{\t}|^{p}+\mathrm{O}(N^{2-2\beta}|\E^{\sigma,\mathbb{L}}(\wt{\mathfrak{p}}^{\mathbb{L}}_{\t})^{p/2}|^{2})+\mathrm{O}_{p}(N^{2\alpha-2+\beta}\E^{\sigma,\mathbb{L}}|\wt{\mathfrak{p}}^{\mathbb{L}}_{\t}|^{p})+\mathrm{O}_{p}(N^{2\alpha-2+\beta}).
\end{align*}
Now, if {\small$\beta<\frac13(4-2\alpha)$}, the first term in the last line dominates the third. If we then define {\small$m_{\t,p}:=\E^{\sigma,\mathbb{L}}(\wt{\mathfrak{p}}^{\mathbb{L}}_{\t})^{p}$}, we get the following by the previous display and Gronwall:
\begin{align*}
m_{\t,p}&\lesssim_{p}{\textstyle\int_{0}^{\t}}\exp\left\{-\kappa N^{2-2\beta}(\t-\s)\right\}\left(N^{2-2\beta}|m_{\s,p/2}|^{2}+N^{2\alpha-2+\beta}\right)\d\s\\
&\lesssim\sup_{\s\geq0}|m_{\s,p/2}|^{2}+N^{2\alpha+3\beta-4}.
\end{align*}
Since the second line above does not depend on $\t\geq0$, we have the following in which $m_{p}:=\sup_{\t\geq0}m_{\t,p}$:
\begin{align*}
m_{p}\lesssim_{p} m_{p/2}^{2}+N^{2\alpha+3\beta-4}.
\end{align*}
Again, if {\small$\beta<\frac13(4-2\alpha)$}, then the last term in the previous display is $\lesssim N^{-\zeta}$ for some $\zeta>0$. In particular, as long as {\small$m_{p/2}\lesssim1$}, the previous estimate gives {\small$m_{p}\lesssim m_{p/2}+N^{-\zeta}$}; we also note {\small$m_{1}=\E^{\sigma,\mathbb{L}}[\mathfrak{p}^{\mathbb{L}}_{\t}-1]=0$}. Thus, if we iterate the bound {\small$m_{p}\lesssim m_{p/2}+N^{-\zeta}$} starting with $p/2=1$ on the RHS, then for $p\lesssim1$, we get {\small$m_{p}\lesssim N^{-\zeta}$}. So, the underlying assumption {\small$m_{p/2}\lesssim 1$} holds, and we get
\begin{align*}
\sup_{\t\geq0}\E^{\sigma,\mathbb{L}}|\mathfrak{p}^{\mathbb{L}}_{\t}-1|^{p}\lesssim_{p} N^{-\upsilon}
\end{align*}
for some $\upsilon\gtrsim1$. (We emphasize that this estimate is global in time, so it actually implies {\small$\E^{\sigma,\mathbb{L}}|\mathfrak{p}^{\mathbb{L}}_{\infty}-1|^{p}\lesssim N^{-\upsilon}$}, where {\small$\mathfrak{p}^{\mathbb{L}}_{\infty}$} is the Radon-Nikodym derivative of the invariant measure of $\eta^{\mathbb{L}}$ with respect to $\mathbb{P}^{\sigma,\mathbb{L}}$.) At this point, we can follow the dynamical Ito formula argument in the proof of Proposition \ref{prop:kv} to establish a version of Proposition \ref{prop:kv} if the speed of $\mathfrak{d}[\cdot]$ is $\mathrm{O}(N^{\alpha})$ and the size of the lattice $\mathbb{L}$ is $|\mathbb{L}|\lesssim N^{\beta}$ with {\small$\beta<\frac13(4-2\alpha)$}. (Note that $\beta$ can be strictly positive for all $\alpha<2$; our calculations in this remark are therefore perhaps useful even for the critical KPZ scaling $\alpha=3/2$. As for the case of $\alpha<5/4$, we can take $\beta=1/2+\e$ for some small $\e$. Thus, we obtain a version of Proposition \ref{prop:kv} where the length-scale of the localized system is $\gg N^{1/2+\e}$, meaning essentially that we can take $n\gg N^{1/2+\e}$ and $\tau\gg N^{-1+\e}$ therein; the constraint of $\tau$ is determined using the diffusive-scaling relation $(N^{2}\tau)^{1/2}\sim|\mathbb{L}|$, since the particles evolve diffusively. The resulting upper bound on the RHS of \eqref{eq:kv} then becomes $\lesssim N^{-3/2-\e}$. Moreover, if $\alpha<5/4$, then the functions $\mathfrak{f}_{i}$ we are interested in have a scaling factor of $N^{\alpha-1/2}\ll N^{3/4}$, which becomes $\ll N^{3/2}$ after squaring. Thus, we arrive ultimately at an error estimate of $\ll N^{3/2}N^{-3/2-\e}\ll1$, so our methods apply for all $\alpha\in[1,5/4)$.)
\end{remark}
\subsubsection{A second Kipnis-Varadhan inequality}
Adopt the same setting as in Proposition \ref{prop:kv}. The following result is meant to go beyond the factor of $|\mathbb{I}_{i}|^{2}$ on the RHS of \eqref{eq:kv}. In particular, we eventually meet functions $\mathfrak{f}_{i}$ (as in Proposition \ref{prop:kv}) that are of the product form $\mathfrak{f}_{i,1}\mathfrak{f}_{i,2}$, where $\mathfrak{f}_{i,1}$ is fluctuating with ``small support" and $\mathfrak{f}_{i,2}$ is bounded with ``large support". In this case, the corresponding factor $|\mathbb{I}_{i}|^{2}$ in \eqref{eq:kv} is bad, even if the fluctuating factor has small support. The result below cures this deficiency of Proposition \ref{prop:kv}. (This ingredient is seemingly new to this paper, e.g. it does not show up in \cite{Y23}, and it is important for a variety of technical reasons ultimately stemming from the lack of explicit invariant measures.)
\begin{prop}\label{prop:kv2}
\fsp Suppose $\mathfrak{f}_{1},\mathfrak{f}_{2}:\{\pm1\}^{\mathbb{L}}\to\R$ satisfy the following conditions:
\begin{enumerate}
\item We have $\E^{\sigma,\mathbb{L}}\mathfrak{f}_{1}=0$ for all $\sigma\in[-1,1]$ and $\mathfrak{f}_{1}[\eta],\mathfrak{f}_{2}[\eta]=\mathrm{O}(1)$ for all $\eta\in\{\pm1\}^{\mathbb{L}}$.
\item Suppose $\mathfrak{f}_{1}[\eta]$ depends only on $\eta_{\x}$ for $\x\in\mathbb{I}_{1}$ and $\mathfrak{f}_{2}[\eta]$ depends only on $\eta_{\x}\in\mathbb{I}_{2}$ for all $\eta\in\{\pm1\}^{\mathbb{L}}$, where the supports $\mathbb{I}_{1},\mathbb{I}_{2}\subseteq\mathbb{L}$ are disjoint sub-intervals.
\end{enumerate}
Now, fix $0<\tau\lesssim N^{-4/3}$ and assume that {\small$|\mathbb{L}|\lesssim N^{1/3+\e_{\mathrm{reg}}}$}. Assume that the distribution of {\small$\eta^{\mathbb{L}}_{\s}$} at time $\s=0$ is given by $\mathbb{P}^{\sigma,\mathbb{L}}$ for some $\sigma\in[-1,1]$. For any $\delta>0$, we have
\begin{align}
\E\left(\left\{\tau^{-1}\int_{0}^{\tau}\mathfrak{f}_{1}[\eta^{\mathbb{L}}_{\s}]\mathfrak{f}_{2}[\eta^{\mathbb{L}}_{\s}]\d\s\right\}^{2}\right)&\lesssim_{\delta}N^{\delta}N^{-2}\tau^{-1}|\mathbb{I}_{1}|^{2}.\label{eq:kv2}
\end{align}
\end{prop}
\begin{proof}
We can assume $\mathfrak{f}_{2}$ never vanishes by a standard approximation (we will never use control on how far it is away from $0$); this is only for convenience of presentation. We follow the proof of Proposition \ref{prop:kv}, except we replace {\small$\mathfrak{A}:=n^{-1}(\mathfrak{f}_{1}+\ldots+\mathfrak{f}_{n})$} therein by $\mathfrak{f}_{1}\mathfrak{f}_{2}$ here. This gives the following for any $\e>0$, in which $C=\mathrm{O}(1)$ and we use the $\mathrm{H}^{-1}$-norm from Lemma \ref{lemma:l2op}:
\begin{align*}
\mathrm{LHS}\eqref{eq:kv2}\lesssim_{\e}N^{{C}\e}\tau^{-1}(\|\mathfrak{f}_{1}\mathfrak{f}_{2}\|_{\mathrm{H}^{-1}}^{2})^{\frac{2}{2+\e}}.
\end{align*}
In particular, to prove \eqref{eq:kv2}, it suffices to prove 
\begin{align}
\|\mathfrak{f}_{1}\mathfrak{f}_{2}\|_{\mathrm{H}^{-1}}^{2}\lesssim N^{-2}|\mathbb{I}_{1}|^{2}.\label{eq:kv2I}
\end{align}
Indeed, since $|\mathbb{I}_{1}|\geq1$, if we plug \eqref{eq:kv2I} into the RHS of the above display, we get a bound on the LHS of \eqref{eq:kv2} of {\small$\lesssim_{\e}N^{{C}\e}N^{-2}\tau^{-1}|\mathbb{I}_{1}|^{2}$}, and if we take $\e>0$ small enough, \eqref{eq:kv2} follows. So, we prove \eqref{eq:kv2I}. We note that
\begin{align}
\|\mathfrak{f}_{1}\mathfrak{f}_{2}\|_{\mathrm{H}^{-1}}^{2}&\lesssim\sup_{\mathfrak{b}:\{\pm1\}^{\mathbb{L}}\to\R}\left\{2\E^{\sigma,\mathbb{L}}\mathfrak{f}_{1}\mathfrak{f}_{2}\mathfrak{b}+\E^{\sigma,\mathbb{L}}\mathfrak{b}\mathscr{L}^{\mathbb{L}}_{N,\mathrm{S}}\mathfrak{b}\right\}.\label{eq:kv2II}
\end{align}
(For this, we use the fact that the generator {\small$\mathscr{L}^{\mathrm{free},\mathbb{L}}_{N}:=\mathscr{L}^{\mathbb{L}}_{N,\mathrm{S}}+\mathscr{L}^{\mathrm{free},\mathbb{L}}_{N,\mathrm{A}}$} is such that {\small$\mathscr{L}^{\mathbb{L}}_{N,\mathrm{S}}$} is self-adjoint and {\small$\mathscr{L}^{\mathrm{free},\mathbb{L}}_{N,\mathrm{A}}$} is anti-symmetric with respect to $\E^{\sigma,\mathbb{L}}$ integration. So, it does not matter if we have {\small$\mathscr{L}^{\mathbb{L}}_{N,\mathrm{S}}$} or {\small$\mathscr{L}^{\mathrm{free},\mathbb{L}}_{N}$} on the RHS above.) Next, we have the standard Dirichlet form computation (see Appendix 1.10 in \cite{KL}) 
\begin{align}
\E^{\sigma,\mathbb{L}}\mathfrak{b}[\eta]\mathscr{L}^{\mathbb{L}}_{N,\mathrm{S}}\mathfrak{b}[\eta]&=-\tfrac12N^{2}\sum_{\x\in\mathbb{L}}\E^{\sigma,\mathbb{L}}|\mathscr{L}_{\x}\mathfrak{b}[\eta]|^{2}\leq-\tfrac12N^{2}\sum_{\x,\x+1\in\mathbb{I}_{1}}\E^{\sigma,\mathbb{L}}|\mathscr{L}_{\x}\mathfrak{b}[\eta]|^{2}=:-\mathscr{D}_{N,\mathbb{I}_{1}}[\mathfrak{b}].\label{eq:kv2III}
\end{align}
(The last identity is just convenient notation for a Dirichlet form restricted to $\mathbb{I}_{1}$.) The last two displays now give
\begin{align}
\|\mathfrak{f}_{1}\mathfrak{f}_{2}\|_{\mathrm{H}^{-1}}^{2}&\lesssim\sup_{\mathfrak{b}:\{\pm1\}^{\mathbb{L}}\to\R}\left\{2\E^{\sigma,\mathbb{L}}\mathfrak{f}_{1}\mathfrak{f}_{2}\mathfrak{b}-\mathscr{D}_{N,\mathbb{I}_{1}}[\mathfrak{b}]\right\}.\nonumber
\end{align}
Since $\mathfrak{f}_{2}$ does not depend on $\eta_{\x}$ for any $\x\in\mathbb{I}_{1}$, we have {\small$\mathscr{D}_{N,\mathbb{I}_{1}}[\mathfrak{b}]=\mathfrak{f}_{2}^{-2}\mathscr{D}_{N,\mathbb{I}_{1}}[\mathfrak{b}\mathfrak{f}_{2}]$}; recall that $\mathfrak{f}_{2}$ never vanishes. Since {\small$\mathscr{D}_{N,\mathbb{I}_{1}}[\mathfrak{b}\mathfrak{f}_{2}]\geq0$} and $\mathfrak{f}_{2}=\mathrm{O}(1)$, we have {\small$\mathscr{D}_{N,\mathbb{I}_{1}}[\mathfrak{b}]=\mathfrak{f}_{2}^{-2}\mathscr{D}_{N,\mathbb{I}_{1}}[\mathfrak{b}\mathfrak{f}_{2}]\geq\kappa\mathscr{D}_{N,\mathbb{I}_{1}}[\mathfrak{b}\mathfrak{f}_{2}]$} for some $\kappa\gtrsim1$. Thus, 
\begin{align}
\|\mathfrak{f}_{1}\mathfrak{f}_{2}\|_{\mathrm{H}^{-1}}^{2}&\lesssim\sup_{\mathfrak{b}:\{\pm1\}^{\mathbb{L}}\to\R}\left\{2\E^{\sigma,\mathbb{L}}\mathfrak{f}_{1}\mathfrak{f}_{2}\mathfrak{b}-\kappa\mathscr{D}_{N,\mathbb{I}_{1}}[\mathfrak{b}\mathfrak{f}_{2}]\right\}.\nonumber
\end{align}
Since $\mathfrak{f}_{2}$ never vanishes, the map $\mathfrak{b}\mapsto\mathfrak{f}_{2}\mathfrak{b}$ is a bijection on the space of functions $\{\pm1\}^{\mathbb{L}}\to\R$. Thus, we have 
\begin{align*}
\|\mathfrak{f}_{1}\mathfrak{f}_{2}\|_{\mathrm{H}^{-1}}^{2}&\lesssim\sup_{\mathfrak{b}':\{\pm1\}^{\mathbb{L}}\to\R}\left\{2\E^{\sigma,\mathbb{L}}\mathfrak{f}_{1}\mathfrak{b}'-\kappa\mathscr{D}_{N,\mathbb{I}_{1}}[\mathfrak{b}']\right\}.
\end{align*}
By convexity of $\mathfrak{b}'\mapsto\mathscr{D}_{N,\mathbb{I}_{1}}[\mathfrak{b}']$ (see Appendix 1.10 in \cite{KL}), we can replace $\mathfrak{b}'$ by its conditional expectation conditioning on $\eta_{\x}$ for $\x\in\mathbb{I}_{1}$ (note also that $\mathfrak{f}_{1}$ depends only on such $\eta_{\x}$). After doing so, we can restrict the supremum on the RHS to functions $\{\pm1\}^{\mathbb{I}_{1}}\to\R$, and the previous display becomes
\begin{align*}
\|\mathfrak{f}_{1}\mathfrak{f}_{2}\|_{\mathrm{H}^{-1}}^{2}&\lesssim\sup_{\mathfrak{f}:\{\pm1\}^{\mathbb{I}}\to\R}\left\{2\E^{\sigma,\mathbb{L}}\mathfrak{f}_{1}\mathfrak{f}-\kappa\mathscr{D}_{N,\mathbb{I}_{1}}[\mathfrak{f}]\right\}.
\end{align*}
We can now follow the proof of Proposition 7 in \cite{GJ15} to get \eqref{eq:kv2I} since $\mathfrak{f}_{1}=\mathrm{O}(1)$.
\end{proof}
\subsubsection{A space-averaging estimate}
We now give a bound on spatial averages as in Proposition \ref{prop:kv} but without the time-averaging (i.e. take $\tau\to0$ in Proposition \ref{prop:kv}). It will not give as good an upper bound since we do not have time-averaging, but it also holds at a much sharper exponential scale (rather than second moment like in \eqref{eq:kv}). Also, because there is no time-averaging, there is no need to deal with the lack of invariant measures, so that Lemma \ref{lemma:azuma} actually appears in \cite{Y23} as Lemma 8.10 therein. (We prove it anyway since it is short.)
\begin{lemma}\label{lemma:azuma}
\fsp Suppose {\small$\{\mathfrak{f}_{i}\}_{i=1}^{n}$} is a collection of functions $\{\pm1\}^{\mathbb{L}}\to\R$ satisfying the following properties.
\begin{enumerate}
\item For any $\mathfrak{f}_{i}$, we have $\E^{\sigma,\mathbb{L}}\mathfrak{f}_{i}=0$ for all $\sigma\in[-1,1]$ and $\sup_{i}|\mathfrak{f}_{i}[\eta]|\lesssim1$ for all $\eta\in\{\pm1\}^{\mathbb{L}}$.
\item There exist pairwise disjoint sub-intervals {\small$\{\mathbb{I}_{i}\}_{i=1}^{n}$} in $\mathbb{L}$ such that $\mathfrak{f}_{i}[\eta]$ depends only on $\eta_{\z}$ for $\z\in\mathbb{I}_{i}$.
\end{enumerate}
Fix any $\sigma_{0}\in[-1,1]$. There exists $c>0$ universal such that for any $K>0$, we have
\begin{align}
\mathbb{P}^{\sigma_{0},\mathbb{L}}\left\{\left|\tfrac{1}{n}\sum_{i=1}^{n}\mathfrak{f}_{i}[\eta]\right|\geq Kn^{-\frac12}\right\}&\lesssim\exp\left\{-cK^{2}\right\}.\label{eq:azuma}
\end{align}
\end{lemma}
\begin{proof}
We claim that the sequence $m\mapsto\mathfrak{f}_{1}[\eta]+\ldots+\mathfrak{f}_{m}[\eta]$ is a martingale with respect to $\mathbb{P}^{\sigma,\mathbb{L}}$ and the filtration $m\mapsto\mathscr{F}_{m}$, where $\mathscr{F}_{m}$ is generated by $\eta_{\x}$ for $\x\in\mathbb{I}_{1}\cup\ldots\cup\mathbb{I}_{m}$. \eqref{eq:azuma} would then follow by Azuma's martingale inequality. To prove the martingale property, it suffices to recall that $\mathbb{I}_{i}$ are pairwise disjoint, that $\mathfrak{f}_{i}$ all vanish under all canonical measure expectations, and that conditioning $\E^{\sigma_{0},\mathbb{L}}$ on $\eta_{\x}$ for $\x$ in some subset $\mathbb{J}$ of $\mathbb{L}$ induces $\E^{\sigma_{1},\mathbb{L}\setminus\mathbb{J}}$ on $\{\pm1\}^{\mathbb{L}\setminus\mathbb{J}}$ (because conditioning always pushes forward uniform measure to uniform measure, and $\mathbb{P}^{\sigma,\mathbb{I}}$ measures are all uniform measures).
\end{proof}
\subsection{Combining probabilistic with analytic estimates}
We now use the results so far in this section to derive estimates that we will use to prove Theorem \ref{theorem:bgp} and Proposition \ref{prop:hl}. We return to the {\small$\eta_{\t}$} process of interest and use the localized process {\small$\eta^{\mathbb{L}}$} (for suitable choices of {\small$\mathbb{L}$}) when relevant. First, recall $\e_{\mathrm{ap}}$ from \eqref{eq:tap} is such that {\small$|\mathbf{Y}^{N}|\lesssim N^{\e_{\mathrm{ap}}}$} deterministically (see \eqref{eq:tst}).
\begin{lemma}\label{lemma:proban}
\fsp Fix a sub-interval $\mathbb{L}\subseteq\mathbb{T}_{N}$. Suppose $\mathfrak{f}_{i}:\{\pm1\}^{\mathbb{T}_{N}}\to\R$ are functions such that:
\begin{enumerate}
\item For any $\eta\in\{\pm1\}^{\mathbb{T}_{N}}$, the quantity $\mathfrak{f}_{i}[\eta]$ depends only on $\eta_{\x}$ for $\x\in\mathbb{I}_{i}\subseteq\mathbb{L}$, where $\mathbb{I}_{i}$ are pairwise disjoint.
\item For every $i$ and $\sigma\in[-1,1]$, we have $\E^{\sigma,\mathbb{L}}\mathfrak{f}_{i}=0$ (recall $\E^{\sigma,\mathbb{L}}$ from Definition \ref{definition:can}).
\end{enumerate}
Fix $0<\tau\lesssim N^{-4/3}$ and {\small$|\mathbb{L}|\lesssim N^{1/3+\e_{\mathrm{reg}}}$}. Assume $N^{2}\tau\lesssim|\mathbb{L}|^{2}$. For any $\delta>0$ and integer $n>0$, we have 
\begin{align}
&\E\left\{\sup_{\t\in[0,1]}\sup_{\x\in\mathbb{T}_{N}}\left|{\int_{0}^{\t}}\sum_{\y\in\mathbb{T}_{N}}\mathbf{H}^{N}_{\s,\t,\x,\y}\times N^{\frac12}\left\{\tfrac{1}{\tau}\int_{0}^{\tau}\tfrac{1}{n}\sum_{i=1}^{n}\mathfrak{f}_{i}[\tau_{\y}\eta_{\s+\r}]\d\r\right\}\mathbf{Y}^{N}_{\s,\y}\d\s\right|\right\}\label{eq:probanIa}\\
&\lesssim_{\delta} N^{\e_{\mathrm{ap}}}\left(N^{\delta}N^{-\frac54}|\mathbb{L}|^{3}n^{-\frac34}\sup_{i=1,\ldots,n}\|\mathfrak{f}_{i}\|_{\infty}^{\frac32}+N^{\delta}N^{-\frac34}\tau^{-\frac34}n^{-\frac34}\sup_{i=1,\ldots,n}|\mathbb{I}_{i}|^{\frac32}\sup_{i=1,\ldots,n}\|\mathfrak{f}_{i}\|_{\infty}^{\frac32}\right)^{2/3},\label{eq:probanIb}
\end{align}
where $\|\|_{\infty}$ means sup-norm over $\eta\in\{\pm1\}^{\mathbb{T}_{N}}$.
\end{lemma}
\begin{remark}
\fsp The constraints $0<\tau\lesssim N^{-4/3}$ and {\small$|\mathbb{L}|\lesssim N^{1/3+\e_{\mathrm{reg}}}$} let us use Proposition \ref{prop:kv}. The constraint $N^{2}\tau\lesssim|\mathbb{L}|^{2}$ just states that the time-scale $\tau$ is ``compatible" with $\mathbb{L}$ in the sense of diffusive scaling.
\end{remark}
\begin{proof}
We first claim the following deterministic estimate for any $\t\in[0,1]$ and $\x\in\mathbb{T}_{N}$ and $\e>0$:
\begin{align*}
&\left|{\int_{0}^{\t}}\sum_{\y\in\mathbb{T}_{N}}\mathbf{H}^{N}_{\s,\t,\x,\y}\times N^{\frac12}\left\{\tfrac{1}{\tau}\int_{0}^{\tau}\tfrac{1}{n}\sum_{i=1}^{n}\mathfrak{f}_{i}[\tau_{\y}\eta^{}_{\s+\r}]\d\r\right\}\mathbf{Y}^{N}_{\s,\y}\d\s\right|\\
&=\left|{\int_{0}^{\t}}|\t-\s|^{-\frac13+\e}|\t-\s|^{\frac13-\e}\sum_{\y\in\mathbb{T}_{N}}\mathbf{H}^{N}_{\s,\t,\x,\y}\times N^{\frac12}\left\{\tfrac{1}{\tau}\int_{0}^{\tau}\tfrac{1}{n}\sum_{i=1}^{n}\mathfrak{f}_{i}[\tau_{\y}\eta^{}_{\s+\r}]\d\r\right\}\mathbf{Y}^{N}_{\s,\y}\d\s\right|\\
&\lesssim\left\{\int_{0}^{\t}|\t-\s|^{-1+3\e}\d\s\right\}^{\frac13}\left\{\int_{0}^{\t}|\t-\s|^{\frac12-\frac32\e}\sum_{\y\in\mathbb{T}_{N}}\mathbf{H}^{N}_{\s,\t,\x,\y}\times N^{\frac34}\left|\tfrac{1}{\tau}\int_{0}^{\tau}\tfrac{1}{n}\sum_{i=1}^{n}\mathfrak{f}_{i}[\tau_{\y}\eta^{}_{\s+\r}]\d\r\right|^{\frac32}|\mathbf{Y}^{N}_{\s,\y}|^{\frac32}\d\s\right\}^{\frac23}\\
&\lesssim_{\e}\left\{\int_{0}^{\t}N^{3\e+\frac32\e_{\mathrm{ap}}}\tfrac{1}{N}\sum_{\y\in\mathbb{T}_{N}}N^{\frac34}\left|\tfrac{1}{\tau}\int_{0}^{\tau}\tfrac{1}{n}\sum_{i=1}^{n}\mathfrak{f}_{i}[\tau_{\y}\eta^{}_{\s+\r}]\d\r\right|^{\frac32}\d\s\right\}^{\frac23}.
\end{align*}
The first $\lesssim$ is by H\"{o}lder with respect to space-time integration and noting that $\mathbf{H}^{N}$ is a probability measure in the forwards space variable (it is a transition probability for a random walk). The last line holds by integration, the bound {\small$\mathbf{H}^{N}_{\s,\t,\x,\y}\lesssim N^{-1+3\e}|\t-\s|^{-1/2+3/2\e}$}, which holds by interpolating the probability bound {\small$\mathbf{H}^{N}_{\s,\t,\x,\y}\leq1$} and the on-diagonal bound {\small$\mathbf{H}^{N}_{\s,\t,\x,\y}\lesssim N^{-1}|\t-\s|^{-1/2}$} from Proposition \ref{prop:hk}, and the bound {\small$\mathbf{Y}^{N}=\mathrm{O}(N^{\e_{\mathrm{ap}}})$} (see \eqref{eq:tst}). Now, we extend the integration in the last line from $[0,\t]$ to $[0,1]$. The resulting quantity is independent of $(\t,\x)$ in the first line of the previous display. Therefore, after we take expectation, the last line bounds \eqref{eq:probanIa}. Moreover, when we take expectation of the last line, by H\"{o}lder or Jensen, we can move the expectation inside the curly brackets. So, if we take $\e>0$ small and let $\delta$ in \eqref{eq:probanIb} depend on $\e$, it suffices to prove 
\begin{align}
&\int_{0}^{1}\tfrac{1}{N}\sum_{\y\in\mathbb{T}_{N}}N^{\frac34}\E|\tfrac{1}{\tau}\int_{0}^{\tau}\tfrac{1}{n}\sum_{i=1}^{n}\mathfrak{f}_{i}[\tau_{\y}\eta^{}_{\s+\r}]\d\r|^{\frac32}\d\s\nonumber\\
&\lesssim N^{\delta}N^{-\frac54}|\mathbb{L}|^{3}n^{-\frac34}\sup_{i=1,\ldots,n}\|\mathfrak{f}_{i}\|_{\infty}^{\frac32}+N^{\delta}N^{-\frac34}\tau^{-\frac34}n^{-\frac34}\sup_{i=1,\ldots,n}|\mathbb{I}_{i}|^{\frac32}\sup_{i=1,\ldots,n}\|\mathfrak{f}_{i}\|_{\infty}^{\frac32}. \label{eq:probanI1}
\end{align}
For convenience, let us introduce notation for the space-average on the LHS of \eqref{eq:probanI1} and a cutoff version:
\begin{align*}
\mathbf{Av}_{n}^{\mathfrak{f}_{\cdot}}[\eta]:=\tfrac{1}{n}\sum_{i=1}^{n}\mathfrak{f}_{i}[\eta]\quad\mathrm{and}\quad \mathbf{CutAv}_{n}^{\mathfrak{f}_{\cdot}}[\eta]:=\mathbf{Av}_{n}^{\mathfrak{f}_{\cdot}}[\eta]\mathbf{1}[|\mathbf{Av}_{n}^{\mathfrak{f}_{\cdot}}[\eta]|\lesssim N^{\e}n^{-1/2}\sup_{i=1,\ldots,n}\|\mathfrak{f}_{i}\|_{\infty}].
\end{align*}
In words, {\small$\mathbf{CutAv}^{\mathfrak{f}_{\cdot}}_{n}$} is just {\small$\mathbf{Av}^{\mathfrak{f}_{\cdot}}_{n}$} but with an a priori upper bound that is essentially the CLT-upper bound; here, we take $\e>0$ small and to be determined (but positive uniformly in $N$). We now have
\begin{align}
\mathrm{LHS}\eqref{eq:probanI1}&\lesssim\int_{0}^{1}\tfrac{1}{N}\sum_{\y\in\mathbb{T}_{N}}N^{\frac34}\E|\tfrac{1}{\tau}{\textstyle\int_{0}^{\tau}}\mathbf{CutAv}_{n}^{\mathfrak{f}_{\cdot}}[\tau_{\y}\eta^{}_{\s+\r}]\d\r|^{\frac32}\d\s\label{eq:probanI2a}\\
&+\int_{0}^{1}\tfrac{1}{N}\sum_{\y\in\mathbb{T}_{N}}N^{\frac34}\E|\tfrac{1}{\tau}{\textstyle\int_{0}^{\tau}}\{\mathbf{Av}_{n}^{\mathfrak{f}_{\cdot}}[\tau_{\y}\eta^{}_{\s+\r}]-\mathbf{CutAv}_{n}^{\mathfrak{f}_{\cdot}}[\tau_{\y}\eta^{}_{\s+\r}]\d\r\}|^{\frac32}\d\s.\label{eq:probanI2b}
\end{align}
We first control \eqref{eq:probanI2b}. We claim
\begin{align*}
\eqref{eq:probanI2b}&\lesssim\int_{0}^{1}\tfrac{1}{N}\sum_{\y\in\mathbb{T}_{N}}N^{\frac34}\tfrac{1}{\tau}{\textstyle\int_{0}^{\tau}}\E|\mathbf{Av}_{n}^{\mathfrak{f}_{\cdot}}[\tau_{\y}\eta^{}_{\s+\r}]-\mathbf{CutAv}_{n}^{\mathfrak{f}_{\cdot}}[\tau_{\y}\eta^{}_{\s+\r}]|^{\frac32}\d\r\d\s\\
&\lesssim\int_{0}^{2}\tfrac{1}{N}\sum_{\y\in\mathbb{T}_{N}}N^{\frac34}\E|\mathbf{Av}_{n}^{\mathfrak{f}_{\cdot}}[\tau_{\y}\eta^{}_{\s}]-\mathbf{CutAv}_{n}^{\mathfrak{f}_{\cdot}}[\tau_{\y}\eta^{}_{\s}]|^{\frac32}\d\s.
\end{align*}
The first line is by H\"{o}lder in the $\d\r$-average in \eqref{eq:probanI2b}. The second line holds first by moving the $\d\r$-integration outside everything. Then, for any $\r$, we note that integrating $\s+\r$ for $\s\in[0,1]$ is the same as integrating $\s$ for $\s\in[\r,\r+1]\subseteq[0,2]$, since $\r\leq\tau\leq1$ by assumption. Thus, for any $\r\in[0,\tau]$, the corresponding $\d\s$-integral is always bounded above by the $\d\s$-integration on $[0,2]$, since the integrand is non-negative. Next, we claim
\begin{align*}
&\int_{0}^{2}\tfrac{1}{N}\sum_{\y\in\mathbb{T}_{N}}N^{\frac34}\E|\mathbf{Av}_{n}^{\mathfrak{f}_{\cdot}}[\tau_{\y}\eta^{}_{\s}]-\mathbf{CutAv}_{n}^{\mathfrak{f}_{\cdot}}[\tau_{\y}\eta^{}_{\s}]|^{\frac32}\d\s\\
&\lesssim N^{\frac34}\kappa^{-1}N^{-2}|\mathbb{L}|^{3}+N^{\frac34}\kappa^{-1}\sup_{\sigma\in[-1,1]}\log\E^{\sigma,\mathbb{L}}\exp\left\{\kappa|\mathbf{Av}_{n}^{\mathfrak{f}_{\cdot}}[\eta]-\mathbf{CutAv}_{n}^{\mathfrak{f}_{\cdot}}[\eta]|^{\frac32}\right\},
\end{align*}
where $\kappa=n^{3/4}\{\sup_{i}\|\mathfrak{f}_{i}\|_{\infty}\}^{-3/2}$. This is immediate by \eqref{eq:localreduc}, because {\small$\mathbf{Av}_{n}^{\mathfrak{f}_{\cdot}}[\eta]$}, and therefore {\small$\mathbf{CutAv}_{n}^{\mathfrak{f}_{\cdot}}[\eta]$}, are determined only by $\mathfrak{f}_{i}$ and thus depend only on $\eta_{\x}$ for $\x\in\mathbb{L}$. Next, {\small$\mathbf{Av}_{n}^{\mathfrak{f}_{\cdot}}[\eta]-\mathbf{CutAv}_{n}^{\mathfrak{f}_{\cdot}}[\eta]$} is zero if and only if {\small$|\mathbf{Av}_{n}^{\mathfrak{f}_{\cdot}}[\eta]|\gtrsim N^{\e}n^{-1/2}\sup_{i}\|\mathfrak{f}_{i}\|_{\infty}$} (see the display after \eqref{eq:probanI1}). On this event, we have {\small$|\mathbf{Av}_{n}^{\mathfrak{f}_{\cdot}}[\eta]-\mathbf{CutAv}_{n}^{\mathfrak{f}_{\cdot}}[\eta]|\lesssim|\mathbf{Av}_{n}^{\mathfrak{f}_{\cdot}}[\eta]|\lesssim\sup_{i}\|\mathfrak{f}_{i}\|_{\infty}$} since {\small$\mathbf{Av}_{n}^{\mathfrak{f}_{\cdot}}[\eta]$} is an average of $\mathfrak{f}_{i}$. From this, we get
\begin{align*}
&\kappa^{-1}\log\E^{\sigma,\mathbb{L}}\exp\left\{\kappa|\mathbf{Av}_{n}^{\mathfrak{f}_{\cdot}}[\eta]-\mathbf{CutAv}_{n}^{\mathfrak{f}_{\cdot}}[\eta]|^{\frac32}\right\}\\
&\leq\kappa^{-1}\log\left[1+\E^{\sigma,\mathbb{L}}\{\mathbf{1}_{|\mathbf{Av}_{n}^{\mathfrak{f}_{\cdot}}[\eta]|\gtrsim N^{\e}n^{-1/2}\sup_{i}\|\mathfrak{f}_{i}\|_{\infty}}\mathrm{e}^{\kappa|\mathbf{Av}_{n}^{\mathfrak{f}_{\cdot}}[\eta]|^{3/2}}\}\right],
\end{align*}
Now, by Cauchy-Schwarz, we have
\begin{align*}
\E^{\sigma,\mathbb{L}}\{\mathbf{1}_{|\mathbf{Av}_{n}^{\mathfrak{f}_{\cdot}}[\eta]|\gtrsim N^{\e}n^{-1/2}\sup_{i}\|\mathfrak{f}_{i}\|_{\infty}}\mathrm{e}^{\kappa|\mathbf{Av}_{n}^{\mathfrak{f}_{\cdot}}[\eta]|^{3/2}}\}&\lesssim\sqrt{\E^{\sigma,\mathbb{L}}\mathbf{1}_{|\mathbf{Av}_{n}^{\mathfrak{f}_{\cdot}}[\eta]|\gtrsim N^{\e}n^{-1/2}\sup_{i}\|\mathfrak{f}_{i}\|_{\infty}}\E^{\sigma,\mathbb{L}}\mathrm{e}^{2\kappa|\mathbf{Av}_{n}^{\mathfrak{f}_{\cdot}}[\eta]|^{3/2}}}
\end{align*}
By Lemma \ref{lemma:azuma}, we know that {\small$n^{1/2}\{\sup_{i}\|\mathfrak{f}_{i}\|_{\infty}\}^{-1}|\mathbf{Av}_{n}^{\mathfrak{f}_{\cdot}}[\eta]|$} has sub-Gaussian tails of variance parameter $\mathrm{O}(1)$. So, the first expectation on the RHS of the above display is $\lesssim\exp\{-KN^{2\e}\}$ for some $K\gtrsim1$, and the second expectation, since {\small$\kappa^{-1}=n^{-3/4}\{\sup_{i}\|\mathfrak{f}_{i}\|_{\infty}\}^{3/2}$}, is $\lesssim\exp\{C|\mathsf{z}|^{3/2}\}$ with $\mathsf{z}\sim\mathscr{N}(0,1)$ and $C=\mathrm{O}(1)$. Thus, the second expectation above is $\mathrm{O}(1)$. Using this with the previous four displays then gives 
\begin{align}
\eqref{eq:probanI2b}&\lesssim N^{\frac34}\kappa^{-1}N^{-2}|\mathbb{L}|^{3}+N^{\frac34}\kappa^{-1}\log\left\{1+\exp\{-K'N^{2\e}\}\right\}\lesssim\mathrm{RHS}\eqref{eq:probanI1}, \quad K'\gtrsim1,\label{eq:probanI3}
\end{align}
since, again, {\small$\kappa^{-1}=n^{-3/4}\{\sup_{i}\|\mathfrak{f}_{i}\|_{\infty}\}^{3/2}$} and {\small$\log\{1+\exp\{-K'N^{2\e}\}\}$} is exponentially small in $N$. We are now left to control the RHS of \eqref{eq:probanI2a}. For this, we first write, with explanation given after,
\begin{align*}
\E|\tfrac{1}{\tau}{\textstyle\int_{0}^{\tau}}\mathbf{CutAv}_{n}^{\mathfrak{f}_{\cdot}}[\tau_{\y}\eta^{}_{\s+\r}]\d\r|^{\frac32}&=\E\E^{\mathrm{dyn}}_{\tau_{\y}\eta^{}_{\s}}|\tfrac{1}{\tau}{\textstyle\int_{0}^{\tau}}\mathbf{CutAv}_{n}^{\mathfrak{f}_{\cdot}}[\omega^{}_{\s+\r}]\d\r|^{\frac32}.
\end{align*}
Above, $\omega_{\s+\cdot}$ is a path sampled from the law of the process {\small$\t\mapsto\eta_{\t}$} (as to not confuse it with {\small$\t\mapsto\eta_{\t}$} itself) after conditioning on its initial data to be {\small$\eta_{\s}$}. The {\small$\E^{\mathrm{dyn}}_{\eta}$} means (``dynamical") expectation with respect to this process as a function of its initial condition $\eta$. So, in the above display, all we do is break up $\E$ by first conditioning on the law of the initial data $\eta_{\s}$, then we shift everything by $\y$, which only moves the $\tau_{\y}$ from the time-average into the initial condition {\small$\tau_{\y}\eta^{}_{\s}$} on the RHS (since the law of the dynamics is translation-invariant). The inner expectation on the RHS is a function of only {\small$\tau_{\y}\eta^{}_{\s}$}.

As we mentioned earlier, {\small$\mathbf{CutAv}_{n}^{\mathfrak{f}_{\cdot}}[\eta^{}_{\s+\r}]$} depends only on {\small$\eta^{}_{\s+\r,\x}$} for $\x\in\mathbb{L}$. Now, let $\mathbb{L}[\e]$ be a neighborhood of $\mathbb{L}$ of radius $N^{\e}|\mathbb{L}|$, where $\e>0$ is small but fixed. We now construct a process $\eta^{\mathbb{L}[\e]}$ (see the beginning of Section \ref{subsection:kv}) that is coupled to $\omega^{}$ for times $\s+\r$ with $\r\in[0,\tau]$ such that the probability of {\small$\eta^{\mathbb{L}[\e]}_{\s+\r,\x}\neq\omega^{}_{\s+\r,\x}$} for some $\x\in\mathbb{L}$ is $\lesssim N^{-100}$, and for all $\x\in\mathbb{L}[\e]$, the initial data of this localized process is {\small$\eta^{\mathbb{L}[\e]}_{\s,\x}=\omega^{}_{\s,\x}=\tau_{\y}\eta_{\s,\x}$}.
\begin{enumerate}
\item When spins at neighboring points in $\mathbb{L}[\e]$ swap in the $\omega^{}$ process, they also swap in the $\eta^{\mathbb{L}[\e]}$ process if the speeds to do so are the same. If the speeds are not the same, they differ by $\mathrm{O}(N^{3/2})$ (since the speed $N^{2}$ part of the system has constant-speed swaps). In this case, we couple so that the probability that one of $\omega^{}$ and $\eta^{\mathbb{L}[\e]}$ swaps but the other does not is $\mathrm{O}(N^{-1/2})$. For neighboring points in $\mathbb{T}_{N}$ that are not both contained in $\mathbb{L}[\e]$, we swap in $\omega^{}$ using any independent Poisson clock. We also use an independent Poisson clock for swapping between the rightmost and leftmost points in $\mathbb{L}[\e]$ (recall $\mathbb{L}[\e]$ has periodic boundary conditions).
\item We define a discrepancy at time $u$ between $\omega^{}$ and $\eta^{\mathbb{L}[\e]}$ to be a point $\x\in\mathbb{L}[\e]$ such that {\small$\omega^{}_{u,\x}\neq\eta^{\mathbb{L}[\e]}_{u,\x}$}.
\item Since the initial configurations at time $\s$ of $\omega^{}$ and $\eta^{\mathbb{L}[\e]}$ agree on $\mathbb{L}[\e]$, at time $\s$, there are no discrepancies. Also, the first discrepancy must appear in an $\mathrm{O}(1)$ neighborhood of the rightmost or leftmost point of $\mathbb{L}[\e]$, since the speeds of swapping depend only on $\mathrm{O}(1)$ neighborhoods of configurations of particles. (Indeed, a discrepancy can only appear if the speeds of spin-swaps in $\omega^{}$ and $\eta^{\mathbb{L}[\e]}$ are different, but this can only happen either near the rightmost or leftmost points in $\mathbb{L}[\e]$ or in a $\mathrm{O}(1)$ neighborhood of a discrepancy.) 
\item We now claim that under this coupling, a discrepancy in $\mathbb{L}[\e]$ evolves according to a simple random walk with symmetric speed $\mathrm{O}(N^{2})$ and asymmetric speed $\mathrm{O}(N^{3/2})$. Indeed, if a discrepancy is located at $\x$, a spin-swap between $\x\pm1,\x$ moves the discrepancy to $\x\pm1$, and the speed of spin-swap between $\x\pm1,\x$ is $\frac12N^{2}+\mathrm{O}(N^{3/2})$, the point being that the order $N^{2}$-term does not depend on $\pm$. Also, a discrepancy, as noted earlier, can be created in an $\mathrm{O}(1)$-neighborhood of a discrepancy.  This can be thought of as having a discrepancy at $\x$, creating another random walk particle at $\x$, and then having one of these random walks jump $\mathrm{O}(1)$ distance. The birth of such a discrepancy happens at speed $\mathrm{O}(N^{2})$, because all Poisson clocks here have at most this speed.
\item For a discrepancy to appear in $\mathbb{L}$, it must appear first in a $\mathrm{O}(1)$ neighborhood of the rightmost or leftmost point of $\mathbb{L}[\e]$ and then travel distance $\gtrsim N^{\e}|\mathbb{L}|$. Any given finite-jump-length random walk of symmetric speed $\mathrm{O}(N^{2})$ and asymmetric speed $\mathrm{O}(N^{3/2})$ will travel maximal distance $\lesssim N^{\rho}N\tau^{1/2}+N^{\rho}N^{3/2}\tau$ by time $\tau$ with probability at least $1-\exp\{-KN^{\rho}\}$ by standard bounds (e.g. Azuma's inequality), where $K>0$ is independent of $N$. Since $\tau\lesssim N^{-1}$ and $N\tau^{1/2}\lesssim|\mathbb{L}|$, this maximal distance is $\lesssim N^{\rho}|\mathbb{L}|$. So, if $\rho=\e/2$, then the probability of a discrepancy random walk propagating into $\mathbb{L}$ is $\lesssim\exp\{-KN^{\rho}\}$. We now note that any discrepancy can only be created by the ringing of a Poisson clock in either $\omega^{}$ or $\eta^{\mathbb{L}[\e]}$, of which there are $\mathrm{O}(N)$ that have speed $\mathrm{O}(N^{2})$. Therefore, by standard Poisson tail bounds, the number of ringings, and thus the number of total discrepancy random walks, before time $\tau\lesssim N^{-1}$ is $\lesssim N^{100}$, for example, with exponentially high probability. So, we union bound $\mathrm{O}(N^{100})$ probabilities that are each $\lesssim\exp\{-KN^{\rho}\}$; thus, the probability of \emph{any} discrepancy appearing in $\mathbb{L}$ by time $\tau$ is exponentially small in $N$. 
\end{enumerate}
Thus, if we let {\small$\E^{\mathrm{dyn},\mathbb{L}[\e]}_{\tau_{\y}\eta^{}_{\s}}$} be expectation with respect to the law of the $\eta^{\mathbb{L}[\e]}$ dynamics, then
\begin{align*}
\E\E^{\mathrm{dyn}}_{\tau_{\y}\eta^{}_{\s}}|\tfrac{1}{\tau}{\textstyle\int_{0}^{\tau}}\mathbf{CutAv}_{n}^{\mathfrak{f}_{\cdot}}[\omega^{}_{\s+\r}]\d\r|^{\frac32}=\E\E^{\mathrm{dyn},\mathbb{L}[\e]}_{\tau_{\y}\eta^{}_{\s}}|\tfrac{1}{\tau}{\textstyle\int_{0}^{\tau}}\mathbf{CutAv}_{n}^{\mathfrak{f}_{\cdot}}[\eta^{\mathbb{L}[\e]}_{\s+\r}]\d\r|^{\frac32}+\mathrm{O}\left(N^{-100}\sup_{i=1,\ldots,n}\|\mathfrak{f}_{i}\|^{\frac32}\right),
\end{align*}
The RHS depends only on {\small$\tau_{\y}\eta^{}_{\s,\x}$} for $\x\in\mathbb{L}[\e]$ by construction. By the two previous displays and other considerations to be explained after, we get
\begin{align}
&\int_{0}^{1}\tfrac{1}{N}\sum_{\y\in\mathbb{T}_{N}}N^{\frac34}\E|\tfrac{1}{\tau}{\textstyle\int_{0}^{\tau}}\mathbf{CutAv}_{n}^{\mathfrak{f}_{\cdot}}[\tau_{\y}\eta^{}_{\s+\r}]\d\r|^{\frac32}\d\s\nonumber\\
&=\int_{0}^{1}\tfrac{1}{N}\sum_{\y\in\mathbb{T}_{N}}N^{\frac34}\E\E^{\mathrm{dyn},\mathbb{L}[\e]}_{\tau_{\y}\eta^{}_{\s}}|\tfrac{1}{\tau}{\textstyle\int_{0}^{\tau}}\mathbf{CutAv}_{n}^{\mathfrak{f}_{\cdot}}[\eta^{\mathbb{L}[\e]}_{\s+\r}]\d\r|^{\frac32}\d\s+\mathrm{O}\left(N^{-50}\sup_{i=1,\ldots,n}\|\mathfrak{f}_{i}\|^{\frac32}\right)\nonumber\\
&\lesssim N^{\frac34}\kappa^{-1}N^{-2}|\mathbb{L}[\e]|^{3}+N^{\frac34}\sup_{\sigma\in[-1,1]}\E^{\sigma,\mathbb{L}[\e]}\E^{\mathrm{dyn},\mathbb{L}[\e]}_{\eta}|\tfrac{1}{\tau}{\textstyle\int_{0}^{\tau}}\mathbf{CutAv}_{n}^{\mathfrak{f}_{\cdot}}[\eta^{\mathbb{L}[\e]}_{\r}]\d\r|^{\frac32},\label{eq:editproban}
\end{align}
where we now choose {\small$\kappa=N^{-3\e/2}n^{3/4}\{\sup_{i=1,\ldots,n}\|\mathfrak{f}_{i}\|_{\infty}\}^{-3/2}$} (the big-Oh term in the first identity thus gets absorbed into the last line above), and where $\eta$ in {\small$\E^{\mathrm{dyn},\mathbb{L}[\e]}_{\eta}$} is sampled according to $\E^{\sigma,\mathbb{L}[\e]}$ above. To see the last line, we first note that in {\small$\E^{\mathrm{dyn},\mathbb{L}[\e]}_{\tau_{\y}\eta^{}_{\s}}|\tfrac{1}{\tau}{\textstyle\int_{0}^{\tau}}\mathbf{CutAv}_{n}^{\mathfrak{f}_{\cdot}}[\eta^{\mathbb{L}[\e]}_{\s+\r}]\d\r|^{3/2}$}, once we condition on the initial configuration to be {\small$\tau_{\y}\eta^{}_{\s}$}, it does not matter if we start {\small$\eta^{\mathbb{L}[\e]}$} at time $\s$ or at time $0$, since the law of the dynamics of {\small$\eta^{\mathbb{L}[\e]}$} is time-homogeneous. The last line then follows immediately by \eqref{eq:localreducII}, because {\small$\E^{\mathrm{dyn},\mathbb{L}[\e]}_{\eta}|\tfrac{1}{\tau}{\textstyle\int_{0}^{\tau}}\mathbf{CutAv}_{n}^{\mathfrak{f}_{\cdot}}[\eta^{\mathbb{L}[\e]}_{\r}]\d\r|^{3/2}\lesssim N^{3\e/2}n^{-3/4}\sup_{i=1,\ldots,n}\|\mathfrak{f}_{i}\|_{\infty}^{3/2}=\kappa^{-1}$} by construction of $\mathbf{CutAv}$ (see the display after \eqref{eq:probanI1}), and this is the only constraint we must satisfy to use \eqref{eq:localreducII}. Now, note that the double expectation in the last line of the above display is expectation with respect to {\small$\eta^{\mathbb{L}[\e]}$} dynamics with initial data distributed according to $\mathbb{P}^{\sigma,\mathbb{L}[\e]}$. Next,
\begin{align*}
&N^{\frac34}\E^{\sigma,\mathbb{L}[\e]}\E^{\mathrm{dyn},\mathbb{L}[\e]}_{\eta}|\tfrac{1}{\tau}{\textstyle\int_{0}^{\tau}}\mathbf{CutAv}_{n}^{\mathfrak{f}_{\cdot}}[\eta^{\mathbb{L}[\e]}_{\r}]\d\r|^{\frac32}\\
&\lesssim N^{\frac34}\E^{\sigma,\mathbb{L}[\e]}\E^{\mathrm{dyn},\mathbb{L}[\e]}_{\eta}|\tfrac{1}{\tau}{\textstyle\int_{0}^{\tau}}\mathbf{Av}_{n}^{\mathfrak{f}_{\cdot}}[\eta^{\mathbb{L}[\e]}_{\r}]\d\r|^{\frac32}\\
&+N^{\frac34}\E^{\sigma,\mathbb{L}[\e]}\E^{\mathrm{dyn},\mathbb{L}[\e]}_{\eta}|\tfrac{1}{\tau}{\textstyle\int_{0}^{\tau}}\{\mathbf{Av}_{n}^{\mathfrak{f}_{\cdot}}[\eta^{\mathbb{L}[\e]}_{\r}]-\mathbf{CutAv}_{n}^{\mathfrak{f}_{\cdot}}[\eta^{\mathbb{L}[\e]}_{\r}]\}\d\r|^{\frac32}.
\end{align*}
We have the following bound due to H\"{o}lder with respect to the double expectation:
\begin{align*}
&\E^{\sigma,\mathbb{L}[\e]}\E^{\mathrm{dyn},\mathbb{L}[\e]}_{\eta}|\tfrac{1}{\tau}{\textstyle\int_{0}^{\tau}}\{\mathbf{Av}_{n}^{\mathfrak{f}_{\cdot}}[\eta^{\mathbb{L}[\e]}_{\r}]-\mathbf{CutAv}_{n}^{\mathfrak{f}_{\cdot}}[\eta^{\mathbb{L}[\e]}_{\r}]\}\d\r|^{\frac32}\\
&\lesssim\left\{\E^{\sigma,\mathbb{L}[\e]}\E^{\mathrm{dyn},\mathbb{L}[\e]}_{\eta}|\tfrac{1}{\tau}{\textstyle\int_{0}^{\tau}}\{\mathbf{Av}_{n}^{\mathfrak{f}_{\cdot}}[\eta^{\mathbb{L}[\e]}_{\r}]-\mathbf{CutAv}_{n}^{\mathfrak{f}_{\cdot}}[\eta^{\mathbb{L}[\e]}_{\r}]\}\d\r|^{2}\right\}^{\frac34}.
\end{align*}
Next, we claim
\begin{align*}
&\E^{\sigma,\mathbb{L}[\e]}\E^{\mathrm{dyn},\mathbb{L}[\e]}_{\eta}|\tfrac{1}{\tau}{\textstyle\int_{0}^{\tau}}\{\mathbf{Av}_{n}^{\mathfrak{f}_{\cdot}}[\eta^{\mathbb{L}[\e]}_{\r}]-\mathbf{CutAv}_{n}^{\mathfrak{f}_{\cdot}}[\eta^{\mathbb{L}[\e]}_{\r}]\}\d\r|^{2}\\
&\lesssim\E^{\sigma,\mathbb{L}[\e]}\E^{\mathrm{dyn},\mathbb{L}[\e]}_{\eta}\tfrac{1}{\tau}{\textstyle\int_{0}^{\tau}}|\mathbf{Av}_{n}^{\mathfrak{f}_{\cdot}}[\eta^{\mathbb{L}[\e]}_{\r}]-\mathbf{CutAv}_{n}^{\mathfrak{f}_{\cdot}}[\eta^{\mathbb{L}[\e]}_{\r}]|^{2}\d\r\\
&=\tfrac{1}{\tau}{\textstyle\int_{0}^{\tau}}\E^{\sigma,\mathbb{L}[\e]}\E^{\mathrm{dyn},\mathbb{L}[\e]}_{\eta}|\mathbf{Av}_{n}^{\mathfrak{f}_{\cdot}}[\eta^{\mathbb{L}[\e]}_{\r}]-\mathbf{CutAv}_{n}^{\mathfrak{f}_{\cdot}}[\eta^{\mathbb{L}[\e]}_{\r}]|^{2}\d\r\\
&\lesssim \tfrac{1}{\tau}\int_{0}^{\tau}\E^{\sigma,\mathbb{L}[\e]}|\mathbf{Av}_{n}^{\mathfrak{f}_{\cdot}}-\mathbf{CutAv}_{n}^{\mathfrak{f}_{\cdot}}|^{2}\d r\lesssim N^{-100}\sup_{i=1,\ldots,n}\|\mathfrak{f}_{i}\|_{\infty}^{\frac32}.
\end{align*}
The first line follows by Cauchy-Schwarz in the $\d\r$-integration. The second follows by Fubini. To obtain the last line, we change measure via Lemma \ref{lemma:lpprod} to get the first bound. Then we use the a priori bound {\small$|\mathbf{Av}_{n}^{\mathfrak{f}_{\cdot}}|\lesssim\sup_{i}\|\mathfrak{f}_{i}\|_{\infty}$} and that {\small$\mathbf{Av}_{n}^{\mathfrak{f}_{\cdot}}-\mathbf{CutAv}_{n}^{\mathfrak{f}_{\cdot}}\neq0$} with exponentially small probability in $N$ with respect to $\mathbb{P}^{\sigma,\mathbb{L}[\e]}$. (We used this earlier in the proof of \eqref{eq:probanI3}.) On the other hand, we have the following for any $\delta>0$ by H\"{o}lder and \eqref{eq:kv}:
\begin{align*}
&N^{\frac34}\E^{\sigma,\mathbb{L}[\e]}\E^{\mathrm{dyn},\mathbb{L}[\e]}_{\eta}|\tfrac{1}{\tau}{\textstyle\int_{0}^{\tau}}\mathbf{Av}_{n}^{\mathfrak{f}_{\cdot}}[\eta^{\mathbb{L}[\e]}_{\r}]\d\r|^{\frac32}\\
&\lesssim N^{\frac34}\left\{\E^{\sigma,\mathbb{L}[\e]}\E^{\mathrm{dyn},\mathbb{L}[\e]}_{\eta}|\tfrac{1}{\tau}{\textstyle\int_{0}^{\tau}}\mathbf{Av}_{n}^{\mathfrak{f}_{\cdot}}[\eta^{\mathbb{L}[\e]}_{\r}]\d\r|^{2}\right\}^{\frac34}\\
&\lesssim_{\delta} N^{\delta}N^{\frac34}N^{-\frac32}\tau^{-\frac34}n^{-\frac34}\sup_{i=1,\ldots,n}|\mathbb{I}_{i}|^{\frac32}\sup_{i=1,\ldots,n}\|\mathfrak{f}_{i}\|_{\infty}^{\frac32}.
\end{align*}
(We can always rescale the $\mathfrak{f}_{i}$ to have uniformly bounded $\|\|_{\infty}$-norm to use \eqref{eq:kv}; the last factor in the previous display is the cost in doing so.) So, the last five displays give
\begin{align}
\mathrm{RHS}\eqref{eq:probanI2a}&\lesssim_{\delta}N^{\frac34}\kappa^{-1}N^{-2}|\mathbb{L}[\e]|^{3}+N^{-100}\sup_{i=1,\ldots,n}\|\mathfrak{f}_{i}\|^{\frac32}_{\infty}+N^{\delta}N^{\frac34}N^{-\frac32}\tau^{-\frac34}n^{-\frac34}\sup_{i=1,\ldots,n}|\mathbb{I}_{i}|^{\frac32}\sup_{i=1,\ldots,n}\|\mathfrak{f}_{i}\|_{\infty}^{\frac32}.\nonumber
\end{align}
Because $|\mathbb{L}[\e]|\lesssim N^{\e}|\mathbb{L}|$ and {\small$\kappa^{-1}=N^{3\e/2}n^{-3/4}\sup_{i}\|\mathfrak{f}_{i}\|_{\infty}^{3/2}$} for $\e>0$ small, the first term on the RHS of the previous display is $\lesssim\mathrm{RHS}\eqref{eq:probanIb}$. Combining this with \eqref{eq:probanI2a}-\eqref{eq:probanI2b} with \eqref{eq:probanI3} and the above display now give \eqref{eq:probanI1}. As mentioned right before \eqref{eq:probanI1}, this completes the proof.
\end{proof}
We now give another estimate which uses Proposition \ref{prop:kv2} instead of Proposition \ref{prop:kv}.
\begin{lemma}\label{lemma:probanfluc}
\fsp Fix a sub-interval $\mathbb{L}$. Suppose $\mathfrak{f}_{1},\mathfrak{f}_{2}:\{\pm1\}^{\mathbb{T}_{N}}\to\R$ are functions such that:
\begin{enumerate}
\item For any $\eta\in\{\pm1\}^{\mathbb{T}_{N}}$, the quantity $\mathfrak{f}_{i}[\eta]$ depends only on $\eta_{\x}$ for $\x\in\mathbb{I}_{i}\subseteq\mathbb{L}$, and $\mathbb{I}_{1}\cap\mathbb{I}_{2}=\emptyset$.
\item We have $\E^{\sigma,\mathbb{L}}\mathfrak{f}_{1}=0$ for all $\sigma\in[-1,1]$.
\end{enumerate}
Fix $0<\tau\lesssim N^{-4/3}$ and {\small$|\mathbb{L}|\lesssim N^{1/3+\e_{\mathrm{reg}}}$}. Assume $N^{2}\tau\lesssim|\mathbb{L}|^{2}$. For any $\delta>0$, we have
\begin{align}
&\E\left[\sup_{\t\in[0,1]}\sup_{\x\in\mathbb{T}_{N}}\left|{\int_{0}^{\t}}\sum_{\y\in\mathbb{T}_{N}}\mathbf{H}^{N}_{\s,\t,\x,\y}\times N^{\frac12}\left\{\tfrac{1}{\tau}\int_{0}^{\tau}\mathfrak{f}_{1}[\tau_{\y}\eta^{}_{\s+\r}]\mathfrak{f}_{2}[\tau_{\y}\eta^{}_{\s+\r}]\d\r\right\}\mathbf{Y}^{N}_{\s,\y}\d\s\right|\right]\label{eq:probanflucIa}\\
&\lesssim_{\delta} N^{\e_{\mathrm{ap}}}\left(N^{\delta}N^{-\frac54}|\mathbb{L}|^{3}\sup_{i=1,2}\|\mathfrak{f}_{i}\|_{\infty}^{\frac32}+N^{\delta}N^{-\frac34}\tau^{-\frac34}|\mathbb{I}_{1}|^{\frac32}\sup_{i=1,2}\|\mathfrak{f}_{i}\|_{\infty}^{\frac32}\right)^{2/3},\label{eq:probanflucIb}
\end{align}
where $\|\|_{\infty}$ means sup-norm over $\eta\in\{\pm1\}^{\mathbb{T}_{N}}$.
\end{lemma}
\begin{proof}
The first display in the proof of Lemma \ref{lemma:proban} still holds if we replace the average of $\mathfrak{f}_{1},\ldots,\mathfrak{f}_{n}$ therein by $\mathfrak{f}_{1},\mathfrak{f}_{2}$. In particular, similar to \eqref{eq:probanI1}, it suffices to prove the following for any $\delta>0$:
\begin{align}
&\int_{0}^{1}\tfrac{1}{N}\sum_{\y\in\mathbb{T}_{N}}N^{\frac34}\E|\tfrac{1}{\tau}{\textstyle\int_{0}^{\tau}}\mathfrak{f}_{1}[\tau_{\y}\eta^{}_{\s+\r}]\mathfrak{f}_{2}[\tau_{\y}\eta^{}_{\s+\r}]\d\r|^{\frac32}\d\s\nonumber\\
&\lesssim_{\delta}N^{\delta}N^{-\frac54}|\mathbb{L}|^{3}\prod_{i=1,2}\|\mathfrak{f}_{i}\|_{\infty}^{\frac32}+N^{\delta}N^{-\frac34}\tau^{-\frac34}|\mathbb{I}_{1}|^{\frac32}\prod_{i=1,2}\|\mathfrak{f}_{i}\|_{\infty}^{\frac32}. \label{eq:probanflucI1}
\end{align}
We follow exactly the proof of Lemma \ref{lemma:proban} with $n=1$ and $\mathfrak{f}_{1}$ there replaced by $\mathfrak{f}_{1}\mathfrak{f}_{2}$. In the language of the proof of Lemma \ref{lemma:proban}, we have {\small$\mathbf{CutAv}^{\mathfrak{f}_{\cdot}}_{n}=\mathbf{Av}^{\mathfrak{f}_{\cdot}}_{n}=\mathfrak{f}_{1}\mathfrak{f}_{2}$} since $n=1$. In particular, \eqref{eq:editproban} translates into
\begin{align*}
\mathrm{LHS}\eqref{eq:probanflucI1}&\lesssim N^{\frac34}\kappa^{-1}N^{-2}|\mathbb{L}[\e]|^{3}+N^{\frac34}\sup_{\sigma\in[-1,1]}\E^{\sigma,\mathbb{L}[\e]}\E^{\mathrm{dyn},\mathbb{L}[\e]}_{\eta}|\tfrac{1}{\tau}{\textstyle\int_{0}^{\tau}}\mathfrak{f}_{1}[\eta^{\mathbb{L}[\e]}_{\r}]\mathfrak{f}_{2}[\eta^{\mathbb{L}[\e]}_{\r}]\d\r|^{\frac32},
\end{align*}
where {\small$\kappa^{-1}\lesssim N^{3\e/2}\sup_{i}\|\mathfrak{f}_{1}\mathfrak{f}_{2}\|_{\infty}$} for small $\e>0$ (see after \eqref{eq:editproban} for this choice of $\kappa$). The double expectation in the previous display is just expectation with respect to the process {\small$\eta^{\mathbb{L}[\e]}$} with initial data distributed according to $\E^{\sigma,\mathbb{L}[\e]}$. We use Proposition \ref{prop:kv2} to bound it by {\small$\lesssim_{\delta}N^{\delta}N^{-3/2}\tau^{-3/4}|\mathbb{I}_{1}|^{3/2}\|\mathfrak{f}_{1}\|_{\infty}\|\mathfrak{f}_{2}\|_{\infty}$} for any $\delta>0$. We then use $|\mathbb{L}[\e]|^{3}\lesssim N^{\e}|\mathbb{L}|$. Ultimately, this turns the previous display into \eqref{eq:probanflucI1}, so the proof is complete.
\end{proof}
We now give an estimate in which we do not time-average, only space-average. The estimate will not be as good, but this estimate will nonetheless be important for our proof of Theorem \ref{theorem:bgp}.
\begin{lemma}\label{lemma:probanspace}
\fsp Fix a sub-interval $\mathbb{L}\subseteq\mathbb{T}_{N}$. Suppose $\mathfrak{f}_{i}:\{\pm1\}^{\mathbb{T}_{N}}\to\R$ are functions such that:
\begin{enumerate}
\item For any $\eta\in\{\pm1\}^{\mathbb{T}_{N}}$, the quantity $\mathfrak{f}_{i}[\eta]$ depends only on $\eta_{\x}$ for $\x\in\mathbb{I}_{i}\subseteq\mathbb{L}$, where $\mathbb{I}_{i}$ are pairwise disjoint.
\item For every $i$ and $\sigma\in[-1,1]$, we have $\E^{\sigma,\mathbb{L}}\mathfrak{f}_{i}=0$ (recall $\E^{\sigma,\mathbb{L}}$ from Definition \ref{definition:can}).
\end{enumerate}
We have the estimate below for any $\e>0$:
\begin{align}
&\E\left[\sup_{\t\in[0,\t]}\sup_{\x\in\mathbb{T}_{N}}\int_{0}^{\t}\sum_{\y\in\mathbb{T}_{N}}\mathbf{H}^{N}_{\s,\t,\x,\y}N^{\frac16}\left|\tfrac{1}{n}\sum_{i=1}^{n}\mathfrak{f}_{i}[\tau_{\y}\eta_{\s}]\right|\d\s\right]\nonumber\\
&\lesssim_{\e} \left(N^{-\frac74}|\mathbb{L}|^{3}n^{-\frac34}\sup_{i=1,\ldots,n}\|\mathfrak{f}_{i}\|_{\infty}^{\frac32}+N^{\frac14+\e}n^{-\frac34}\sup_{i=1,\ldots,n}\|\mathfrak{f}_{i}\|_{\infty}^{\frac32}\right)^{2/3}.\label{eq:probanspace}
\end{align}
\end{lemma}
\begin{proof}
Note that this is just a special case of Lemma \ref{lemma:proban} but with $\tau=0$. So, the same argument holds. (Of course, we do not use all of the argument, since $\tau=0$ gives an infinite term in \eqref{eq:probanIb}.) In particular, we can follow the first display in the proof of Lemma \ref{lemma:proban} (i.e. the H\"{o}lder inequality argument) to reduce the proof of \eqref{eq:probanspace} to the proof the following estimate (similar to \eqref{eq:probanI1}):
\begin{align}
&\int_{0}^{1}\tfrac{1}{N}\sum_{\y\in\mathbb{T}_{N}}N^{\frac14}\E|\tfrac{1}{n}\sum_{i=1}^{n}\mathfrak{f}_{i}[\tau_{\y}\eta^{}_{\s}]|^{\frac32}\d\s\nonumber\\
&\lesssim N^{-\frac74}|\mathbb{L}|^{3}n^{-\frac34}\sup_{i=1,\ldots,n}\|\mathfrak{f}_{i}\|_{\infty}^{\frac32}+N^{\frac14+\e}n^{-\frac34}\sup_{i=1,\ldots,n}\|\mathfrak{f}_{i}\|_{\infty}^{\frac32}.\label{eq:probanspace1}
\end{align}
(We have $N^{1/4}$ on the LHS of \eqref{eq:probanspace1} and not $N^{3/4}$ as in \eqref{eq:probanI1} because this factor comes from the $3/2$-th power of the $N^{1/6}$ factor on the LHS of \eqref{eq:probanspace}, whereas $N^{3/4}$ in \eqref{eq:probanI1} is the $3/2$-th power of $N^{1/2}$ in \eqref{eq:probanIa}.) Now use \eqref{eq:probanI2a}-\eqref{eq:probanI2b} and the first inequality in \eqref{eq:probanI3}. (Again, all of the {\small$N^{3/4}$} factors there are replaced by {\small$N^{1/4}$} for the reason that we just explained.) This gives
\begin{align}
\mathrm{LHS}\eqref{eq:probanspace1}&\lesssim\int_{0}^{1}\tfrac{1}{N}\sum_{\y\in\mathbb{T}_{N}}N^{\frac14}\E|\mathbf{CutAv}_{n}^{\mathfrak{f}_{\cdot}}[\tau_{\y}\eta^{}_{\s}]|^{\frac32}\d\s\nonumber\\
&+N^{\frac14}\kappa^{-1}N^{-2}|\mathbb{L}|^{3}+N^{\frac14}\kappa^{-1}\log\{1+\exp\{-KN^{2\e}\}\},\label{eq:probanspace2}
\end{align}
where $\kappa^{-1}=n^{-3/4}\{\sup_{i}\|\mathfrak{f}_{i}\|_{\infty}\}^{3/2}$ (see the paragraph before \eqref{eq:probanI3}), where $K\gtrsim1$, and, for any $\e>0$,
\begin{align*}
\mathbf{CutAv}_{n}^{\mathfrak{f}_{\cdot}}[\eta]:=\mathbf{Av}_{n}^{\mathfrak{f}_{\cdot}}[\eta]\mathbf{1}[|\mathbf{Av}_{n}^{\mathfrak{f}_{\cdot}}[\eta]|\lesssim N^{\e}n^{-1/2}\sup_{i=1,\ldots,n}\|\mathfrak{f}_{i}\|_{\infty}].
\end{align*}
The last term in \eqref{eq:probanspace2} is exponentially small in $N$ (times $\sup_{i}\|\mathfrak{f}_{i}\|_{\infty}^{3/2}$) since $K\gtrsim1$. Moreover, the first term on the RHS of \eqref{eq:probanspace2} is {\small$\lesssim N^{\e}N^{1/4}n^{-3/4}\sup_{i}\|\mathfrak{f}_{i}\|_{\infty}^{3/2}$} by definition of {\small$\mathbf{CutAv}^{\mathfrak{f}}_{n}[\eta]$}. In particular, \eqref{eq:probanspace1} follows from the previous two displays, so we are done.
\end{proof}
%
%
%
\section{Proof of Theorem \ref{theorem:bgp} and Proposition \ref{prop:hl}}\label{section:bgphlproof}
We spend almost all of this section on Theorem \ref{theorem:bgp}; the end of this section explains how the same argument proves Proposition \ref{prop:hl}. (Proposition \ref{prop:hl} is actually much easier, but the same argument suffices nonetheless.) 

We first need some notation. Fix an integer $\mathfrak{l}\geq1$. For any $\mathfrak{f}:\{\pm1\}^{\mathbb{T}_{N}}\to\R$, we define the following function of $\eta\in\{\pm1\}^{\mathbb{T}_{N}}$ (which we explain afterwards):
\begin{align}
\mathsf{E}^{\mathrm{can},\mathfrak{l}}[\eta;\mathfrak{f}]:=\E^{0}\left\{\mathfrak{f}[\wt{\eta}]\middle|\tfrac{1}{\mathfrak{l}}\sum_{\k=0}^{\mathfrak{l}-1}\wt{\eta}_{-\k}=\tfrac{1}{\mathfrak{l}}\sum_{\k=0}^{\mathfrak{l}-1}\eta_{-\k}\right\}.\label{eq:canexp0}
\end{align}
Above, $\eta\in\{\pm1\}^{\mathbb{T}_{N}}$ is fixed. The expectation $\E^{0}$ is with respect to $\wt{\eta}\sim\mathbb{P}^{0}$. So, we are taking expectation of $\mathfrak{f}[\wt{\eta}]$ with respect to $\wt{\eta}\sim\mathbb{P}^{0}$, but we are conditioning on the density of $\wt{\eta}$ over a block of length $\mathfrak{l}$ to the left of $0$ to be the average of $\eta$ over the same block. In particular, \eqref{eq:canexp0} is a function of $\eta$, and it depends only on $\eta_{\x}$ for $\x$ in the block $\{-\mathfrak{l}+1,\ldots,0\}$ of length $\mathfrak{l}$. (The orientation to the left is chosen to agree with the ``support" of $\overline{\mathfrak{q}}$, i.e. $\overline{\mathfrak{q}}[\eta]$ depends only on {\small$\eta_{\x}$} for $\x$ to the left of $0$; see \eqref{eq:overlineqterm}.) 

The following estimate on $\mathsf{E}^{\mathrm{can},\mathfrak{l}}$ is key for this section. Before we state it, first recall the canonical measure expectations $\E^{\sigma,\mathbb{I}}$ from Definition \ref{definition:can}.
\begin{lemma}\label{lemma:canexp}
\fsp Now, suppose that {\small$\partial_{\sigma}^{\k}\E^{\sigma}\mathfrak{f}|_{\sigma=0}=0$} for all $0\leq\k\leq\mathfrak{k}$ and suppose that $\mathfrak{k}\leq2$. Assume also that there exists $\mathbb{I}$ so that $\mathfrak{f}[\eta]$ depends only on $\eta_{\x}$ for $\x\in\mathbb{I}$, and $|\mathbb{I}|\lesssim1$, and that $\mathbb{I}$ is to the left of $0$. For any $|\mathbb{I}|<\mathfrak{l}\leq\mathfrak{l}_{\mathrm{reg}}$ (see \eqref{eq:tregX}) and $\y\in\mathbb{T}_{N}$ and $\s\leq\mathfrak{t}_{\mathrm{stop}}$ (see \eqref{eq:tst}), we have
\begin{align}
|\mathsf{E}^{\mathrm{can},\mathfrak{l}}[\tau_{\y}\eta_{\s};\mathfrak{f}]|\lesssim N^{3\e_{\mathrm{ap}}}|\mathfrak{l}|^{-\frac{\mathfrak{k}+1}{2}}\|\mathfrak{f}\|_{\infty}.\label{eq:canexp}
\end{align}
Above, $\|\|_{\infty}$ denotes sup-norm for functions $\{\pm1\}^{\mathbb{T}_{N}}\to\R$.
\end{lemma}
Before we give the proof, we note that $\mathfrak{f}=\overline{\mathfrak{q}}$ satisfies the conditions needed for \eqref{eq:canexp} with $\mathfrak{k}=2$. Indeed, by inspection of \eqref{eq:overlineqterm}, we know {\small$\partial_{\sigma}^{\k}\E^{\sigma}\overline{\mathfrak{q}}|_{\sigma=0}=0$} for $\k=0,1$. We also know it for $\k=2$ by \eqref{eq:assumequad} and the fact that {\small$\partial_{\sigma}^{2}\E^{\sigma}\mathfrak{f}|_{\sigma=0}=0$} for any constant function $\mathfrak{f}$ and any linear function $\mathfrak{f}$ of $\eta_{\x}$-terms. Similarly, the constraint is true for $\mathfrak{f}=\mathfrak{s},\mathfrak{g}$ (see \eqref{eq:sterm} and \eqref{eq:gterm}) for $\mathrm{k}=0$ by construction. Lastly, the constraint $\mathfrak{l}>|\mathbb{I}|$ is to make sure that the input function $\mathfrak{f}$ does not depend on {\small$\eta_{\x}$} for any $\x\not\in\mathbb{I}$; in words, we want to condition the $\eta$-density on a block containing the ``support" of $\mathfrak{f}$.
\begin{proof}
We use Proposition 8 of \cite{GJ15} to get the following, which essentially says that conditioning the $\eta$-density on a neighborhood containing $\mathbb{I}$ has asymptotically little effect as the length $\mathfrak{l}$ of the neighborhood diverges. (Intuitively, the ``interior" of a Brownian bridge looks like a Brownian motion without conditioning.) Precisely, we get the following (recall that $\E^{\sigma}$ is product Bernoulli expectation on $\{\pm1\}^{\mathbb{T}_{N}}$ such that {\small$\E^{\sigma}\eta_{\x}=\sigma$} for all $\x$):
\begin{align*}
\mathsf{E}^{\mathrm{can},\mathfrak{l}}[\tau_{\y}\eta_{\s};\mathfrak{f}]&=\E^{\sigma[\tau_{\y}\eta_{\s};\mathfrak{l}]}\mathfrak{f}+\tfrac{1}{2\mathfrak{l}}\chi(\sigma[\tau_{\y}\eta_{\s};\mathfrak{l}])\partial_{\sigma}^{2}\E^{\sigma}\mathfrak{f}|_{\sigma=\sigma[\tau_{\y}\eta_{\s};\mathfrak{l}]}+\mathrm{O}(|\mathfrak{l}|^{-2}\|\mathfrak{f}\|_{\infty}),
\end{align*}
where $\sigma[\tau_{\y}\eta_{\s};\mathfrak{l}]:=\mathfrak{l}^{-1}(\eta_{\s,\y-\mathfrak{l}+1}+\ldots\eta_{\s,\y})$ is just convenient shorthand, and $\chi(\sigma):=\sigma(1-\sigma)$ is just some polynomial of bounded degree. Now, observe that $\E^{\sigma}\mathfrak{f}$ is a polynomial in $\sigma$ of degree $\mathrm{O}(1)$ since $\mathfrak{f}$ is local. Thus, it is smooth with uniformly bounded derivatives. By Taylor expanding the second-derivative on the RHS, 
\begin{align*}
\tfrac{1}{2\mathfrak{l}}\chi(\sigma[\tau_{\y}\eta_{\s};\ell])\partial_{\sigma}^{2}\E^{\sigma}\mathfrak{f}|_{\sigma=\sigma[\tau_{\y}\eta_{\s}]}&=\tfrac{1}{2\mathfrak{l}}\chi(\sigma[\tau_{\y}\eta_{\s};\mathfrak{l}])\partial_{\sigma}^{2}\E^{\sigma}\mathfrak{f}|_{\sigma=0}+\mathrm{O}(|\mathfrak{l}|^{-1}|\sigma[\tau_{\y}\eta_{\s};\mathfrak{l}]|\|\mathfrak{f}\|_{\infty}).
\end{align*}
If $\mathfrak{k}=0,1$, we bound the first term on the RHS by $\mathfrak{l}^{-1}\|\mathfrak{f}\|_{\infty}\lesssim\mathfrak{l}^{-(\mathfrak{k}+1)/2}\|\mathfrak{f}\|_{\infty}$. If $\mathfrak{k}=2$, then the first term on the RHS of the previous display is $0$ by assumption. Thus, the bound $\mathfrak{l}^{-1}\|\mathfrak{f}\|_{\infty}\lesssim\mathfrak{l}^{-(\mathfrak{k}+1)/2}\|\mathfrak{f}\|_{\infty}$ is always true for all $\mathfrak{k}=0,1,2$. On the other hand, by the formulas \eqref{eq:hf} and \eqref{eq:gartner}, we have {\small$|\sigma[\tau_{\y}\eta_{\s};\mathfrak{l}]|=N^{1/2}\mathfrak{l}^{-1}|\mathbf{h}^{N}_{\s,\y}-\mathbf{h}^{N}_{\s,\y-\mathfrak{l}}|\lesssim N^{1/2}\mathfrak{l}^{-1}|\log\mathbf{Z}^{N}_{\s,\y}-\log\mathbf{Z}^{N}_{\s,\y-\mathfrak{l}}|$}. Because {\small$\s\leq\mathfrak{t}_{\mathrm{stop}}$}, we know {\small$\mathbf{Z}^{N}_{\s,\cdot}\gtrsim N^{-\e_{\mathrm{ap}}}$} and {\small$|\mathbf{Z}^{N}_{\s,\y}-\mathbf{Z}^{N}_{\s,\y-\mathfrak{l}}|\lesssim N^{2\e_{\mathrm{ap}}}N^{-1/2}|\mathfrak{l}|^{1/2}$} deterministically. By these two bounds and elementary calculus, we deduce {\small$|\sigma[\tau_{\y}\eta_{\s};\mathfrak{l}]|\lesssim N^{1/2}\mathfrak{l}^{-1}|\log\mathbf{Z}^{N}_{\s,\y}-\log\mathbf{Z}^{N}_{\s,\y-\mathfrak{l}}|\lesssim N^{\e_{\mathrm{ap}}}N^{2\e_{\mathrm{ap}}}\mathfrak{l}^{-1/2}$}. So, the last term in the above display is {\small$\lesssim N^{3\e_{\mathrm{ap}}}\mathfrak{l}^{-3/2}\|\mathfrak{f}\|_{\infty}\lesssim N^{3\e_{\mathrm{ap}}}\mathfrak{l}^{-(\mathfrak{k}+1)/2}\|\mathfrak{f}\|_{\infty}$} since $\mathfrak{k}\leq2$. Ultimately, by the previous bound and this paragraph, we deduce
\begin{align*}
\tfrac{1}{2\mathfrak{l}}\chi(\sigma[\tau_{\y}\eta_{\s};\ell])\partial_{\sigma}^{2}\E^{\sigma}\mathfrak{f}|_{\sigma=\sigma[\tau_{\y}\eta_{\s}]}&=\mathrm{O}(N^{3\e_{\mathrm{ap}}}\mathfrak{l}^{-\frac{\mathfrak{k}+1}{2}}\|\mathfrak{f}\|_{\infty}).
\end{align*}
Combining this with the first display shows that to finish the proof, it is enough to justify
\begin{align*}
\E^{\sigma[\tau_{\y}\eta_{\s};\mathfrak{l}]}\mathfrak{f}&=\E^{0}\mathfrak{f}+\partial_{\sigma}\E^{\sigma}\mathfrak{f}|_{\sigma=0}\sigma[\tau_{\y}\eta_{\s};\mathfrak{l}]+\tfrac12\partial_{\sigma}^{2}\E^{\sigma}\mathfrak{f}|_{\sigma=0}\sigma[\tau_{\y}\eta_{\s};\mathfrak{l}]^{2}+\mathrm{O}(|\sigma[\tau_{\y}\eta_{\s};\mathfrak{l}]|^{3}\|\mathfrak{f}\|_{\infty})\\
&=\mathrm{O}(N^{3\e_{\mathrm{ap}}}|\mathfrak{l}|^{-\frac{\mathfrak{k}+1}{2}}\|\mathfrak{f}\|_{\infty}).
\end{align*}
The first line is by Taylor expansion. The second line can be readily verified by the assumption {\small$\partial_{\sigma}^{\k}\E^{\sigma}\mathfrak{f}|_{\sigma=0}=0$} for all $0\leq\k\leq\mathfrak{k}$ (for $\mathfrak{k}\leq2$) and the deterministic bound {\small$|\sigma[\tau_{\y}\eta_{\s};\mathfrak{l}]|\lesssim N^{\e_{\mathrm{ap}}}N^{2\e_{\mathrm{ap}}}\mathfrak{l}^{-1/2}$} that we showed in an earlier paragraph in this proof. (I.e., more vanishing derivatives means contribution from higher powers of {\small$|\sigma[\tau_{\y}\eta_{\s};\mathfrak{l}]|\lesssim N^{\e_{\mathrm{ap}}}N^{2\e_{\mathrm{ap}}}\mathfrak{l}^{-1/2}$}.) This completes the proof.
\end{proof}
In view of Lemma \ref{lemma:canexp} and the paragraph after it, to show Theorem \ref{theorem:bgp}, we want to justify replacing $\overline{\mathfrak{q}}[\tau_{\y}\eta_{\s}]$ by {\small$\mathsf{E}^{\mathrm{can},\mathfrak{l}_{\mathrm{reg}}}[\tau_{\y}\eta_{\s};\overline{\mathfrak{q}}]$}, since Lemma \ref{lemma:canexp} implies that {\small$\mathsf{E}^{\mathrm{can},\mathfrak{l}_{\mathrm{reg}}}[\tau_{\y}\eta_{\s};\overline{\mathfrak{q}}]=\mathrm{O}(N^{3\e_{\mathrm{ap}}}\mathfrak{l}_{\mathrm{reg}}^{-3/2})=\mathrm{O}(N^{-1/2-3\e_{\mathrm{reg}}/2+3\e_{\mathrm{ap}}})$}, and because $\e_{\mathrm{reg}}\geq C\e_{\mathrm{ap}}$ for some big constant $C>0$ (see \eqref{eq:tregX}), this beats the $N^{1/2}$ factor in $\Upsilon^{\mathrm{BG}}$ (see \eqref{eq:upsilonbg}). To justify such a replacement, we will follow the general schematic of \cite{CYau,GJ15,GJS15,SX} and do it ``step-by-step" with respect to $\mathfrak{l}$. The following construction makes this systematic.
\begin{definition}\label{definition:multiscale}
\fsp We fix $\e_{\mathrm{step}}>0$ small (it will be ``much smaller" that $\e_{\mathrm{ap}},\e_{\mathrm{reg}}$ from \eqref{eq:tap}, \eqref{eq:tregX}, and \eqref{eq:tregT} but still independent of $N$). We let $\mathrm{K}_{\mathrm{step}}>0$ be the smallest positive integer for which $\mathrm{K}_{\mathrm{step}}\e_{\mathrm{step}}=1/3+\e_{\mathrm{reg}}$ (we must choose $\e_{\mathrm{step}}$ depending on $\e_{\mathrm{reg}}$ to make this possible, but its smallness is unaffected by this). 

For any integer $\k\geq0$, we define
\begin{align}
\wt{\mathsf{R}}_{\k}[\eta;\overline{\mathfrak{q}}]:=\begin{cases}\overline{\mathfrak{q}}[\eta]-\mathsf{E}^{\mathrm{can},N^{\e_{\mathrm{step}}}}[\eta;\overline{\mathfrak{q}}]&\k=0\\\mathsf{E}^{\mathrm{can},N^{\k\e_{\mathrm{step}}}}[\eta;\overline{\mathfrak{q}}]-\mathsf{E}^{\mathrm{can},N^{(\k+1)\e_{\mathrm{step}}}}[\eta;\overline{\mathfrak{q}}]&\k\geq1\end{cases}
\end{align}
(The same definition works for any $\mathfrak{f}:\{\pm1\}^{\mathbb{T}_{N}}\to\R$ in place of $\overline{\mathfrak{q}}$, as long as $\mathfrak{f}[\eta]$ depends only on $\eta_{\x}$ for $\x\in\mathbb{I}$ with $|\mathbb{I}|\lesssim1$ and with $\mathbb{I}$ to the left of $0$.)
\end{definition}
The main ingredient to prove Theorem \ref{theorem:bgp} is the following estimate for $\wt{\mathsf{R}}_{\k}$ terms.
\begin{prop}\label{prop:bgpr}
\fsp Fix any integer $0\leq\k<\mathrm{K}_{\mathrm{step}}$. There exists a universal constant $\beta_{\mathrm{BG}}>0$ such that
\begin{align}
\E\left[\sup_{\t\in[0,1]}\sup_{\x\in\mathbb{T}_{N}}|{\textstyle\int_{0}^{\t}}\mathrm{e}^{[\t-\s]\mathscr{T}_{N}}[N^{\frac12}\wt{\mathsf{R}}_{\k}[\tau_{\cdot}\eta_{\s};\overline{\mathfrak{q}}]\mathbf{Y}^{N}_{\s,\cdot}]_{\x}\d\s|\right]\lesssim N^{-\beta_{\mathrm{BG}}}.\label{eq:bgpr}
\end{align}
\end{prop}
\begin{proof}[Proof of Theorem \ref{theorem:bgp} assuming Proposition \ref{prop:bgpr}]
Recall $\Upsilon^{\mathrm{BG}}$ from \eqref{eq:upsilonbg}; this is what we want to control. Recall $\mathfrak{l}_{\mathrm{reg}}:=N^{1/3+\e_{\mathrm{reg}}}=N^{K_{\mathrm{step}}\e_{\mathrm{step}}}$ from \eqref{eq:tregX}. By the triangle inequality, we have
\begin{align}
\E\left[\sup_{\t\in[0,1]}\sup_{\x\in\mathbb{T}_{N}}|\Upsilon^{\mathrm{BG}}_{\t,\x}|\right]&\lesssim\sum_{\k=0}^{\mathrm{K}_{\mathrm{step}}-1}\E\left[\sup_{\t\in[0,1]}\sup_{\x\in\mathbb{T}_{N}}|{\textstyle\int_{0}^{\t}}\mathrm{e}^{[\t-\s]\mathscr{T}_{N}}[N^{\frac12}\wt{\mathsf{R}}_{\k}[\tau_{\cdot}\eta_{\s};\overline{\mathfrak{q}}]\mathbf{Y}^{N}_{\s,\cdot}]_{\x}\d\s|\right]\\
&+\E\left[\sup_{\t\in[0,1]}\sup_{\x\in\mathbb{T}_{N}}|{\textstyle\int_{0}^{\t}}\mathrm{e}^{[\t-\s]\mathscr{T}_{N}}[N^{\frac12}\mathsf{E}^{\mathrm{can},\mathfrak{l}_{\mathrm{reg}}}[\tau_{\cdot}\eta_{\s};\overline{\mathfrak{q}}]\mathbf{Y}^{N}_{\s,\cdot}]_{\x}\d\s|\right].
\end{align}
Because $\mathrm{K}_{\mathrm{step}}$ depends only on $\e_{\mathrm{step}},\e_{\mathrm{reg}}$, which are fixed and independent of $N$, we know $\mathrm{K}_{\mathrm{step}}\lesssim1$. Thus, Proposition \ref{prop:bgpr} controls the first term on the RHS. For the second term on the RHS, we know $\mathbf{Y}^{N}$ vanishes after time $\mathfrak{t}_{\mathrm{stop}}$ (see \eqref{eq:tst}). And, until $\mathfrak{t}_{\mathrm{stop}}$, we also know $\mathbf{Y}^{N}\lesssim N^{\e_{\mathrm{ap}}}$ uniformly and deterministically, and we also know {\small$|\mathsf{E}^{\mathrm{can},\mathfrak{l}_{\mathrm{reg}}}[\tau_{\cdot}\eta_{\s};\overline{\mathfrak{q}}]|\lesssim N^{3\e_{\mathrm{ap}}}N^{-1/2-\e_{\mathrm{reg}}}$} as discussed before Definition \ref{definition:multiscale}. With this and contractivity of the heat semigroup on $\mathrm{L}^{\infty}(\mathbb{T}_{N})$ (its kernel is a probability measure on $\mathbb{T}_{N}$), we get that the second line above is $\lesssim N^{-\e_{\mathrm{reg}}+4\e_{\mathrm{ap}}}\lesssim N^{-\beta}$ with $\beta>0$ universal, since $\e_{\mathrm{reg}}$ is a very big multiple of $\e_{\mathrm{ap}}$, so we are done.
\end{proof}
Our proof of Proposition \ref{prop:bgpr} has a number of steps. The ``zero-th" (preliminary) step is to replace the function {\small$\eta\mapsto\wt{\mathsf{R}}_{\k}[\eta;\overline{\mathfrak{q}}]$} with something, denoted by $\eta\mapsto\mathsf{R}_{\k}[\eta;\overline{\mathfrak{q}}]$, with the following suitable properties:
\begin{enumerate}
\item If $\mathbb{L}$ is any set such that $\mathsf{R}_{\k}[\eta;\overline{\mathfrak{q}}]$ depends only on $\eta_{\x}$ for $\x\in\mathbb{L}$, then $\mathsf{R}_{\k}[\eta;\overline{\mathfrak{q}}]$ vanishes in expectation with respect to $\E^{\sigma,\mathbb{L}}$ for any $\sigma$. This property is already true for $\wt{\mathsf{R}}_{\k}[\eta;\overline{\mathfrak{q}}]$; see the proof of Lemma 2 in \cite{GJ15}.
\item The function $\eta\mapsto\mathsf{R}_{\k}[\eta;\overline{\mathfrak{q}}]$ is uniformly bounded by $N^{10\e_{\mathrm{ap}}}|\mathfrak{l}|^{-3/2}$ with {\small$\mathfrak{l}=N^{\k\e_{\mathrm{step}}}$}. (The constant $10$ is unimportant.) Let us clarify this. Lemma \ref{lemma:canexp} tells us that if we evaluate {\small$\wt{\mathsf{R}}_{\k}[\eta;\overline{\mathfrak{q}}]$} at $\tau_{\y}\eta_{\s}$ for any $\s\leq\mathfrak{t}_{\mathrm{stop}}$, then we have this bound. We want to modify $\wt{\mathsf{R}}_{\k}$ to make it true \emph{for all} $\eta$. (Of course, because we only care about $\wt{\mathsf{R}}_{\k}[\eta;\overline{\mathfrak{q}}]$ at $\tau_{\y}\eta_{\s}$, we can always modify $\wt{\mathsf{R}}_{\k}$ to ``forget about" other configurations $\eta$. And, if we are careful, we can keep property (1). This step is purely for convenience, and it is already done in \cite{Y23}.)
\end{enumerate}
Next, we replace {\small${\mathsf{R}}_{\k}$} by a space-average, which we then replace by a time-average. Finally, we use Lemma \ref{lemma:proban}. 
\subsection{Replacing $\wt{\mathsf{R}}_{\k}$ by $\mathsf{R}_{\k}$}
The goal of this subsection is the following preliminary replacement.
\begin{lemma}\label{lemma:bgprcut}
\fsp Fix any $0\leq\k<\mathrm{K}_{\mathrm{step}}$. With notation explained afterwards, we have the deterministic estimate
\begin{align*}
\wt{\mathsf{R}}_{\k}[\tau_{\y}\eta_{\s};\overline{\mathfrak{q}}]\mathbf{Y}^{N}_{\s,\y}&=\mathsf{R}_{\k}[\tau_{\y}\eta_{\s};\overline{\mathfrak{q}}]\mathbf{Y}^{N}_{\s,\y}+\mathrm{O}(N^{-100}).
\end{align*}
The function $\eta\mapsto\mathsf{R}_{\k}[\eta;\overline{\mathfrak{q}}]$ satisfies the following constraints:
\begin{itemize}
\item $\E^{\sigma,\mathbb{L}}\mathsf{R}_{\k}[\eta;\overline{\mathfrak{q}}]=0$ for all $\sigma$ and sub-intervals $\mathbb{L}$ such that $\mathsf{R}_{\k}[\eta;\overline{\mathfrak{q}}]$ depends only on {\small$\eta_{\x}$} for $\x\in\mathbb{L}$, and {\small$|\mathsf{R}_{\k}[\eta;\overline{\mathfrak{q}}]|\lesssim N^{10\e_{\mathrm{ap}}}N^{-3\k\e_{\mathrm{step}}/2}$};
\item For any $\eta$, the value $\mathsf{R}_{\k}[\eta;\overline{\mathfrak{q}}]$ depends only on $\eta_{\x}$ for $\x\in\{-N^{(\k+1)\e_{\mathrm{reg}}},\ldots,0\}$.
\end{itemize}
\end{lemma}
\begin{proof}
The argument is exactly the same as the proof of Lemma 10.1 in \cite{Y23} (where $\mathsf{R}_{\delta}$ therein is our $\wt{\mathsf{R}}_{\k}$ and $\mathsf{R}^{\mathrm{cut}}$ therein is our $\mathsf{R}_{\k}$, and where $\delta$ therein is defined so that $N^{\delta}=N^{(\k+1)\e_{\mathrm{step}}}$). The only difference is that in this paper, Lemma \ref{lemma:canexp} gives a bound on canonical measure expectations {\small$|\mathsf{E}^{\mathrm{can},\mathfrak{l}}[\tau_{\s}\eta_{\y};\overline{\mathfrak{q}}]|$} of $\lesssim N^{3\e_{\mathrm{ap}}}|\mathfrak{l}|^{-3/2}$, whereas \cite{Y23} uses a bound of $\lesssim N^{3\e_{\mathrm{ap}}}|\mathfrak{l}|^{-1}$, which is why the bound on $\mathsf{R}^{\mathrm{cut}}_{\delta}$ in Lemma 10.1 of \cite{Y23} is only $N^{10\e_{\mathrm{ap}}}N^{-\delta}$, not $N^{10\e_{\mathrm{ap}}}N^{-3\delta/2}$. We do not reproduce the argument here, because it does not make reference to any dynamics (and in particular, it does not make reference to the choice of driving function $\mathfrak{d}[\cdot]$).
\end{proof}
\subsection{Introducing a space-average}
The goal of this subsection is to prove the following. (Before we state it, let us describe it intuitively. We want to replace {\small$\mathsf{R}_{\k}[\tau_{\cdot}\eta_{\s};\overline{\mathfrak{q}}]$} by an average of $n_{\k}$-many spatial shifts. We want these shifts to have length given by integer multiples of {\small$2N^{(k+1)\e_{\mathrm{step}}}$} for the following reason, where $2$ is just some constant ``for safety". In view of Lemma \ref{lemma:proban}, we want the shifts to depend on $\eta_{\x}$ for $\x$ in mutually disjoint sets. This is guaranteed if we shift by multiplies of {\small$2N^{(k+1)\e_{\mathrm{step}}}$}, since {\small$\mathsf{R}_{\k}[\tau_{\cdot}\eta_{\s};\overline{\mathfrak{q}}]$} depends on $\eta_{\s,\x}$ for $\x$ in an interval of size $N^{(\k+1)\e_{\mathrm{step}}}$ to the left of $\cdot$. Now, to satisfy the space-time size constraints in Lemma \ref{lemma:proban}, we want to ensure the average of $n_{\k}$-many shifts of {\small$\mathsf{R}_{\k}[\tau_{\cdot}\eta_{\s};\overline{\mathfrak{q}}]$} depends only on $\eta_{\s,\x}$ for $\x$ in a set of size of order {\small$\mathfrak{l}_{\mathrm{reg}}=N^{1/3+\e_{\mathrm{reg}}}$}. This determines $n_{\k}$ to be order given by the set size $\mathfrak{l}_{\mathrm{reg}}$ divided by the shift-length $N^{(\k+1)\e_{\mathrm{step}}}$.)  
\begin{lemma}\label{lemma:bgprx}
\fsp Fix any $0\leq\k<\mathrm{K}_{\mathrm{step}}$. With notation explained afterwards, we have the deterministic estimate
\begin{align}
{\int_{0}^{\t}}\mathrm{e}^{[\t-\s]\mathscr{T}_{N}}[N^{\frac12}\mathsf{R}_{\k}[\tau_{\cdot}\eta_{\s};\overline{\mathfrak{q}}]\mathbf{Y}^{N}_{\s,\cdot}]_{\x}\d\s&={\int_{0}^{\t}}\mathrm{e}^{[\t-\s]\mathscr{T}_{N}}\left\{N^{\frac12}\left(\tfrac{1}{n_{\k}}\sum_{j=0}^{n_{\k}-1}\mathsf{R}_{\k}[\tau_{\cdot-2jN^{(\k+1)\e_{\mathrm{step}}}}\eta_{\s};\overline{\mathfrak{q}}]\right)\mathbf{Y}^{N}_{\s,\cdot}\right\}_{\x}\d\s\nonumber\\
&+{\textstyle\int_{0}^{\t}}\mathrm{e}^{[\t-\s]\mathscr{T}_{N}}[N^{\frac12}\mathsf{R}_{\k}[\tau_{\cdot}\eta_{\s};\overline{\mathfrak{q}}]\mathfrak{z}_{\k}[\tau_{\cdot}\eta_{\s}]\mathbf{Y}^{N}_{\s,\cdot}]_{\x}\d\s+\mathrm{O}(N^{-\beta}). \label{eq:bgprx}
\end{align}
%
\begin{itemize}
\item Above, $n_{\k}$ is the unique integer for which $n_{\k}N^{(\k+1)\e_{\mathrm{step}}}=N^{\mathrm{K}_{\mathrm{step}}\e_{\mathrm{step}}}=\mathfrak{l}_{\mathrm{reg}}$ (i.e. {\small$n_{\k}=N^{(K_{\mathrm{step}}-\k-1)\e_{\mathrm{step}}}$}).
\item We have $|\mathfrak{z}_{\k}[\eta]|\lesssim N^{3\e_{\mathrm{ap}}}N^{-1/2}\mathfrak{l}_{\mathrm{reg}}^{1/2}=N^{3\e_{\mathrm{ap}}}N^{-1/3+\e_{\mathrm{reg}}/2}$ deterministically. Second, $\mathfrak{z}_{\k}[\eta]$ depends only on $\eta_{\x}$ for $\x=1,\ldots,\mathrm{O}(\mathfrak{l}_{\mathrm{reg}})$ (whereas $\mathsf{R}_{\k}[\eta;\overline{\mathfrak{q}}]$ depends only on $\eta_{\x}$ for $\x\leq0$).
\item The constant $\beta>0$ is universal.
\end{itemize}
\end{lemma}
\begin{proof}
Replacing $\mathsf{R}_{\k}$ by its space-average means controlling error terms via summation-by-parts. In particular, we first note the identity below obtained by writing $\mathsf{R}_{\k}$ minus its space-shift as a space-gradient:
\begin{align*}
&={\int_{0}^{\t}}\mathrm{e}^{[\t-\s]\mathscr{T}_{N}}\left\{N^{\frac12}\left(\mathsf{R}_{\k}[\tau_{\cdot}\eta_{\s}]-\tfrac{1}{n_{\k}}\sum_{j=0}^{n_{\k}-1}\mathsf{R}_{\k}[\tau_{\cdot-2jN^{(\k+1)\e_{\mathrm{step}}}}\eta_{\s};\overline{\mathfrak{q}}]\right)\mathbf{Y}^{N}_{\s,\cdot}\right\}_{\x}\d\s\\
&=-\tfrac{1}{n_{\k}}\sum_{j=0}^{n_{\k}-1}{\int_{0}^{\t}}\sum_{\y\in\mathbb{T}_{N}}\mathbf{H}^{N}_{\s,\t,\x,\y}N^{\frac12}\left\{\grad^{\mathbf{X}}_{-2jN^{(k+1)\e_{\mathrm{step}}}}\mathsf{R}_{\k}[\tau_{\y}\eta_{\s}]\right\}\mathbf{Y}^{N}_{\s,\y}\d\s.
\end{align*}
An elementary summation-by-parts calculation shows that the adjoint of {\small$\grad^{\mathbf{X}}_{\mathfrak{l}}$} with respect to uniform measure on $\mathbb{T}_{N}$ is equal to {\small$\grad^{\mathbf{X}}_{-\mathfrak{l}}$}. We also note the discrete Leibniz rule {\small$\grad^{\mathbf{X}}_{\mathfrak{l}}(\varphi\psi)_{\x}=\varphi_{\x}\grad^{\mathbf{X}}_{\mathfrak{l}}\psi_{\x}+\psi_{\x+\mathfrak{l}}\grad^{\mathbf{X}}_{\mathfrak{l}}\varphi_{\x}$}. Thus,
\begin{align*}
&{\int_{0}^{\t}}\sum_{\y\in\mathbb{T}_{N}}\mathbf{H}^{N}_{\s,\t,\x,\y}N^{\frac12}\left\{\grad^{\mathbf{X}}_{-2jN^{(k+1)\e_{\mathrm{step}}}}\mathsf{R}_{\k}[\tau_{\y}\eta_{\s}]\right\}\mathbf{Y}^{N}_{\s,\y}\d\s\\
&={\int_{0}^{\t}}\sum_{\y\in\mathbb{T}_{N}}N^{\frac12}\mathsf{R}_{\k}[\tau_{\y}\eta_{\s}]\grad^{\mathbf{X}}_{2jN^{(k+1)\e_{\mathrm{step}}}}\left\{\mathbf{H}^{N}_{\s,\t,\x,\y}\mathbf{Y}^{N}_{\s,\y}\right\}\d\s\\
&={\int_{0}^{\t}}\sum_{\y\in\mathbb{T}_{N}}\mathbf{H}^{N}_{\s,\t,\x,\y}N^{\frac12}\mathsf{R}_{\k}[\tau_{\y}\eta_{\s}]\grad^{\mathbf{X}}_{2jN^{(k+1)\e_{\mathrm{step}}}}\mathbf{Y}^{N}_{\s,\y}\d\s\\
&+{\int_{0}^{\t}}\sum_{\y\in\mathbb{T}_{N}}\grad^{\mathbf{X}}_{2jN^{(k+1)\e_{\mathrm{step}}}}\mathbf{H}^{N}_{\s,\t,\x,\y}N^{\frac12}\mathsf{R}_{\k}[\tau_{\y}\eta_{\s}]\mathbf{Y}^{N}_{\s,\y+2jN^{(k+1)\e_{\mathrm{step}}}}\d\s,
\end{align*}
where the gradient acts either on $\x$ or $\y$ in the heat kernel since the heat kernel depends only on the difference $\x-\y$ (its generator is space-homogeneous). We first control the last line above. By gradient estimates for the heat kernel (see Proposition \ref{prop:hk}), the gradient {\small$\grad^{\mathbf{X}}_{\mathfrak{l}}$} of $\exp\{[\t-\s]\mathscr{T}_{N}\}$ has $\mathrm{L}^{\infty}(\mathbb{T}_{N})\to\mathrm{L}^{\infty}(\mathbb{T}_{N})$-norm bounded by {\small$\lesssim N^{-1}\mathfrak{l}|\t-\s|^{-1/2}$}. Using this with $|\mathbf{Y}^{N}|\lesssim N^{\e_{\mathrm{ap}}}$ (see \eqref{eq:tst}) gives 
\begin{align*}
&{\int_{0}^{\t}}\sum_{\y\in\mathbb{T}_{N}}\grad^{\mathbf{X}}_{2jN^{(k+1)\e_{\mathrm{step}}}}\mathbf{H}^{N}_{\s,\t,\x,\y}N^{\frac12}\mathsf{R}_{\k}[\tau_{\y}\eta_{\s}]\mathbf{Y}^{N}_{\s,\y+2jN^{(k+1)\e_{\mathrm{step}}}}\d\s\\
&\lesssim {\int_{0}^{\t}}N^{-\frac12}jN^{(\k+1)\e_{\mathrm{step}}}N^{11\e_{\mathrm{ap}}}|\t-\s|^{-\frac12}\d\s\lesssim N^{11\e_{\mathrm{ap}}}N^{-\frac12}jN^{(\k+1)\e_{\mathrm{step}}}\lesssim N^{-\beta},
\end{align*}
where the last bound follows since {\small$jN^{(\k+1)\e_{\mathrm{step}}}\leq n_{\k}N^{(\k+1)\e_{\mathrm{step}}}=\mathfrak{l}_{\mathrm{reg}}=N^{1/3+\e_{\mathrm{reg}}}$} by construction of $n_{\k}$, and because $\e_{\mathrm{ap}},\e_{\mathrm{reg}}$ are small. (We also used the bound {\small$\mathsf{R}_{\k}=\mathrm{O}(N^{10\e_{\mathrm{ap}}})$} from Lemma \ref{lemma:bgprcut}, which is sub-optimal but enough.) By the previous three displays, it now suffices to show that 
\begin{align}
\tfrac{1}{n_{\k}}\sum_{j=0}^{n_{\k}-1}\grad^{\mathbf{X}}_{2jN^{(k+1)\e_{\mathrm{step}}}}\mathbf{Y}^{N}_{\s,\y}&=\mathfrak{z}_{\k}[\tau_{\cdot}\eta_{\s}]\mathbf{Y}^{N}_{\s,\y}\label{eq:bgprx1}
\end{align}
for some function $\mathfrak{z}_{\k}$ satisfying the constraints of Lemma \ref{lemma:bgprx}. If $\s>\mathfrak{t}_{\mathrm{stop}}$, both sides are $0$, and if $\s\leq\mathfrak{t}_{\mathrm{stop}}$, we have {\small$\mathbf{Y}^{N}_{\s,\cdot}=\mathbf{Z}^{N}_{\s,\cdot}$}; see \eqref{eq:tst}. So, we restrict to $\s\leq\mathfrak{t}_{\mathrm{stop}}$, in which case we can use \eqref{eq:gartner} to get
\begin{align*}
\tfrac{1}{n_{\k}}\sum_{j=0}^{n_{\k}-1}\grad^{\mathbf{X}}_{2jN^{(k+1)\e_{\mathrm{step}}}}\mathbf{Y}^{N}_{\s,\y}&=\tfrac{1}{n_{\k}}\sum_{j=0}^{n_{\k}-1}\grad^{\mathbf{X}}_{2jN^{(k+1)\e_{\mathrm{step}}}}\mathbf{Z}^{N}_{\s,\y}\\
&=\left\{\tfrac{1}{n_{\k}}\sum_{j=0}^{n_{\k}-1}\left(\exp\left\{-N^{-\frac12}\sum_{\ell=1,\ldots,2jN^{(k+1)\e_{\mathrm{step}}}}\eta_{\s,\y+\ell}\right\}-1\right)\right\}\mathbf{Z}^{N}_{\s,\y}\\
&=\left\{\tfrac{1}{n_{\k}}\sum_{j=0}^{n_{\k}-1}\left(\exp\left\{-N^{-\frac12}\sum_{\ell=1,\ldots,2jN^{(k+1)\e_{\mathrm{step}}}}\eta_{\s,\y+\ell}\right\}-1\right)\right\}\mathbf{Y}^{N}_{\s,\y}.
\end{align*}
As shown in the proof of Lemma \ref{lemma:canexp}, since $jN^{(\k+1)\e_{\mathrm{reg}}}\leq n_{\k}N^{(\k+1)\e_{\mathrm{reg}}}=\mathfrak{l}_{\mathrm{reg}}$ (see the statement of Lemma \ref{lemma:bgprx}) and since $\s\leq\mathfrak{t}_{\mathrm{stop}}$, the sum over $\ell$ in the previous display has square-root cancellations. In particular, we are allowed to insert into the last line the indicator function
\begin{align*}
\mathbf{1}\left[\bigcap_{j=0,\ldots,n_{\k}-1}\left\{\sum_{\ell=1,\ldots,2jN^{(k+1)\e_{\mathrm{step}}}}\eta_{\s,\y+\ell}=\mathrm{O}(N^{3\e_{\mathrm{ap}}}\mathfrak{l}_{\mathrm{reg}}^{\frac12})\right\}\right]
\end{align*}
We now claim that 
\begin{align*}
\mathfrak{z}_{\k}[\eta]&:=\tfrac{1}{n_{\k}}\sum_{j=0}^{n_{\k}-1}\left(\exp\left\{-N^{-\frac12}\sum_{\ell=1,\ldots,2jN^{(k+1)\e_{\mathrm{step}}}}\eta_{\y+\ell}\right\}-1\right)\\
&\times\mathbf{1}\left[\bigcap_{j=0,\ldots,n_{\k}-1}\left\{\sum_{\ell=1,\ldots,2jN^{(k+1)\e_{\mathrm{step}}}}\eta_{\ell}=\mathrm{O}(N^{3\e_{\mathrm{ap}}}\mathfrak{l}_{\mathrm{reg}}^{\frac12})\right\}\right]
\end{align*}
satisfies both of the constraints from the lemma. The dependence on $\eta_{\x}$ for only $\x\in\{1,\ldots,\mathrm{O}(\mathfrak{l}_{\mathrm{reg}})\}$ holds immediately by inspection and {\small$\mathfrak{l}_{\mathrm{reg}}=n_{\k}N^{(k+1)\e_{\mathrm{step}}}\geq jN^{(k+1)\e_{\mathrm{step}}}$}. The necessary bound is automatic from the indicator and Taylor expansion. Indeed, the indicator function implies that each summand above is of the form {\small$\exp\{\mathrm{O}(N^{-1/2}N^{3\e_{\mathrm{ap}}}\mathfrak{l}_{\mathrm{reg}}^{1/2})\}-1$}; since $\mathfrak{l}_{\mathrm{reg}}=N^{1/3+\e_{\mathrm{reg}}}$ and $\e_{\mathrm{ap}},\e_{\mathrm{reg}}$ are small (see \eqref{eq:tap} and \eqref{eq:tregX}), Taylor expansion shows that {\small$\exp\{\mathrm{O}(N^{-1/2}N^{3\e_{\mathrm{ap}}}\mathfrak{l}_{\mathrm{reg}}^{1/2})\}-1=\mathrm{O}(N^{-1/2}N^{3\e_{\mathrm{ap}}}\mathfrak{l}_{\mathrm{reg}}^{1/2})$}.  This finishes the proof.
\end{proof}
\subsection{Introducing a time-average}
Note that Lemma \ref{lemma:bgprx} lets us introduce a space-average, but at the cost of an error term (given by the second line in \eqref{eq:bgprx}) which needs its own probabilistic analysis. In this step, we will therefore want to replace both the space-average and the $\mathsf{R}_{\k}\mathfrak{z}_{\k}$ in \eqref{eq:bgprx} by time-averages. 
\begin{lemma}\label{lemma:bgprt}
\fsp Fix any integer $0\leq\k<\mathrm{K}_{\mathrm{step}}$. Let $\tau_{\mathrm{av}}=N^{-4/3}$. There exists $\beta>0$ universal such that 
\begin{align}
&{\int_{0}^{\t}}\mathrm{e}^{[\t-\s]\mathscr{T}_{N}}\left\{N^{\frac12}\left(\tfrac{1}{n_{\k}}\sum_{j=0}^{n_{\k}-1}\mathsf{R}_{\k}[\tau_{\cdot-2jN^{(\k+1)\e_{\mathrm{step}}}}\eta_{\s};\overline{\mathfrak{q}}]\right)\mathbf{Y}^{N}_{\s,\cdot}\right\}_{\x}\d\s\label{eq:bgprtI}\\
&={\int_{0}^{\t}}\mathrm{e}^{[\t-\s]\mathscr{T}_{N}}\left\{N^{\frac12}\left(\tfrac{1}{\tau_{\mathrm{av}}}{\textstyle\int_{0}^{\tau_{\mathrm{av}}}}\tfrac{1}{n_{\k}}\sum_{j=0}^{n_{\k}-1}\mathsf{R}_{\k}[\tau_{\cdot-2jN^{(\k+1)\e_{\mathrm{step}}}}\eta_{\s+\r};\overline{\mathfrak{q}}]\d\r\right)\mathbf{Y}^{N}_{\s,\cdot}\right\}_{\x}\d\s+\Theta_{\t,\x},\nonumber
\end{align}
where $\E\|\Theta\|_{\mathrm{L}^{\infty}([0,1]\times\mathbb{T}_{N})}\lesssim N^{-\beta}$. We have the similar estimate
\begin{align}
&{\int_{0}^{\t}}\mathrm{e}^{[\t-\s]\mathscr{T}_{N}}[N^{\frac12}\mathsf{R}_{\k}[\tau_{\cdot}\eta_{\s};\overline{\mathfrak{q}}]\mathfrak{z}_{\k}[\tau_{\cdot}\eta_{\s}]\mathbf{Y}^{N}_{\s,\cdot}]_{\x}\d\s\label{eq:bgprtII}\\
&={\int_{0}^{\t}}\mathrm{e}^{[\t-\s]\mathscr{T}_{N}}[N^{\frac12}\left(\tfrac{1}{\tau_{\mathrm{av}}}{\int_{0}^{\tau_{\mathrm{av}}}}\mathsf{R}_{\k}[\tau_{\cdot}\eta_{\s+\r};\overline{\mathfrak{q}}]\mathfrak{z}_{\k}[\tau_{\cdot}\eta_{\s+\r}]\d\r\right)\mathbf{Y}^{N}_{\s,\cdot}]_{\x}\d\s+\Phi_{\t,\x}.\nonumber
\end{align}
\end{lemma}
\begin{proof}
We first prove \eqref{eq:bgprtI}. For convenience, we use the notation below for the space-average on the LHS:
\begin{align*}
\mathbf{Av}^{\mathsf{R}_{\k}}_{n_{\k}}[\tau_{\cdot}\eta_{\s}]:=\tfrac{1}{n_{\k}}\sum_{j=0}^{n_{\k}-1}\mathsf{R}_{\k}[\tau_{\cdot-2jN^{(\k+1)\e_{\mathrm{step}}}}\eta_{\s};\overline{\mathfrak{q}}].
\end{align*}
(This notation omits the dependence on $\overline{\mathfrak{q}}$ and $\k$; since this notation is localized to this proof, hopefully this omission does not introduce any confusion.) Recall the time-gradients $\grad^{\mathbf{T}}$ from Section \ref{section:notation}. We first compute
\begin{align*}
&{\textstyle\int_{0}^{\t}}\mathrm{e}^{[\t-\s]\mathscr{T}_{N}}[N^{\frac12}(\mathbf{Av}^{\mathsf{R}_{\k}}_{n_{\k}}[\tau_{\cdot}\eta_{\s}]-\tfrac{1}{\tau_{\mathrm{av}}}{\textstyle\int_{0}^{\tau_{\mathrm{av}}}}\mathbf{Av}^{\mathsf{R}_{\k}}_{n_{\k}}[\tau_{\cdot}\eta_{\s+\r}]\d\r)\mathbf{Y}^{N}_{\s,\cdot}]_{\x}\d\s\\
&=\tfrac{1}{\tau_{\mathrm{av}}}{\int_{0}^{\tau_{\mathrm{av}}}}\d\r\int_{0}^{\t}\sum_{\y\in\mathbb{T}_{N}}\mathbf{H}^{N}_{\s,\t,\x,\y}N^{\frac12}\left(\mathbf{Av}^{\mathsf{R}_{\k}}_{n_{\k}}[\tau_{\y}\eta_{\s}]-\mathbf{Av}^{\mathsf{R}_{\k}}_{n_{\k}}[\tau_{\y}\eta_{\s+\r}]\right)\mathbf{Y}^{N}_{\s,\y}\d\s\\
&=-\tfrac{1}{\tau_{\mathrm{av}}}{\int_{0}^{\tau_{\mathrm{av}}}}\d\r\sum_{\y\in\mathbb{T}_{N}}\int_{\R}\mathbf{1}_{\s\in[0,\t]}\mathbf{H}^{N}_{\s,\t,\x,\y}N^{\frac12}\left(\grad^{\mathbf{T}}_{\r}\mathbf{Av}^{\mathsf{R}_{\k}}_{n_{\k}}[\tau_{\y}\eta_{\s}]\right)\mathbf{Y}^{N}_{\s,\y}\d\s.
\end{align*}
The reason for using the indicator of $\s\in[0,\t]$ will be clarified shortly. Elementary change-of-variables computation shows that the adjoint of {\small$\grad^{\mathbf{T}}_{\r}$} with respect to uniform measure on $\R$ is {\small$\grad^{\mathbf{T}}_{-\r}$} (this is essentially integration-by-parts). We also have the readily-verifiable Leibniz rule {\small$\grad^{\mathbf{T}}_{-\r}(\mathbf{1}_{\s\in[0,\t]}\mathbf{H}^{N}_{\s,\t,\x,\y}\mathbf{Y}^{N}_{\s,\y})=\mathbf{1}_{\s\in[0,\t]}\mathbf{H}^{N}_{\s,\t,\x,\y}\grad^{\mathbf{T}}_{-\r}\mathbf{Y}^{N}_{\s,\y}+\mathbf{Y}^{N}_{\s-\r,\y}\grad^{\mathbf{T}}_{-\r}(\mathbf{1}_{\s\in[0,\t]}\mathbf{H}^{N}_{\s,\t,\x,\y})=\mathbf{1}_{\s\in[0,\t]}\mathbf{H}^{N}_{\s,\t,\x,\y}\grad^{\mathbf{T}}_{-\r}\mathbf{Y}^{N}_{\s,\y}+\mathbf{1}_{\s\in[0,\t]}\mathbf{Y}^{N}_{\s-\r,\y}\grad^{\mathbf{T}}_{-\r}\mathbf{H}^{N}_{\s,\t,\x,\y}+\mathbf{H}^{N}_{\s-\r,\t,\x,\y}\mathbf{Y}^{N}_{\s-\r,\y}\grad^{\mathbf{T}}_{-\r}\mathbf{1}_{\s\in[0,\t]}$}; here, $\mathbf{Y}$ evaluated at a negative time means evaluating it at $0$, and time-gradients always act on $\s$ (even in $\mathbf{H}^{N}$). In particular, we have 
\begin{align*}
&\sum_{\y\in\mathbb{T}_{N}}\int_{\R}\mathbf{1}_{\s\in[0,\t]}\mathbf{H}^{N}_{\s,\t,\x,\y}N^{\frac12}\left(\grad^{\mathbf{T}}_{\r}\mathbf{Av}^{\mathsf{R}_{\k}}_{n_{\k}}[\tau_{\y}\eta_{\s}]\right)\mathbf{Y}^{N}_{\s,\y}\d\s\\
&=\sum_{\y\in\mathbb{T}_{N}}\int_{\R}\mathbf{1}_{\s\in[0,\t]}\mathbf{H}^{N}_{\s,\t,\x,\y}N^{\frac12}\mathbf{Av}^{\mathsf{R}_{\k}}_{n_{\k}}[\tau_{\y}\eta_{\s}]\grad^{\mathbf{T}}_{-\r}\mathbf{Y}^{N}_{\s,\y}\d\s\\
&+\sum_{\y\in\mathbb{T}_{N}}\int_{\R}\mathbf{1}_{\s\in[0,\t]}\{\grad^{\mathbf{T}}_{-\r}\mathbf{H}^{N}_{\s,\t,\x,\y}\}N^{\frac12}\mathbf{Av}^{\mathsf{R}_{\k}}_{n_{\k}}[\tau_{\y}\eta_{\s}]\mathbf{Y}^{N}_{\s-\r,\y}\d\s\\
&+\sum_{\y\in\mathbb{T}_{N}}\int_{\R}\{\grad^{\mathbf{T}}_{-\r}\mathbf{1}_{\s\in[0,\t]}\}\mathbf{H}^{N}_{\s-\r,\t,\x,\y}N^{\frac12}\mathbf{Av}^{\mathsf{R}_{\k}}_{n_{\k}}[\tau_{\y}\eta_{\s}]\mathbf{Y}^{N}_{\s-\r,\y}\d\s.
\end{align*}
For the last line, we note that {\small$\grad^{\mathbf{T}}_{-\r}\mathbf{1}_{\s\in[0,\t]}=0$} except for $\s\in[0,\r]\cup[\t,\t+\r]$. If we use this with contractivity of the heat semigroup on $\mathrm{L}^{\infty}(\mathbb{T}_{N})\to\mathrm{L}^{\infty}(\mathbb{T}_{N})$ (it is a Markov semigroup) and {\small$\mathbf{Av}^{\mathsf{R}_{\k}}_{n_{\k}}=\mathrm{O}(N^{10\e_{\mathrm{ap}}})$} (it is an average of terms that satisfy this estimate by Lemma \ref{lemma:bgprcut}) and $\mathbf{Y}^{N}=\mathrm{O}(N^{\e_{\mathrm{ap}}})$ (see \eqref{eq:tst}), then for $\r\in[0,\tau_{\mathrm{av}}]$,
\begin{align*}
\sum_{\y\in\mathbb{T}_{N}}\int_{\R}\{\grad^{\mathbf{T}}_{-\r}\mathbf{1}_{\s\in[0,\t]}\}\mathbf{H}^{N}_{\s-\r,\t,\x,\y}N^{\frac12}\mathbf{Av}^{\mathsf{R}_{\k}}_{n_{\k}}[\tau_{\y}\eta_{\s}]\mathbf{Y}^{N}_{\s-\r,\y}\d\s&\lesssim\int_{[0,\r]\cup[\t,\t+\r]}N^{\frac12+11\e_{\mathrm{ap}}}\d\s\\
&\lesssim N^{\frac12+11\e_{\mathrm{ap}}}\tau_{\mathrm{av}}\lesssim N^{-\beta} 
\end{align*}
with $\beta\gtrsim1$ since $\tau_{\mathrm{av}}=N^{-4/3}$ and $\e_{\mathrm{ap}}>0$ is small. Next, by Proposition \ref{prop:hk}, we know that the $\mathrm{L}^{\infty}(\mathbb{T}_{N})\to\mathrm{L}^{\infty}(\mathbb{T}_{N})$ operator norm of {\small$\grad^{\mathbf{T}}_{-\r}\exp\{[\t-\s]\mathscr{T}_{N}\}$}, where the time-gradient acts on $\s$, is $\lesssim \r^{1-\e}|\t-\s|^{-1+\e}$ for any $\e>0$. If we again use {\small$\mathbf{Av}^{\mathsf{R}_{\k}}_{n_{\k}}=\mathrm{O}(N^{10\e_{\mathrm{ap}}})$} and $\mathbf{Y}^{N}=\mathrm{O}(N^{\e_{\mathrm{ap}}})$, we deduce that for $\r\in[0,\tau_{\mathrm{av}}]$, we have
\begin{align*}
\sum_{\y\in\mathbb{T}_{N}}\int_{\R}\mathbf{1}_{\s\in[0,\t]}\{\grad^{\mathbf{T}}_{-\r}\mathbf{H}^{N}_{\s,\t,\x,\y}\}N^{\frac12}\mathbf{Av}^{\mathsf{R}_{\k}}_{n_{\k}}[\tau_{\y}\eta_{\s}]\mathbf{Y}^{N}_{\s-\r,\y}\d\s\lesssim\int_{0}^{\t}N^{\frac12+11\e_{\mathrm{ap}}}\r^{1-\e}|\t-\s|^{-1+\e}\d\s\lesssim_{\e}N^{-\beta}
\end{align*}
again with $\beta\gtrsim1$, since $\r\leq\tau_{\mathrm{av}}=N^{-4/3}$, and $\e_{\mathrm{ap}},\e$ are small. Now, we again use the discrete Leibniz rule to unfold further {\small$\grad^{\mathbf{T}}_{-\r}\mathbf{Y}^{N}_{\s,\y}=\grad^{\mathbf{T}}_{-\r}(\mathbf{Z}^{N}_{\s,\y}\mathbf{1}_{\s\leq\mathfrak{t}_{\mathrm{stop}}})=\mathbf{1}_{\s\leq\mathfrak{t}_{\mathrm{stop}}}\grad^{\mathbf{T}}_{-\r}\mathbf{Z}^{N}_{\s,\y}+\grad^{\mathbf{T}}_{-\r}\mathbf{1}_{\s\leq\mathfrak{t}_{\mathrm{stop}}}\mathbf{Z}^{N}_{\s-\r,\y}$}. By \eqref{eq:tap}, \eqref{eq:tregT}, and \eqref{eq:tst}, we know {\small$|\mathbf{1}_{\s\leq\mathfrak{t}_{\mathrm{stop}}}\grad^{\mathbf{T}}_{-\r}\mathbf{Z}^{N}_{\s,\y}|\lesssim N^{10\e_{\mathrm{ap}}}\{N^{-2}\vee\r\}^{1/4}$}, for example (the $10$ is unimportant). Also note that {\small$|\grad^{\mathbf{T}}_{-\r}\mathbf{1}_{\s\leq\mathfrak{t}_{\mathrm{stop}}}\mathbf{Z}^{N}_{\s-\r,\y}|\lesssim\mathbf{1}_{\s\in[\mathfrak{t}_{\mathrm{stop}},\mathfrak{t}_{\mathrm{stop}}+\r]}|\mathbf{Z}^{N}_{\s-\r,\y}|\lesssim N^{\e_{\mathrm{ap}}}\mathbf{1}_{\s\in[\mathfrak{t}_{\mathrm{stop}},\mathfrak{t}_{\mathrm{stop}}+\r]}$}; see \eqref{eq:tap}. Putting all this together gives
\begin{align*}
&\sum_{\y\in\mathbb{T}_{N}}\int_{\R}\mathbf{1}_{\s\in[0,\t]}\mathbf{H}^{N}_{\s,\t,\x,\y}N^{\frac12}\mathbf{Av}^{\mathsf{R}_{\k}}_{n_{\k}}[\tau_{\y}\eta_{\s}]\grad^{\mathbf{T}}_{-\r}\mathbf{Y}^{N}_{\s,\y}\d\s\\
&\lesssim N^{10\e_{\mathrm{ap}}}\{N^{-2}\vee\r\}^{\frac14}\int_{0}^{\t}\sum_{\y\in\mathbb{T}_{N}}\mathbf{H}^{N}_{\s,\t,\x,\y}N^{\frac12}|\mathbf{Av}^{\mathsf{R}_{\k}}_{n_{\k}}[\tau_{\y}\eta_{\s}]|\d\s\\
&+ N^{\e_{\mathrm{ap}}}{\textstyle\int_{\mathfrak{t}_{\mathrm{stop}}}^{\mathfrak{t}_{\mathrm{stop}}+\r}}\sum_{\y\in\mathbb{T}_{N}}\mathbf{H}^{N}_{\s,\t,\x,\y}N^{\frac12}|\mathbf{Av}^{\mathsf{R}_{\k}}_{n_{\k}}[\tau_{\y}\eta_{\s}]|\d\s.
\end{align*}
We use contractivity of the heat kernel on $\mathrm{L}^{\infty}(\mathbb{T}_{N})$ and {\small$\mathbf{Av}^{\mathsf{R}_{\k}}_{n_{\k}}=\mathrm{O}(N^{10\e_{\mathrm{ap}}})$} as before to get that the last line is {\small$\lesssim N^{1/2+11\e_{\mathrm{ap}}}\r\lesssim N^{-\beta}$} with $\beta\gtrsim1$ since $\r\leq\tau_{\mathrm{av}}=N^{-4/3}$. Moreover, since we only take $\r\leq \tau_{\mathrm{av}}=N^{-4/3}$, we have $\{N^{-2}\wedge\r\}^{1/4}\lesssim N^{-1/3}$, so the effective factor of $N$ in the second line above is $N^{1/6}N^{10\e_{\mathrm{ap}}}$. Thus, since $\e_{\mathrm{ap}}>0$ is small, if we combine every display in this proof thus far, to prove \eqref{eq:bgprtI}, it suffices to show
\begin{align}
\tfrac{1}{\tau_{\mathrm{av}}}{\int_{0}^{\tau_{\mathrm{av}}}}\E\left[\sup_{\t\in[0,\t]}\sup_{\x\in\mathbb{T}_{N}}\int_{0}^{\t}\sum_{\y\in\mathbb{T}_{N}}\mathbf{H}^{N}_{\s,\t,\x,\y}N^{\frac16}|\mathbf{Av}^{\mathsf{R}_{\k}}_{n_{\k}}[\tau_{\y}\eta_{\s}]|\d\s\right]\lesssim N^{-\beta}\label{eq:bgprtI1}
\end{align}
with $\beta\gtrsim1$ independent of $\e_{\mathrm{ap}}$. We want to use Lemma \ref{lemma:probanspace} with:
\begin{itemize}
\item $n=n_{\k}$ and {\small$\mathfrak{f}_{j}[\eta]:=\mathsf{R}_{\k}[\tau_{-2jN^{(\k+1)\e_{\mathrm{step}}}}\eta;\overline{\mathfrak{q}}]$} for $j=1,\ldots,n_{\k}$.
\end{itemize}
Recall from Lemma \ref{lemma:bgprcut} that $\mathfrak{f}_{j}[\eta]$  vanishes with respect to any canonical measure expectation and depends only on $\eta_{\x}$ for {\small$\x\in\mathbb{I}_{j}=\{-N^{(\k+1)\e_{\mathrm{step}}},\ldots,0\}-2jN^{(\k+1)\e_{\mathrm{step}}}$}. These are disjoint sub-intervals, and their union is contained in $\{-Cn_{\k}N^{(\k+1)\e_{\mathrm{step}}},\ldots,0\}$ for some $C=\mathrm{O}(1)$. So, in our application of Lemma \ref{lemma:proban}, we can take $\mathbb{L}$ therein to satisfy {\small$|\mathbb{L}|\asymp n_{\k}N^{(k+1)\e_{\mathrm{ap}}}=\mathfrak{l}_{\mathrm{reg}}=N^{1/3+\e_{\mathrm{reg}}}$} (see Lemma \ref{lemma:bgprx}), where $\asymp$ means $\gtrsim$ and $\lesssim$ with different implied constants. Thus, the LHS of \eqref{eq:bgprtI1} is $\lesssim$ the $2/3$-th power of 
\begin{align*}
&N^{-\frac74}|\mathbb{L}|^{3}n_{\k}^{-\frac34}\sup_{i=1,\ldots,n}\|\mathfrak{f}_{i}\|_{\infty}^{\frac32}+N^{\frac14+\e}n_{\k}^{-\frac34}\sup_{i=1,\ldots,n}\|\mathfrak{f}_{i}\|_{\infty}^{\frac32}\\
&\lesssim N^{-\frac74}N^{1+3\e_{\mathrm{reg}}}n_{\k}^{-\frac34}\|\mathsf{R}_{k}[\cdot;\overline{\mathfrak{q}}]\|_{\infty}^{\frac32}+N^{\frac14+\e}n_{\k}^{-\frac34}\|\mathsf{R}_{k}[\cdot;\overline{\mathfrak{q}}]\|_{\infty}^{\frac32}.
\end{align*}
We now use Lemma \ref{lemma:bgprcut} to get {\small$n_{\k}^{-3/4}\|\mathsf{R}_{\k}[\cdot;\overline{\mathfrak{q}}]\|_{\infty}^{3/2}\lesssim N^{10\e_{\mathrm{ap}}}n_{\k}^{-3/4}N^{-9\k\e_{\mathrm{step}}/4}$}. We again use $n_{\k}N^{(\k+1)\e_{\mathrm{step}}}=\mathfrak{l}_{\mathrm{reg}}=N^{1/3+\e_{\mathrm{reg}}}$ and extend it to {\small$n_{\k}^{-3/4}\|\mathsf{R}_{\k}[\cdot;\overline{\mathfrak{q}}]\|_{\infty}^{3/2}\lesssim N^{10\e_{\mathrm{ap}}}n_{\k}^{-3/4}N^{-9\k\e_{\mathrm{step}}/4}=N^{10\e_{\mathrm{ap}}}(n_{\k}^{-1}N^{-3\k\e_{\mathrm{step}}})^{3/4}\lesssim N^{10\e_{\mathrm{ap}}+\e_{\mathrm{step}}}N^{-1/4-3\e_{\mathrm{reg}}/4}$}. Hence, the last line in the previous display is 
\begin{align*}
&\lesssim N^{-\frac74}N^{1+3\e_{\mathrm{reg}}}N^{-\frac14}+N^{\frac14+\e}N^{10\e_{\mathrm{ap}}+\e_{\mathrm{step}}}N^{-\frac14-\frac34\e_{\mathrm{reg}}}\lesssim N^{-\beta}
\end{align*}
with $\beta\gtrsim1$ since $\e_{\mathrm{ap}},\e_{\mathrm{step}}>0$ are small constants times $\e_{\mathrm{reg}}>0$, which itself is also small. The previous two displays give \eqref{eq:bgprtI1}, so the proof of \eqref{eq:bgprtI} is complete. We now prove \eqref{eq:bgprtII}. We claim that it suffices to prove the following analog of \eqref{eq:bgprtI1}, where $\beta\gtrsim1$ is independent of $\e_{\mathrm{ap}}$:
\begin{align}
\tfrac{1}{\tau_{\mathrm{av}}}{\int_{0}^{\tau_{\mathrm{av}}}}\E\left[\sup_{\t\in[0,\t]}\sup_{\x\in\mathbb{T}_{N}}\int_{0}^{\t}\sum_{\y\in\mathbb{T}_{N}}\mathbf{H}^{N}_{\s,\t,\x,\y}N^{\frac16}|\mathsf{R}_{\k}[\tau_{\cdot}\eta_{\s};\overline{\mathfrak{q}}]\mathfrak{z}_{\k}[\tau_{\cdot}\eta_{\s}]|\d\s\right]\lesssim N^{-\beta}.\label{eq:bgprtII1}
\end{align}
Indeed, the sufficiency of this estimate uses only the crude bound {\small$\mathbf{Av}^{\mathsf{R}_{\k}}_{n_{\k}}[\tau_{\cdot}\eta]=\mathrm{O}(N^{10\e_{\mathrm{ap}}})$}, and we also have {\small$\mathsf{R}_{\k}[\eta;\overline{\mathfrak{q}}]\mathfrak{z}_{\k}[\eta]=\mathrm{O}(N^{10\e_{\mathrm{ap}}})$} as well by Lemmas \ref{lemma:bgprcut} and \ref{lemma:bgprx}. (We actually have much more from $\mathfrak{z}_{\k}$ from Lemma \ref{lemma:bgprx}; we use this shortly.) To prove \eqref{eq:bgprtII1}, we can proceed with a direct estimate. By Lemmas \ref{lemma:bgprcut} and \ref{lemma:bgprx}, we know {\small$\mathsf{R}_{\k}[\eta;\overline{\mathfrak{q}}]=\mathrm{O}(N^{10\e_{\mathrm{ap}}})$} and {\small$\mathfrak{z}_{\k}[\eta]=\mathrm{O}(N^{3\e_{\mathrm{ap}}+\e_{\mathrm{reg}}}N^{-1/3})$}. Thus, {\small$\mathsf{R}_{\k}[\tau_{\y}\eta_{\s};\overline{\mathfrak{q}}]\mathfrak{z}_{\k}[\tau_{\y}\eta_{\s}]=\mathrm{O}(N^{C\e_{\mathrm{ap}}+\e_{\mathrm{reg}}}N^{-1/3})$}. Plug this estimate into the LHS of \eqref{eq:bgprtII1} and use contractivity of the heat kernel (as the kernel for an operator {\small$\mathrm{L}^{\infty}(\mathbb{T}_{N})\to\mathrm{L}^{\infty}(\mathbb{T}_{N})$}) to get \eqref{eq:bgprtII1}. (No expectation is even needed in \eqref{eq:bgprtII1}.)
\end{proof}
\subsection{Final estimates}
We have introduced space and time averages. Now, we use Lemmas \ref{lemma:proban} and \ref{lemma:probanfluc}.
\begin{lemma}\label{lemma:bgprfinal}
\fsp Fix any integer $0\leq\k<\mathrm{K}_{\mathrm{step}}$. There exists $\beta>0$ universal such that 
\begin{align}
&\E\left[\sup_{\t\in[0,1]}\sup_{\x\in\mathbb{T}_{N}}\left|{\int_{0}^{\t}}\mathrm{e}^{[\t-\s]\mathscr{T}_{N}}\left\{N^{\frac12}\left(\tfrac{1}{\tau_{\mathrm{av}}}{\textstyle\int_{0}^{\tau_{\mathrm{av}}}}\tfrac{1}{n_{\k}}\sum_{j=0}^{n_{\k}-1}\mathsf{R}_{\k}[\tau_{\cdot-2jN^{(\k+1)\e_{\mathrm{step}}}}\eta_{\s+\r};\overline{\mathfrak{q}}]\d\r\right)\mathbf{Y}^{N}_{\s,\cdot}\right\}_{\x}\d\s\right|\right]\label{eq:bgprfinal}\\
&+\E\left[\sup_{\t\in[0,1]}\sup_{\x\in\mathbb{T}_{N}}\left|{\int_{0}^{\t}}\mathrm{e}^{[\t-\s]\mathscr{T}_{N}}[N^{\frac12}\left(\tfrac{1}{\tau_{\mathrm{av}}}{\int_{0}^{\tau_{\mathrm{av}}}}\mathsf{R}_{\k}[\tau_{\cdot}\eta_{\s+\r};\overline{\mathfrak{q}}]\mathfrak{z}_{\k}[\tau_{\cdot}\eta_{\s+\r}]\d\r\right)\mathbf{Y}^{N}_{\s,\cdot}]_{\x}\d\s\right|\right]\lesssim N^{-\beta}.\nonumber
\end{align}
\end{lemma}
\begin{proof}
Take the first term on the LHS of \eqref{eq:bgprfinal}. With explanation given shortly, we use Lemma \ref{lemma:proban} with:
\begin{itemize}
\item $n=n_{\k}$ and $\tau=\tau_{\mathrm{av}}=N^{-4/3}$ and {\small$\mathfrak{f}_{j}[\eta]:=\mathsf{R}_{\k}[\tau_{-2jN^{(\k+1)\e_{\mathrm{step}}}}\eta;\overline{\mathfrak{q}}]$} for $j=1,\ldots,n_{k}$.  
\end{itemize}
Recall from Lemma \ref{lemma:bgprcut} that $\mathfrak{f}_{j}[\eta]$  vanishes with respect to any canonical measure expectation and depends only on $\eta_{\x}$ for {\small$\x\in\mathbb{I}_{j}=\{-N^{(\k+1)\e_{\mathrm{step}}},\ldots,0\}-2jN^{(\k+1)\e_{\mathrm{step}}}$}. These are disjoint sub-intervals, and their union over $j$ is contained in $\{-Cn_{\k}N^{(\k+1)\e_{\mathrm{step}}},\ldots,0\}$ for some $C=\mathrm{O}(1)$. Thus, in our application of Lemma \ref{lemma:proban}, we can take $\mathbb{L}$ therein to satisfy {\small$|\mathbb{L}|\asymp n_{\k}N^{(k+1)\e_{\mathrm{ap}}}=\mathfrak{l}_{\mathrm{reg}}=N^{1/3+\e_{\mathrm{reg}}}$} (see Lemma \ref{lemma:bgprx}), where $\asymp$ means $\gtrsim$ and $\lesssim$. One can readily check $\tau_{\mathrm{av}}|\mathbb{L}|\lesssim N^{-4/3}\mathfrak{l}_{\mathrm{reg}}\lesssim1$ and $\tau_{\mathrm{av}}=N^{-4/3}\lesssim N^{-1}$ and $N^{2}\tau_{\mathrm{av}}=N^{2}N^{-4/3}=N^{2/3}\lesssim N^{2/3+2\e_{\mathrm{reg}}}\lesssim|\mathbb{L}|^{2}$, so that Lemma \ref{lemma:proban} actually applies. In particular, we deduce that for any $\delta>0$ small, the first term on the LHS of \eqref{eq:bgprfinal} is $\lesssim_{\delta}$ the $2/3$-th power of
\begin{align*}
&N^{\frac32\e_{\mathrm{ap}}}N^{\delta}N^{-\frac54}|\mathbb{L}|^{3}n_{\k}^{-\frac34}\|\mathsf{R}_{k}[\cdot;\overline{\mathfrak{q}}]\|_{\infty}^{\frac32}+N^{\frac32\e_{\mathrm{ap}}}N^{\delta}N^{-\frac34}\tau_{\mathrm{av}}^{-\frac34}n_{\k}^{-\frac34}\sup_{j=1,\ldots,n_{\k}}|\mathbb{I}_{j}|^{\frac32}\|\mathsf{R}_{k}[\cdot;\overline{\mathfrak{q}}]\|_{\infty}^{\frac32}.
\end{align*}
Now, we plug in $|\mathbb{L}|\lesssim N^{1/3+\e_{\mathrm{reg}}}$ and $\tau_{\mathrm{av}}=N^{-4/3}$ and {\small$|\mathbb{I}_{j}|\lesssim N^{(\k+1)\e_{\mathrm{step}}}$} and {\small$n_{\k}N^{(\k+1)\e_{\mathrm{step}}}=\mathfrak{l}_{\mathrm{reg}}=N^{1/3+\e_{\mathrm{reg}}}$} and $|\mathsf{R}_{\k}[\cdot;\overline{\mathfrak{q}}]|\lesssim N^{10\e_{\mathrm{ap}}}N^{-3\k\e_{\mathrm{step}}/2}$ from Lemma \ref{lemma:bgprcut}. This shows that the previous display is 
\begin{align*}
&\lesssim N^{C\e_{\mathrm{ap}}}N^{\delta}\left\{N^{-\frac54}N^{1+3\e_{\mathrm{reg}}}n_{\k}^{-\frac34}N^{-\frac94\k\e_{\mathrm{step}}}+N^{-\frac34}Nn_{\k}^{-\frac34}N^{\frac{3(\k+1)}{2}\e_{\mathrm{step}}}N^{-\frac94\k\e_{\mathrm{step}}}\right\}\\
&\lesssim N^{C\e_{\mathrm{ap}}+C\e_{\mathrm{step}}}N^{\delta}\left\{N^{-\frac54}N^{1+3\e_{\mathrm{reg}}}\mathfrak{l}_{\mathrm{reg}}^{-\frac34}+N^{\frac14}\mathfrak{l}_{\mathrm{reg}}^{-\frac34}\right\}\lesssim N^{C\e_{\mathrm{ap}}+C\e_{\mathrm{step}}}N^{\delta}\left\{N^{-\frac14+C\e_{\mathrm{reg}}}+N^{-\frac34\e_{\mathrm{reg}}}\right\},
\end{align*}
where $C=\mathrm{O}(1)$. Since $\e_{\mathrm{ap}},\e_{\mathrm{step}},\delta$ are small multiplies of $\e_{\mathrm{reg}}$, and because $\e_{\mathrm{reg}}\gtrsim1$ is also small, the previous display is $\lesssim N^{-\beta}$ for $\beta\gtrsim1$. It remains to show that the second term on the LHS of \eqref{eq:bgprfinal} is $\lesssim N^{-\beta}$. For this, we now use Lemma \ref{lemma:probanfluc} with the following choices (with explanation given shortly):
\begin{itemize}
\item $\mathfrak{f}_{1}[\eta]=\mathsf{R}_{\k}[\eta;\overline{\mathfrak{q}}]$ and $\mathfrak{f}_{2}[\eta]=\mathfrak{z}_{\k}[\eta]$ (see the second term on the LHS of \eqref{eq:bgprfinal}).
\end{itemize}
Note that $\mathfrak{f}_{1}[\eta],\mathfrak{f}_{2}[\eta]$ depend only on $\eta_{\x}$ for $\x\in\mathbb{L}$ with the same choice of $\mathbb{L}$ from our application of Lemma \ref{lemma:proban}. Moreover, by Lemmas \ref{lemma:bgprcut} and \ref{lemma:bgprx}, $\mathfrak{f}_{1}[\eta]$ depends only on $\eta_{\x}$ for $\x\leq0$ and $\mathfrak{f}_{2}[\eta]$ depends only on $\eta_{\x}$ for $\x\geq1$. More precisely, $\mathfrak{f}_{1}[\eta]$ depends only on $\eta_{\x}$ for a $\x$ in a sub-interval $\mathbb{I}_{1}$ of length $\mathrm{O}(N^{(\k+1)\e_{\mathrm{step}}})$ as noted in our application of Lemma \ref{lemma:proban} above. Thus, since we again take $\tau=\tau_{\mathrm{av}}$, all the constraints for Lemma \ref{lemma:probanfluc} are satisfied (since they were satisfied in our application of Lemma \ref{lemma:proban}, and none of the parameters $|\mathbb{L}|$ or $\tau$ have changed). This impliies that the second term on the LHS of \eqref{eq:bgprfinal} is $\lesssim$ the $2/3$-th power of
\begin{align*}
&\lesssim_{\delta} N^{\frac32\e_{\mathrm{ap}}}N^{\delta}\left\{N^{-\frac54}|\mathbb{L}|^{3}\prod_{i=1,2}\|\mathfrak{f}_{i}\|_{\infty}^{\frac32}+N^{-\frac34}\tau_{\mathrm{av}}^{-\frac34}|\mathbb{I}_{1}|^{\frac32}\prod_{i=1,2}\|\mathfrak{f}_{i}\|_{\infty}^{\frac32}\right\}.
\end{align*}
Now use {\small$\tau_{\mathrm{av}}=N^{-4/3}$} and {\small$|\mathbb{I}_{1}|\lesssim N^{(\k+1)\e_{\mathrm{step}}}$} and {\small$|\mathbb{L}|\lesssim \mathfrak{l}_{\mathrm{reg}}=N^{1/3+\e_{\mathrm{reg}}}$} and {\small$\|\mathfrak{f}_{1}\|_{\infty}\lesssim N^{10\e_{\mathrm{ap}}}N^{-3\k\e_{\mathrm{step}}/2}$} (see Lemma \ref{lemma:bgprcut}) and {\small$\|\mathfrak{f}_{2}\|_{\infty}\lesssim N^{-1/3+\e_{\mathrm{reg}}/2}$} (see Lemma \ref{lemma:bgprx}). This shows that the previous display is 
\begin{align*}
&\lesssim N^{C\e_{\mathrm{ap}}}N^{\delta}\left\{N^{-\frac54}N^{1+3\e_{\mathrm{reg}}}N^{-\frac94\k\e_{\mathrm{step}}}N^{-\frac12+\frac34\e_{\mathrm{reg}}}+N^{-\frac34}NN^{\frac32(\k+1)\e_{\mathrm{step}}}N^{-\frac94\k\e_{\mathrm{step}}}N^{-\frac12+\frac34\e_{\mathrm{reg}}}\right\}\\
&\lesssim N^{C\e_{\mathrm{ap}}+C\e_{\mathrm{reg}}+C\e_{\mathrm{step}}}N^{\delta}\left\{N^{-\frac54+1-\frac12}+N^{-\frac34+1-\frac12}\right\}\lesssim N^{-\beta}
\end{align*}
for $C=\mathrm{O}(1)$ and $\beta\gtrsim1$, since $\e_{\mathrm{ap}},\e_{\mathrm{reg}},\e_{\mathrm{step}},\delta>0$ are small. This controls the second term on the LHS of \eqref{eq:bgprfinal} by $\lesssim N^{-\beta}$ with $\beta\gtrsim1$, so this completes the proof.
\end{proof}
\subsection{Proof of Proposition \ref{prop:hl}}
The same argument (which we just explained in detail for the proof of Theorem \ref{theorem:bgp}) works. Indeed, the only fact we used about $\overline{\mathfrak{q}}$ (versus $\mathfrak{s}$ or $\mathfrak{g}$) is {\small$|\mathsf{E}^{\mathrm{can},\mathfrak{l}}[\tau_{\y}\eta_{\s};\overline{\mathfrak{q}}]|\lesssim N^{3\e_{\mathrm{ap}}}\mathfrak{l}^{-3/2}$} for $\mathfrak{l}\lesssim N^{1/3+\e_{\mathrm{reg}}}$ (see Lemma \ref{lemma:canexp}). However, since $\E^{0}\mathfrak{s}=0$, we have {\small$|\mathsf{E}^{\mathrm{can},\mathfrak{l}}[\tau_{\y}\eta_{\s};\mathfrak{s}]|\lesssim N^{3\e_{\mathrm{ap}}}\mathfrak{l}^{-1/2}$} by Lemma \ref{lemma:canexp}. This is worse by $\mathfrak{l}^{-1}$, but $\mathfrak{s}$ does not come with a factor of $N^{1/2}$, unlike $\overline{\mathfrak{q}}$ (see \eqref{eq:upsilonbg} and \eqref{eq:upsilonhl}). So, we gain $N^{-1/2}$. The total factor we must now include is $N^{-1/2}\mathfrak{l}$. Since $\mathfrak{l}\lesssim N^{1/3+\e_{\mathrm{reg}}}$, this is small. This also holds for $\mathfrak{g}$ instead of $\mathfrak{s}$. \qed
\appendix
\section{General estimates}
\subsection{Heat kernel estimates}
We recall $\mathbf{H}^{N}$ and its $\mathscr{T}_{N}$-semigroup from Definition \ref{definition:heat}. The following result is standard and can be found as Proposition A.3 in \cite{Y23}.
\begin{prop}\label{prop:hk}
\fsp 
Fix $0\leq\s\leq\t\lesssim1$ and $\x,\y\in\mathbb{T}_{N}$. Fix any $\mathfrak{l}\in\Z$ and $\mathfrak{t}\geq0$ and $\upsilon\in(0,1)$. We have
\begin{align}
0\leq\mathbf{H}^{N}_{\s,\t,\x,\y}&\lesssim N^{-1}|\t-\s|^{-\frac12}\\
|\mathbf{H}^{N}_{\s,\t,\x+\mathfrak{l},\y}-\mathbf{H}^{N}_{\s,\t,\x,\y}|+|\mathbf{H}^{N}_{\s,\t,\x,\y+\mathfrak{l}}-\mathbf{H}^{N}_{\s,\t,\x,\y}|&\lesssim N^{-1-\upsilon}|\t-\s|^{-\frac12-\frac12\upsilon}|\mathfrak{l}|^{\upsilon}\\
|\mathbf{H}^{N}_{\s,\t+\mathfrak{t},\x,\y}-\mathbf{H}^{N}_{\s,\t,\x,\y}|&\lesssim N^{-1}|\t-\s|^{-\frac12-\upsilon}|\mathfrak{t}|^{\upsilon}.
\end{align}
We also have the averaged regularity estimates as well:
\begin{align}
\|\mathbf{H}^{N}_{\s,\t,\x,\cdot}\|_{\mathrm{L}^{1}(\mathbb{T}_{N})}&\lesssim1,\\
\|\mathbf{H}^{N}_{\s,\t,\x+\mathfrak{l},\cdot}-\mathbf{H}^{N}_{\s,\t,\x,\cdot}\|_{\mathrm{L}^{1}(\mathbb{T}_{N})}+\|\mathbf{H}^{N}_{\s,\t,\x,\cdot+\mathfrak{l}}-\mathbf{H}^{N}_{\s,\t,\x,\cdot}\|_{\mathrm{L}^{1}(\mathbb{T}_{N})}&\lesssim N^{-\upsilon}|\t-\s|^{-\frac12\upsilon}|\mathfrak{l}|^{\upsilon}\\
\|\mathbf{H}^{N}_{\s,\t+\mathfrak{t},\x,\cdot}-\mathbf{H}^{N}_{\s,\t,\x,\cdot}\|_{\mathrm{L}^{1}(\mathbb{T}_{N})}&\lesssim |\t-\s|^{-\upsilon}|\mathfrak{t}|^{\upsilon}.
\end{align}
Next, we have the operator norm estimates below:
\begin{align}
\|\exp\{[\t-\s]\mathscr{T}_{N}\}\|_{\mathrm{L}^{\infty}(\mathbb{T}_{N})\to\mathrm{L}^{\infty}(\mathbb{T}_{N})}&\lesssim1\\
\|\grad^{\mathbf{X}}_{\mathfrak{l}}\exp\{[\t-\s]\mathscr{T}_{N}\}\|_{\mathrm{L}^{\infty}(\mathbb{T}_{N})\to\mathrm{L}^{\infty}(\mathbb{T}_{N})}&\lesssim N^{-\upsilon}|\t-\s|^{-\frac12\upsilon}|\mathfrak{l}|^{\upsilon},\\
\|\grad^{\mathbf{T}}_{\mathfrak{t}}\exp\{[\t-\s]\mathscr{T}_{N}\}\|_{\mathrm{L}^{\infty}(\mathbb{T}_{N})\to\mathrm{L}^{\infty}(\mathbb{T}_{N})}&\lesssim |\t-\s|^{-\upsilon}|\mathfrak{t}|^{\upsilon}.
\end{align}
\end{prop}
\subsection{General entropy inequality}
For a proof of the results below, see Appendix 1.8 in \cite{KL}.
\begin{lemma}\label{lemma:entropyinequality}
\fsp Suppose that $\mu$ is a probability measure on a probability space $\Omega$, and $\mathfrak{p}$ is a probability density with respect to $\mu$. For any function $\mathfrak{f}:\Omega\to\R$ and any $\kappa>0$, we have 
\begin{align*}
\int_{\Omega}\mathfrak{f}(\omega)\mathfrak{p}(\omega)\mu(\d\omega)&\leq\kappa^{-1}\int_{\Omega}\mathfrak{p}(\omega)\log\mathfrak{p}(\omega)\mu(\d\omega)+\kappa^{-1}\log\int_{\Omega}\mathrm{e}^{\kappa\mathfrak{f}(\omega)}\mu(\d\omega).
\end{align*}
\end{lemma}
%
%
%

\end{document}